\documentclass[12pt]{article}
\usepackage{amsmath,amsfonts,amssymb}
\usepackage{alltt}
\usepackage{hyperref}

\usepackage{mathscinet} 
\usepackage{rotating}

\usepackage{mathtools}
\usepackage{amsmath,amsfonts,ifthen,fullpage,enumerate}  
\usepackage{mathrsfs}
\usepackage{mathtools}

\usepackage{enumerate}

\usepackage{amsmath,amsfonts,amssymb}
\numberwithin{equation}{subsection}

\usepackage{graphicx}
\usepackage{makecell}  
\usepackage{multirow}
\usepackage[english]{babel}

\usepackage{caption}
\usepackage{graphicx}        

\usepackage[table]{xcolor}   

\DeclareRobustCommand{\svdots}{
  \vbox{%
    \baselineskip=0.33333\normalbaselineskip
    \lineskiplimit=0pt
    \hbox{.}\hbox{.}\hbox{.}%
    \kern-0.1\baselineskip
  }%
}

\def\spam{\mathop{\rm span}\nolimits}

\def\grey{\cellcolor{gray!25}}

\def\Re{\mathop{\rm Re}\nolimits}
\def\Im{\mathop{\rm Im}\nolimits}

\def\Harm{\mathop{\rm Harm}\nolimits}

\def\Sym{\mathop{\rm Sym}\nolimits}

\def\bmat#1{\begin{bmatrix}#1\end{bmatrix}}
\def\pmat#1{\begin{pmatrix}#1\end{pmatrix}}
\def\mat#1{\begin{matrix}#1\end{matrix}}
\def\smallmat#1{\begin{smallmatrix}#1\end{smallmatrix}}

\def\question#1{{\bf Question: }#1}
\def\question#1{}
\def\identifywith{\quad\longleftrightarrow\quad}
\def\Sp{\mathop{\rm Sp}\nolimits}

\def\moveR{\xrightarrow{R}}
\def\moveRb{\xleftarrow{R^*\hskip-0.2truecm}}
\def\moveL{{\scriptstyle L} \hskip-0.1truecm \downarrow}
\def\moveLb{{\scriptstyle L^*} \hskip-0.1truecm \uparrow}

\long\def\proclaim#1#2\endproclaim{\medbreak
  \noindent{\bf#1.\enspace}{\it\ignorespaces#2\par}%
  \ifdim\lastskip<\medskipamount \removelastskip\penalty55\medskip\fi}

\def\cZ{{\cal Z}}

\def\R{\mathbb{R}}

\def\CC{\mathbb{C}}

\def\FF{\mathbb{F}}
\def\HH{\mathbb{H}}
\def\KK{\mathbb{K}}
\def\LL{\mathbb{L}}

\def\ZZ{\mathbb{Z}}

\def\Cd{{\CC^d}}
\def\Fd{\FF^d}

\def\Hd{{\HH^d}}

\def\Rd{\R^d}

\def\C{\mathbb{C}}

\def\SS{\mathbb{S}}

\newcommand{\RR}{\mathbb{R}}

\newtheorem{theorem}{Theorem}[section]
\newtheorem{ansatz}[theorem]{Ansatz}
\newtheorem{schematic}[theorem]{Schematic}
\newtheorem{corollary}[theorem]{Corollary}
\newtheorem{lemma}[theorem]{Lemma}
\newtheorem{example}[theorem]{Example}

\newtheorem{proposition}[theorem]{Proposition}

\newenvironment{proof}{{\noindent \it Proof.}}{\hfill$\Box$\medskip}

\newif\ifdraft\def\draft{\drafttrue\hoffset=.8truecm\showlabeltrue
\def\comment##1{{\bf comment: ##1}}
\headline={\sevenrm \hfill \ifx\filenamed\undefined\jobname\else\filenamed\fi%
(.tex) (as of \ifx\updated\undefined???\else\updated\fi)
 \TeX'ed at {\hour\time\divide\hour by 60{}%
\minutes\hour\multiply\minutes by 60{}%
\advance\time by -\minutes
\the\hour:\ifnum\time<10{}0\fi\the\time\  on \today\hfill}}
}

\def\inpro#1{\langle#1\rangle}
\def\ip#1{\langle\kern-.28em\langle#1\rangle\kern-.28em\rangle_\nu}

\def\cH{{\cal H}}

\def\norm#1{\Vert#1\Vert}
\def\openR{{{\rm I}\kern-.16em {\rm R}}}

\def\Fd{\FF^d}
\let\ga\alpha

\let\gb\beta

\let\gG\Gamma
\let\gd\delta
\let\gD\Delta
\let\gep\varepsilon

\let\gs\sigma

\let\ga\alpha

\let\gG\Gamma
\let\gb\beta
\let\gd\delta
\let\gs\sigma
\def\inpro#1{\langle#1\rangle}
\def\dinpro#1{\langle\hskip-0.25em\langle#1\rangle\hskip-0.25em\rangle}
\def\dinpro#1{\langle\hskip-0.35em\langle#1\rangle\hskip-0.35em\rangle}
\def\dinpro#1{\langle\hskip-0.35em\langle\hskip-0.35em\langle#1\rangle\hskip-0.35em\rangle\hskip-0.35em\rangle}
\def\dinpro#1{\langle#1\rangle_\partial}

\def\Hom{\mathop{\rm Hom}\nolimits}

\def\Im{\mathop{\rm Im}\nolimits}

\def\ker{\mathop{\rm ker}\nolimits}

\def\Iff{\hskip1em\Longleftrightarrow\hskip1em}
\def\Implies{\hskip1em\Longrightarrow\hskip1em}

\def\formeq{\the\sectionno.\the\equationno}  
\def\elabel#1/#2/#3/{\global\advance\equationno by 1 %
\ifx#1\empty\else\emember#1%
\ifshowlabel\marginal{\string#1}\fi\fi%
\ifmmode\eqno{#3(\formeq#2)}\else#3\formeq#2\fi} 

\def\makeblanksquare#1#2{
\dimen0=#1pt\advance\dimen0 by -#2pt
      \vrule height#1pt width#2pt depth0pt\kern-#2pt
      \vrule height#1pt width#1pt depth-\dimen0 \kern-#1pt
      \vrule height#2pt width#1pt depth0pt \kern-#2pt
      \vrule height#1pt width#2pt depth0pt
}

\title{\bf
An explicit construction of the unitarily invariant quaternionic polynomial spaces
on the sphere
}

\author{Shayne Waldron and Mozhgan Mohammadpour
\\ }

\begin{document}

\maketitle

\begin{abstract}

The decomposition of the polynomials on the quaternionic unit sphere in $\Hd$ 
into irreducible modules under the action of the quaternionic unitary (symplectic) 
group and quaternionic scalar multiplication has been studied by several 
authors. Typically, these abstract decompositions into 
``quaternionic spherical harmonics'' specify 
the irreducible representations involved and their multiplicities.

The elementary constructive approach taken here 
gives an orthogonal direct sum of irreducibles, 
which can be described by some low-dimensional subspaces,
to which commuting linear operators $L$ and $R$ are applied.
These operators map harmonic polynomials to harmonic polynomials,
and zonal polynomials to zonal polynomials. We give explicit formulas for the
relevant ``zonal polynomials'' and describe the symmetries, dimensions, and ``complexity''
of the subspaces involved.

Possible applications include the construction and analysis of desirable 
sets of points in quaternionic space, such as equiangular lines, lattices and
spherical designs (cubature rules).


\end{abstract}

\bigskip
\vfill

\noindent {\bf Key Words:}
irreducible representations of the quaternionic unitary group,
symplectic group,
quaternionic polynomials,
spherical harmonic polynomials,
zonal functions,
projective spherical $t$-designs,
finite tight frames,
quaternionic equiangular lines,

\bigskip
\noindent {\bf AMS (MOS) Subject Classifications:}
primary
15B33, \ifdraft (Matrices over special rings (quaternions, finite fields, etc.) \else\fi
20G20, \ifdraft (Linear algebraic groups over the reals, the complexes, the quaternions) \else\fi
33C55, \ifdraft (Spherical harmonics) \else\fi
42C15, \ifdraft General harmonic expansions, frames  \else\fi
\quad
secondary
17B10, \ifdraft (Representations of Lie algebras and Lie superalgebras, algebraic theory (weights)) \else\fi
42-08, \ifdraft (Computational methods for problems pertaining to harmonic analysis on Euclidean spaces) \else\fi
46S05, \ifdraft (Quaternionic functional analysis) \else\fi
51M20, \ifdraft (Polyhedra and polytopes; regular figures, division of spaces [See also 51F15]) \else\fi

\vskip .5 truecm
\hrule
\newpage

\section{Introduction}

There are several desirable sets of points that have been, and are, 
studied in real, complex and quaternionic space, which includes
equiangular lines \cite{ACFW18}, \cite{W18},
spherical designs (cubature rules) \cite{H84} and lattices \cite{BN02}.
These are usually classified up to ``unitary equivalence'', and are often
constructed as a group orbit of a ``unitary action''. 
These considerations have led to this paper.

We consider the invariant subspaces of harmonic polynomials
on quaternionic space $\Hd$ under the natural action of
the quaternionic
unitary matrices (the symplectic group) and scalar multiplication
by quaternions (acting on the other side),
and the associated zonal polynomials and reproducing
kernels. This question has been considered several times, independently,
e.g., \cite{S74}, \cite{BN02}, \cite{Ghent14}, \cite{BSW17}, \cite{ACMM20}.
The exact answer given depends on the precise definition of the
harmonic polynomials, in particular, the field in which they may
take values, and the precise group and its action.
The devil is in the details.

Here we give an elementary examples driven development of this question,
motivated by the more well known real and complex cases, the 
only partly trivial case of $\HH^1$, and our
interest in the construction of spherical designs for the quaternionic
sphere \cite{W20}, \cite{W20b}.
This proceeds from certain unambiguous definitions and well known facts
(the details). We hope that this illuminates the above literature as it applies,
and our results can be used for practical computations.
Key aspects of our development include:
\begin{itemize}
\item By considering the action of scalar multiplication by quaternions on 
polynomials $\Hd\to\CC$,
we are naturally led to the operators $L$ and $R$. The operator $R$ appears
in \cite{Ghent18} as $\gep^\dagger$, and implicitly in the development
of the irreducible representations of the multiplicative group $\HH^*$
given in \cite{F95}.
\item There is a natural correspondence between results for homogeneous
polynomials and for harmonic polynomials (given by the Fisher decomposition).
Ultimately, we are primarily interested in irreducible representations of 
harmonic polynomials. Sometimes we start with the homogeneous polynomials,
as these have natural inner products and explicit bases (of monomials).
\item There are many technical calculations 
that could clutter the presentation. Some of these are proved later, and
we often give explicit examples, e.g., the operator $R$ in 
one dimension, or a zonal polynomial with pole $e_1=(1,0,\ldots0)$,
to convey the basic ideas behind the results.
\end{itemize}

\section{The devil is in the details}

We assume basic familiarity with the quaternions $\HH$,
with the basis elements $1,i,j,k$.
The noncommutative multiplication requires subtle modifications
to the associated linear algebra (see \cite{B51}, \cite{W20}).
Of particular use is the ``commutativity'' formula
\begin{equation}
\label{jzcommute}
j z= \overline{z}j, \qquad z\in\CC.
\end{equation}
We will consider polynomials on real, complex and quaternionic
space $\Rd$, $\Cd$ and $\Hd$. With the {\bf Euclidean inner product}
\begin{equation}
\label{Innerproductdefn}
\inpro{v,w}:=v^*w=\sum_j \overline{v_j} w_j, \qquad v,w\in\Fd,
\quad\FF=\RR,\CC,\HH,
\end{equation}
where $\overline{q}$ is the conjugate on $\HH$ (and hence $\RR$ and
$\CC$). In the sum above, $j$ is an index, rather than the quaternion $j$,
for which we also use the same symbol (this is commonly done).
We will use $\FF$ and $\KK$ to stand
for either of $\RR,\CC,\HH$, independently ($9$ cases in all
for maps $\FF^d\to\KK$).
Given our choice (\ref{Innerproductdefn}), it is natural to then treat
$\Hd$ as a right $\HH$-module, with the $\HH$-linear maps $L$ acting on the left, i.e.,
\begin{equation}
\inpro{v\ga,w\gb}= \overline{\ga}\inpro{v,w}\gb, \qquad
(Lv)\ga = L(v\ga),
\label{linmapscalarmultinpro}
\end{equation}
and in turn, to make the identification (\ref{CDHwithC}).

There are natural identifications of $\FF^d$ with $\RR^{md}$, where $m:=\dim_\RR(\FF)$,
given by the Cayley-Dickson construction of $\CC$ and $\HH$ from $\RR$,
e.g., with $(i_1,i_2,i_3,i_4):=(1,i,j,k)$, we have
 (the $\RR$-linear map)
$[\cdot]_{\RR^{md}} :\Fd\to\RR^{md}$ given by
\begin{equation}
\label{CayleyDforF}
[x_1+i_2 x_2+\cdots+i_m x_m]_{\RR^{md}} = (x_1,\ldots,x_m),
\qquad x_1,\ldots,x_m\in\Rd.
\end{equation}  
We say $f:\Fd\to\RR$ is a {\bf polynomial} if $f([\cdot]_{\RR^{md}}^{-1}):\RR^{md}\to\RR$
is a polynomial (of $md$ real variables).
In this way, we can define {\bf homogeneous} and ({\bf homogeneous})
{\bf harmonic} polynomials $f:\Fd\to\RR$ of {\bf degree} $k$.
These polynomials are real-valued, and naturally form $\RR$-vector spaces,
which we denote by
$\Hom_k(\Fd,\RR)$ and $\Harm_k(\Fd,\RR)$.

There is a purely algebraic way to make a finite-dimensional real-vector
space into a complex-vector space, and into a (left or right) $\HH$-vector
space ($\HH$-module), by formally multiplying by complex and quaternion
scalars. In this way, we define the $\KK$-valued $\KK$-vector spaces of
homogeneous and harmonic polynomials $\Fd\to\KK$, which we denote by
$\Hom_k(\Fd,\KK)$ and $\Harm_k(\Fd,\KK)$. Clearly, with $ r=\dim_\RR(\KK)$,
we have
$$ f=f_1+f_2i_2+\cdots+f_ri_r\in\Harm_k(\Fd,\KK)
\Iff f_1,\dots,f_r\in\Harm_k(\Fd,\RR), $$
and similarly for $\Hom_k(\Fd,\KK)$.
Such a $\KK$-vector spaces can also be viewed as a $\LL$-vector spaces,
where $\LL\in\{\RR,\CC,\HH\}$ and $\LL\subset\KK$.
We thereby have (from the real case) the following dimension formulas
\begin{align}
\label{HarmHomdims}
\dim_\LL(\Hom_k(\Fd,\KK)) &= {k+md-1\choose md-1} \dim_\LL(\KK) , \cr
\dim_\LL(\Harm_k(\Fd,\KK)) &=
\Bigl\{ {k+md-1\choose md-1}-{k+md-3\choose md-1}\Bigr\} \dim_\LL(\KK) ,
\quad  md\ne1 \cr 
& \hskip-1.5cm
= (2k+md-2){(k+md-3)!\over (md-2)!k!} \dim_\LL(\KK), \quad
k+md-3\ge0.
\end{align}
For polynomials $f:\Fd\to\KK$, we can define an action of
a group $G$ from its action on $\Fd$ via
\begin{equation}
\label{actiononpolys}
(g\cdot f)(x) := f( g\cdot x), \qquad x\in\Fd,
\end{equation}
provided that 
$[g]_{\RR^{md}}:\RR^{md}\to\RR^{md}:[x]_{\RR^{md}}\mapsto[g\cdot x]_{\RR^{md}}$
is $\RR$-linear.
Such a group action preserves the harmonic polynomials of degree $k$
provided that $[g]_{\RR^{md}}$ is orthogonal.
Since we are only interested in the invariant polynomial subspaces under
such an action, it makes no essential difference if we take
a left or right action. We say that a nonzero $\KK$-subspace $V$ of harmonic polynomials
$\Fd\to\KK$ is {\bf irreducible} (under the action of $G$) if its only
$G$-invariant subspaces are $V$ and $\{0\}$, i.e., for every nonzero
$f\in V$ we have $\spam_\KK\{f\}=V$.

We are primarily concerned with polynomials restricted to the
(unit) {\bf sphere}
$$ \SS:=\{x\in\Fd:\norm{[x]_{\RR^{md}}}=1\}. $$
Hence the linear maps $[g]_{\RR^{md}}$ above must be orthogonal,
i.e., belong to the orthogonal group $O(md)=U(\RR^{md})$.
We note that $f\mapsto f|_\SS$ gives a $\KK$-vector space
isomorphism between $\Harm_k(\Fd,\KK)$ and $\Harm_k(\Fd,\KK)|_\SS$,
with terms {\bf solid} and {\bf surface} used if it is necessary
to distinguish between them.
The basic principles in play are:
\begin{itemize}
\item We mostly consider polynomials $\Hd\to\CC$, since
$\HH$-valued polynomials do not commute, and there is a
well developed theory of representations over $\CC$.
\item Smaller subgroups $G$ of $O(md)$ give smaller irreducible
subspaces, which may lead to finer decompositions (more irreducibles).
\item Enlarging the field $\KK$ (to $\CC$ or $\HH$) preserves
invariance of subspaces, but may not preserve irreducibility,
which may lead to finer decompositions
(Example \ref{enlargefield}).

\item 
The irreducibles that are involved in a decomposition are of interest. 
The sum of all subspaces isomorphic to a given irreducible
is called the {\it homogeneous} or {\it isotypic component} 
(for the irreducible), and it is unique.
As an extreme case, all the irreducibles for the action of the trivial
group are the $1$-dimensional subspaces, and there is a single
(uninteresting) homogeneous component.

\item The reproducing kernel $K(x,y)$ for a unitarily invariant polynomial space
should depend only on the inner product $\inpro{x,y}$.
\item We begin with general homogeneous polynomials for which there are
natural (monomial) bases and useful inner products. We then specialise
to those which are harmonic, and then, ultimately,  zonal.
\end{itemize}

To develop explicit formulas, we 
use the
Cayley-Dickson identifications $\Cd\cong \R^{2d}$ of 
(\ref{CayleyDforF}), and $\Hd\cong\C^{2d}$
given by
\begin{equation}
\label{CDCwithR}
x+iy\in\Cd \identifywith [x+iy]_\RR:=[x+iy]_{\RR^{2d}}=(x,y)\in\RR^{2d},
\end{equation}
\begin{equation}
\label{CDHwithC}
z+jw\in\Hd \identifywith [z+jw]_\CC:=(z,w)\in\CC^{2d}.
\end{equation}
In particular, the identification
(\ref{CDHwithC})
ensures that $[\cdot]_{\CC}:\Hd\to\CC^{2d}$ is $\CC$-linear,
for $\Hd$ as a right vector space, i.e.,
$$ [(z+jw)\ga]_\CC=[z\ga+jw\ga]_\CC=\pmat{z\ga\cr w\ga}=\pmat{z\cr w} \ga
= [(z+jw)]_\CC\, \ga, \qquad\ga\in\CC. $$
It is convenient to define an identification $\Hd\cong\RR^{4d}$ by
\begin{equation}
\label{CDHwithR}
[z+jw]_\RR:= [[(z+jw)]_\CC]_\RR=[(z,w)]_\RR=(\Re(z),\Re(w),\Im(z),\Im(w)).
\end{equation}
We note that
$$ [z+jw]_{\RR^{4d}} 
=(\Re(z),\Im(z),\Re(w),-\Im(w)]. $$

We will use standard multi-index notation for monomials of degree $k$, e.g.,
$$ z^\ga := z_1^{\ga_1}z_2^{\ga_2}\cdots z_d^{\ga_d}, \qquad
|\ga|:= \ga_1+\cdots+\ga_d, \qquad
\ga\in\ZZ_+^d. $$
By a dimension count, the ${k+4d-1\choose 4d-1}$ monomials 
\begin{equation}
\label{monomialma}
m_a:\Hd\to\CC: z+jw\mapsto z^{a_1} w^{a_2}\overline{z}^{a_3}\overline{w}^{a_4},\qquad
 |a|=k, \quad a=(a_1,\ldots,a_4)\in\ZZ_+^{4d}, 
\end{equation}
are a basis for $\Hom_k(\Hd,\CC)$. 
We will often write $z^{a_1} w^{a_2}\overline{z}^{a_3}\overline{w}^{a_4}$ for
the monomial $m_a$.

For $z=x+iy\in\CC$, we define the {\bf Wirtinger derivatives} in the usual way, i.e.,
\begin{equation}
\label{Wirtingerdefn}
{\partial\over\partial z}
:= {1\over2}\Bigl({\partial\over\partial x}-i{\partial\over\partial y}\Bigr), \qquad
{\partial\over\partial\overline{z}}
:= {1\over2}\Bigl({\partial\over\partial x}+i{\partial\over\partial y}\Bigr).
\end{equation}
Let $\gD$ be the {\bf Laplacian} operator on functions $\Hd\to\CC$, which is
given by
\begin{equation}
\label{Laplaciandefn}
{1\over4}\gD = \sum_{j=1}^d \Bigl(
{\partial^2\over\partial\overline{z_j} \partial z_j}
+{\partial^2\over\partial\overline{w_j} \partial w_j}
\Bigr). 
\end{equation}
By applying this, we see that the monomials
$$ z^\ga w^\gb, \qquad \overline{z}^\ga \overline{w}^\gb,
\qquad z^\ga \overline{w}^\gb, \qquad \overline{z}^\ga {w}^\gb, $$
are harmonic, i.e., in the kernel of $\gD$. 


Let $U\in\CC^{4d\times 4d}$ be the unitary matrix given by
$$ [z+wj]_\RR = \pmat{\Re(z)\cr\Re(w)\cr\Im(z)\cr\Im(w)}
= U\pmat{z\cr w\cr\overline{z}\cr\overline{w}}, \quad z,w\in\C, \qquad
U={1\over2}\pmat{1&0&1&0 \cr 0&1&0&1\cr -i&0&i&0\cr 0&-i&0&i}. $$

\begin{example}
Right scalar multiplication of $\Hd$ by $\ga+j\gb\in\HH$ under these
identifications is given by 
\begin{align*}
& \Hd\to\Hd : z+jw \mapsto (z+jw)(\ga+j\gb) 
= (z\ga-\overline{w}\gb)+j(\overline{z}\gb+w\ga), \cr
& \CC^{2d}\to\CC^{2d} : \pmat{z\cr w} \mapsto
\pmat{ z\ga-\overline{w}\gb \cr \overline{z}\gb+w\ga}, \cr
& \RR^{4d}\to\RR^{4d} : [z+wj]_\RR\mapsto M_{\ga+j\gb}[z+wj]_\RR,
\end{align*}
$$ M_{\ga+j\gb} 
= U\pmat{\ga&0&0&-\gb\cr 0&\ga&\gb&0\cr 0&-\overline{\gb}&\overline{\ga}&0\cr
\overline{\gb}&0&0&\overline{\ga}}\otimes I U^* 
= {1\over2} \pmat{\Re(\ga)&-\Re(\gb)&-\Im(\ga)&-\Im(\gb)\cr
\Re(\gb)&\Re(\ga)&\Im(\gb)&-\Im(\ga)\cr
\Im(\ga)&-\Im(\gb)&\Re(\ga)&\Re(\gb) \cr
\Im(\gb)&\Im(\ga)&-\Re(\gb)&\Re(\ga)}\otimes I, $$
where $I=I_d$ is the $d\times d$ identity matrix.
Note that this map is only $\RR$-linear, and so there are no matrix 
representations for it as a map from $\Hd\to\Hd$ or $\CC^{2d}\to\CC^{2d}$.
\end{example}

\begin{example} Consider left multiplication by a 
linear map $L=A+jB$, $A,B\in\CC^{d\times d}$. 
By (\ref{jzcommute}), we
obtain the following matrix representations
under our identifications
\begin{align*}
& \Hd\to\Hd : z+jw \mapsto (A+jB)(z+jw)
= Az-\overline{B}w + j(Bz+\overline{A}w), \cr
& \CC^{2d}\to\CC^{2d} : \pmat{z\cr w} 
\mapsto \pmat{A&-\overline{B}\cr B&\overline{A}}\pmat{z\cr w}, \cr
& \RR^{4d}\to\RR^{4d} : [z+wj]_\RR\mapsto M_{\ga,\gb}[z+wj]_\RR,
\end{align*}
$$ M_{A,B} = U \pmat{A&-\overline{B}&0&0 \cr B&\overline{A} &0&0 \cr
0&0& \overline{A}&-B \cr
0&0& \overline{B}&A} U^*
= {1\over2} \pmat{\Re(A)&-\Re(B)&-\Im(A)&-\Im(B)\cr
\Re(B)&\Re(A)&-\Im(B)&\Im(A)\cr
\Im(A)&\Im(B)&\Re(A)&-\Re(B) \cr
\Im(B)&-\Im(A)&\Re(B)&\Re(A)}. $$
\end{example}

Matrix multiplication on the left (which includes left scalar multiplication) 
commutes with right scalar multiplication, by the associative law
\begin{equation}
\label{Lmatscalarcommute/associative}
(L v)\ga=L(v\ga), \qquad v\in\Hd,\ \ga\in\HH.
\end{equation}
Conversely, those matrices in $\RR^{4d}\times\RR^{4d}$ which commute with
all $M_{\ga+j\gb}$  (equivalently $M_1$, $M_i$, $M_j$ and $M_k$) correspond to the 
matrices $L\in\HH^{d\times d}$, and are said to be {\bf symplectic}.

The {\bf compact symplectic group} $\Sp(d)$,
{\bf quaternionic unitary group} $U(\Hd)$ or {\bf hyperunitary group}
(see \cite{H15} \S1.2.8) is the group of unitary matrices in $\HH^{d\times d}$ 
for the inner product (\ref{Innerproductdefn}), or, equivalently, the 
symplectic matrices in $\RR^{4d\times4d}$ which are orthogonal.
These may also be viewed as the unitary matrices of the form
$$ \CC^{2d}\to\CC^{2d} : \pmat{z\cr w} 
\mapsto \pmat{A&-\overline{B}\cr B&\overline{A}}\pmat{z\cr w}, \qquad
A^*A+B^*B=I, \quad A^TB-B^TA=0. $$
In particular, $\Sp(1)=\HH^*$ is the group of unit quaternions 
or, the special unitary group
$$ SU(2)=\bigl\{\pmat{\ga&-\overline{\gb}\cr\gb&\overline{\ga}}:
|\ga|^2+|\gb|^2=1 , \ga,\gb\in\CC \bigr\}, $$
which therefore have the same irreducible representations (see \cite{F95}
\S5.4).

\section{The operators $R$ and $L$}

A subspace $V$ of $\Harm_k(\HH^d,\CC)$ is {\bf invariant} under
the right multiplication by $\HH^*$ if $f(z+jw)\in V$ implies
$f\bigl((z+jw)(\ga+j\gb)\bigr)\in V$, and similarly for left multiplication.
The following elementary example shows how we came to the operators $R$ and $L$.

\begin{example}
\label{Harm3example}
Suppose that $V\subset \Harm_3(\HH^1,\CC)$ is invariant under right multiplication 
by $\HH^*$,
and $f\bigl((z+jw)\bigr)=z^2w\in V$, then
\begin{align*}
f\bigl((z+jw)(\ga+j\gb)\bigr) 
&= ( z\ga-\overline{w}\gb)^2(\overline{z}\gb+w\ga) \cr
& = \ga^3 z^2w 
+\ga^2\gb(z^2\overline{z}-2zw\overline{w})
+\ga\gb^2(w\overline{w}^2-2z\overline{z}\overline{w}) 
+\gb^3\overline{z}\overline{w}^2\in V. 
\end{align*}
By taking different choices for $\ga$ and $\gb$, it is not hard to see
that the ``coefficients'' of the monomials $\ga^3,\ga^2\gb,\ga\gb^2,\gb^3$
above are in $V$, i.e.,
$$ z^2w, z^2\overline{z}-2zw\overline{w}, w\overline{w}^2-2z\overline{z}\overline{w},
\overline{z}\overline{w}^2\in V.  $$
\end{example}

This example naturally generalises as follows.


\begin{lemma}
\label{Rmultbyscalars}
Let $f\in\Hom_k(\HH^d,\CC)$
be given by $f=f(z+jw)=F(z, w, \overline{z}, \overline{w})$,
and $V_f$ be the subspace invariant
under right multiplication by $\HH^*$ generated by $f$. Then $V_f$ contains
$$ f\left((z+jw)(\ga+j\gb)\right)
= F(z\ga-\overline{w}\gb,\overline{z}\gb+w\ga,
\overline{z}\overline{\ga}-w\overline{\gb},z\overline{\gb}
+\overline{w}\overline{\ga}), \quad \ga,\gb\in\CC, $$
and all its partial derivatives in the variables
$\ga,\overline{\ga},\gb,\overline{\gb}$, including
\begin{equation}
\label{Raf}
R_a f 
:= {\partial^{a_1+a_2+a_3+a_4}\over \partial\ga^{a_1} \partial\gb^{a_2} 
\partial\overline{\ga}^{a_3} \partial\overline{\gb}^{a_4}}
 f\left((z+jw)(\ga+j\gb)\right)\Big|_{\ga=1,\gb=0}.
\end{equation}
Moreover, if $f$ is harmonic, then so are all the polynomials in $V_f$.
\end{lemma}

\begin{proof}
Clearly, $V_f$ is the  
subspace of $\Hom_k(\Hd,\CC)$ given by 
$$ V_f=\spam_\CC\{f\left((z+jw)q\right):q=\ga+j\gb\in\HH^*\}. $$
and hence for $q=\ga+j\gb$ nonzero, we have
$$ f\left((z+jw)(\ga+j\gb)\right)=|q|^k f((z+jw)(q/|q|))\in V_f. $$
Since $V_f$ is a finite-dimensional vector space, it follows that the 
first order partials in $\ga,\overline{\ga},\gb,\overline{\gb}$,
which are limits of Newton quotients in $V_f$, are in $V_f$, 
and therefore so are all the partial derivatives.

A calculation shows
that if
$f:\RR^n\to\RR$ is harmonic, then so is $f\circ U$ for $U:\RR^n\to\RR^n$ orthogonal.
Since scalar multiplication of $\Hd$ by a unit quaternion
(left or right) is an orthogonal map $\RR^{4d}\to\RR^{4d}$,
it follows that if $f(z+jw)$ is harmonic, then so is $f((z+jw)q)$, $q\in\HH^*$,
and hence every polynomial in $V_f$.
\end{proof}

In other words, if a subspace $V\subset\Hom_k(\Hd,\CC)$ 
is invariant under right multiplication by $\HH^*$, then it is invariant under the 
action of the operators $R_a$ of (\ref{Raf}).
Since the partial derivatives in (\ref{Raf}) for $|a|=a_1+a_2+a_3+a_4=k$
do not depend on $\ga,\gb,\overline{\ga},\overline{\gb}$, they can 
be ``evaluated'' at $\ga=0,\gb=0$, to obtain the Taylor formula
\begin{equation}
\label{RafTaylor}
f((z+jw)(\ga+j\gb)) = \sum_{|a|=k}
 R_a(f) {\ga^{a_1}\gb^{a_2}\overline{\ga}^{a_3}\overline{\gb}^{a_4}\over
a_1!a_2!a_3!a_4!}, \quad f\in\Hom_k(\Hd,\CC).
\end{equation}
We therefore have the following converse result.

\begin{proposition} A subspace $V\subset\Hom_k(\Hd,\CC)$ is invariant under right multiplication by $\HH^*$ if and only if it is invariant under the operators $R_a$, $|a|=k$.
\end{proposition}

\begin{proof} As already observed, the forward implication follows 
from Lemma \ref{Rmultbyscalars}.

Conversely, suppose that $V$ is invariant under right multiplication by $\HH^*$,
and $f\in V$.
Since the monomials in $\ga,\gb,\overline{\ga},\overline{\gb}$ 
in the Taylor formula (\ref{RafTaylor})
are linearly independent, it follows that
$$ V_f =\spam_\CC\{ R_a f:|a|=k\} \subset V, $$
and so $V$ is invariant under right multiplication by the operators $R_a$, $|a|=k$.
\end{proof}

This observation means that 
\begin{itemize}
\item We can replace {\it invariance under right multiplication by the continuous group $\HH^*$}
by {\it invariance under the discrete set of operators $\{R_a\}_{|a|=k}$}.
\end{itemize}
This is the basic spirit of our development, where
\begin{itemize}
\item The operators $\{R_a\}_{|a|=k}$ are homogeneous differential operators
of order $k$ which are replaced by two first order operators $R$ and $R^*$,
with $R^{k+1}=(R^*)^{k+1}=0$.
\end{itemize}

There is an obvious parallel development for the left multiplication by $\HH^*$ 
where the role of $R_a$ is played by $L_a$, where
$$ f\left( (\ga+j\gb) (z+jw) \right) = F( \ga z -\overline{\gb} w, \overline{\ga}w+\gb z,
\overline{\ga}\overline{z}-\gb\overline{w},\ga\overline{w}+\overline{\gb}\overline{z}), $$
\begin{equation}
\label{Laf}
L_a f 
:= {\partial^{a_1+a_2+a_3+a_4}\over \partial\ga^{a_1} \partial\gb^{a_2} 
\partial\overline{\ga}^{a_3} \partial\overline{\gb}^{a_4}}
 f\left( (\ga+j\gb) (z+jw) \right)\Big|_{\ga=1,\gb=0}.
\end{equation}

\begin{example}
\label{Harm3exampleII}
For the Example \ref{Harm3example}, i.e., $f=z^2w$, the nonzero terms in  (\ref{RafTaylor})
are
$$ R_{3,0,0,0}f=6z^2w, \quad
   R_{2,1,0,0}f=2z^2\overline{z}-4zw\overline{w}, \quad
   R_{1,2,0,0}f=2w\overline{w}^2-4z\overline{z}\overline{w}, \quad
   R_{0,3,0,0}f=6\overline{z}\overline{w}^2. $$
The nonzero polynomials $R_a f$ are not a basis for $V_f$ in general, 
e.g., for $f=z\overline{w}$ one has
$$ R_{1,0,1,0}f=z\overline{w}, \quad
R_{0,1,0,1 }f=-z\overline{w}, \quad
R_{1,0,0,1}f=z^2, \quad
R_{0,1,1,0}f=-\overline{w}^2. $$
\end{example}

To illustrate the nature of the definition (\ref{Raf}) of $R_a$, we consider
the mechanics of calculating $R_{0,1,0,0} f$ for $d=1$. 
Let $C=(z\ga-\overline{w}\gb,\overline{z}\gb+w\ga,
\overline{z}\overline{\ga}-w\overline{\gb},z\overline{\gb}
+\overline{w}\overline{\ga})$. Then differentiating gives
\begin{align*}
{\partial\over\partial\gb}\,
 f\left((z+jw)(\ga+j\gb)\right)
& = D_1F(C){\partial\over\partial\gb}(z\ga-\overline{w}\gb)
+D_2F(C){\partial\over\partial\gb}(\overline{z}\gb+w\ga)+\cdots \cr
& = -\overline{w}D_1F(C)+\overline{z}D_2F(C)
+0D_3F(C)+0D_4F(C),
\end{align*}
and evaluating this at $\ga=1$, $\gb=0$ gives
$C=(z, w, \overline{z}, \overline{w})$, and hence
\begin{equation}
\label{betapartial}
R_{0,1,0,0}f = {\partial\over\partial\gb}\,
 f\left((z+jw)(\ga+j\gb)\right)\Big|_{\ga=1,\gb=0}
= -\overline{w}{\partial f\over\partial z}+\overline{z}
{\partial f \over \partial w}.
\end{equation}
The calculation for $d\ge1$, is the same, leading to 
$$ R_{0,1,0,0}f = \sum_{j=1}^d \Bigl( -\overline{w_j}{\partial f\over\partial z_j}
+\overline{z_j} {\partial f \over \partial w_j}\Bigr). $$
For readability, we will often give the $d=1$ case, with general case
following by replacing $z$ by $z_j$, etc, and summing over $j$ (or applying
the Einstein summation convention).

For the first order differential operators $R_a$, 
we will use the suggestive notation
$$ R_\ga := R_{1,0,0,0} ={\partial\over\partial\ga}\Big|, \qquad
R_\gb := R_{0,1,0,0} ={\partial\over\partial\gb}\Big|, $$
$$ R_{\overline{\ga}} := R_{0,0,1,0} ={\partial\over\partial\overline{\ga}}\Big|, \qquad
R_{\overline{\gb}} := R_{0,0,0,1} ={\partial\over\partial\overline{\gb}}\Big|, $$
and similarly for $L_\ga,L_\gb,\ldots$.
All of these first order operators $T$ satisfy the product rule
\begin{equation}
\label{RLproductrule}
T(fg)=T(f)g+fT(g).
\end{equation}
For $d=1$, they are
\begin{equation}
\label{Rbeta1dim}
R_\gb = -\overline{w}{\partial\over\partial z}+\overline{z}
{\partial \over \partial w}, \qquad
R_{\overline{\gb}} = -w{\partial\over\partial\overline{z}}
+z {\partial \over \partial\overline{w}},
\end{equation}
\begin{equation}
\label{Ralpha1dim}
R_\ga = z{\partial\over\partial z}+w {\partial \over \partial w}, \qquad
R_{\overline{\ga}} = \overline{z}{\partial\over\partial\overline{z}}
+\overline{w} {\partial \over \partial\overline{w}},
\end{equation}
\begin{equation}
\label{Lbeta1dim}
L_\gb = z {\partial \over \partial w} 
-\overline{w} {\partial\over\partial\overline{z}}, \qquad
L_{\overline{\gb}} = \overline{z} {\partial \over \partial\overline{w}} 
-w{\partial\over\partial z}, 
\end{equation}
\begin{equation}
\label{Lalpha1dim}
L_\ga = z {\partial \over \partial z} 
+ \overline{w}{\partial \over \partial\overline{w}} , \qquad
L_{\overline{\ga}} = w {\partial \over \partial w} 
+ \overline{z} {\partial \over \partial\overline{z}}.
\end{equation}
Of particular interest, are the operators 
\begin{equation}
\label{RLdefn}
R:=-R_\gb, \qquad R^*:=R_{\overline{\gb}}, \qquad
L:=-L_{\overline{\gb}}, \qquad L^*:= L_{\gb},
\end{equation}
which in the $1$-dimensional case have the form
\begin{equation}
\label{Rformula}
R = \overline{w}{\partial\over\partial z}-\overline{z} {\partial\over\partial w}, \qquad
R^* = -w{\partial\over\partial\overline{z}}+z {\partial\over\partial\overline{w}},
\end{equation}
\begin{equation}
\label{Lformula}
L = w{\partial\over\partial z}-\overline{z} {\partial\over\partial\overline{w}}, \qquad
L^*=z{\partial\over\partial w}-\overline{w} {\partial\over\partial\overline{z}}.
\end{equation}
and for general $d$ are given by
\begin{equation}
\label{Rformulagen}
R = \sum_{j=1}^d\Bigl(\overline{w_j}{\partial\over\partial z_j}
-\overline{z_j} {\partial\over\partial w_j}\Bigr), \qquad
R^* = \sum_{j=1}^d\Bigl(-w_j{\partial\over\partial\overline{z_j}}
+z_j {\partial\over\partial\overline{w_j}}\Bigr),
\end{equation}
\begin{equation}
\label{Lformulagen}
L = \sum_{j=1}^d\Bigl(w_j{\partial\over\partial z_j}
-\overline{z_j} {\partial\over\partial\overline{w_j}}\Bigr), \qquad
L^*=\sum_{j=1}^d\Bigl(z_j{\partial\over\partial w_j}
-\overline{w_j} {\partial\over\partial\overline{z_j}}\Bigr).
\end{equation}
The notation $R^*$ and $L^*$ is used, as we will see (Lemma \ref{Rstarisadjoint})
that they are the adjoints of $R$ and $L$, respectively, 
for two natural inner products. The operators $R$ and $R^*$ 
(but not $L$ and $L^*$) appear in the work of \cite{Ghent14}, \cite{Ghent18}
as $\gep=R^*$ and $\gep^\dagger= R$. Operators of this type (for $d=1$) also
appear in the construction of irreducible representations of 
$SU(2)=\HH^*$ on the homogeneous polynomials in $z$ and $w$ of degree $k$
given in \cite{F95}.

The following ``ansatz'' indicates our approach.

\begin{ansatz} 
\label{ansatzhopetoprove}
On $\Hom_k(\Hd,\CC)$, each operator $R_a$, $|a|=k$, can be written as
a polynomial of degree $k$ with real coefficients in the noncommuting
variables $R$ and $R^*$.
Therefore, a subspace  of  $\Hom_k(\Hd,\CC)$ is invariant under right multiplication by scalars in
$\HH^*$ if and only if it is invariant under $R$ and $R^*$.
\end{ansatz}

The analogous statement holds for left multiplication and the action of $L$ and $L^*$. 
It was hoped to prove this, by exhibiting the polynomials in $R$ and $R^*$ 
explicitly, or by a simple inductive argument, but this is complicated by the 
fact that $R$ and $R^*$ do not commute. We now discuss some relevant
algebraic properties 
of $R$, $R^*$, etc.

\begin{lemma} We have the commutativity relations
\begin{equation}
\label{Rcommuterels1}
R_\gb R_{\overline{\gb}} -  R_{\overline{\gb}} R_\gb = R_\ga - R_{\overline{\ga}},
\qquad  R_\ga R_{\overline{\ga}} = R_{\overline{\ga}}R_\ga,
\end{equation}
\begin{equation}
\label{Rcommuterels2}
R_\gb R_\ga-R_\ga R_\gb = R_\gb, \qquad
R_\ga R_{\overline{\gb}} -R_{\overline{\gb}}  R_\ga
= R_{\overline{\gb}}, 
\end{equation}
\begin{equation}
\label{Rcommuterels3}
R_{\overline{\ga}} R_\gb -R_\gb R_{\overline{\ga}} = R_\gb, \qquad
R_{\overline{\gb}}  R_{\overline{\ga}} -R_{\overline{\ga}} R_{\overline{\gb}} 
= R_{\overline{\gb}}. 
\end{equation}
Furthermore, on $\Hom_k(\Hd,\CC)$ we have
\begin{equation}
\label{addRalpaRalphab}
R_\ga+R_{\overline{\ga}} = k I, 
\end{equation}
and hence
\begin{equation}
\label{RalphainR}
R_\ga = {1\over2}( R^*R - R R^*+kI), \qquad
R_{\overline{\ga}} = {1\over2}( RR^* - R^*R +kI).
\end{equation}
\end{lemma}

\begin{proof}
The relations (\ref{Rcommuterels1}), (\ref{Rcommuterels2}) and (\ref{Rcommuterels3})
directly from the formulas of the type 
(\ref{Rbeta1dim}) and (\ref{Ralpha1dim}),
e.g., the first is proved in Lemma \ref{LRL*R*commute}. 
For (\ref{addRalpaRalphab}), apply $R_\ga$ and $R_{\overline{\ga}}$ to 
a monomial $f=z^{a_1}w^{a_2}\overline{z}^{a_3}\overline{w}^{a_4}$, 
$|a|=|a_1|+|a_2|+|a_3|+|a_4|=k$,
$$ R_\ga f 
=\sum_j \Bigl( z_j{\partial f\over z_j} +w_j{\partial f\over w_j} \Bigr)
= \sum_j \bigl( (a_1)_j f + (a_2)_j f \bigr)
=(|a_1|+|a_2|)f, $$
$$ R_{\overline{\ga}} f 
=\sum_j \Bigl( \overline{z_j}{\partial f\over\overline{z_j}} 
+\overline{w_j}{\partial f\over\overline{w_j}} \Bigr)
= \sum_j \bigl( (a_3)_j f + (a_4)_j f \bigr)
=(|a_3|+|a_4|)f, $$
and add to get $R_\ga f+R_{\overline{\ga}} f=kf$. The equations
(\ref{RalphainR}) then follow by using (\ref{addRalpaRalphab}) to eliminate
$R_{\overline{\ga}}$ and $R_\ga$ from $R_\gb R_\ga-R_\ga R_\gb = R_\gb$.
\end{proof}

In view of (\ref{RalphainR}) all ``polynomials'' in $R_\gb,R_{\overline{\gb}},
R_\ga,R_{\overline{\ga}}$ can be written as polynomials in the 
noncommuting variables $R$ and $R^*$.

\begin{example}
The commutativity relations (\ref{Rcommuterels2}) and
(\ref{Rcommuterels3}) can be used to exchange $\ga$ and $\gb$ factors
as follows
$$ R_\gb R_\ga = (R_\ga+I) R_\gb, \qquad
R_\ga R_\gb = R_\gb (R_\ga -I),  $$ 
$$ R_\ga R_{\overline{\gb}} = R_{\overline{\gb}} (R_\ga+I), \qquad
R_{\overline{\gb}}  R_\ga = (R_\ga-I) R_{\overline{\gb}}, $$
$$ R_{\overline{\ga}} R_\gb 
=R_\gb(R_{\overline{\ga}} +I),\qquad
R_\gb R_{\overline{\ga}} = (R_{\overline{\ga}}-I) R_\gb, $$
$$ R_{\overline{\gb}}  R_{\overline{\ga}} = 
(R_{\overline{\ga}}+I) R_{\overline{\gb}}, \qquad 
R_{\overline{\ga}} R_{\overline{\gb}} = 
R_{\overline{\gb}}  (R_{\overline{\ga}}-I). $$
In this way, one might hope to put the ``polynomials'' giving
$R_a$ in terms of $R$ and $R^*$ alluded to 
in Ansatz \ref{ansatzhopetoprove} into some canonical form,
which could then be proved, e.g., by verifying it on the monomials.
To this end, we have formulas such as
$$ R_{a_1,a_2,a_3,0}= {\partial^{a_1+a_2+a_3}\over\partial\ga^{a_1}\partial\gb^{a_2}\partial\overline{\ga}^{a_3}}
\Big| = \Bigl( \prod_{j_1=0}^{a_1-1}(R_\ga-j_1I)\Bigr)
R_\gb^{a_2}
\Bigl( \prod_{j_3=0}^{a_3-1}(R_{\overline{\ga}}-j_3I) \Bigr), $$ 
$$ R_{a_1,0,a_3,a_4}=  {\partial^{a_1+a_3+a_4}\over\partial\ga^{a_1} \partial\overline{\ga}^{a_3} \partial\overline{\gb}^{a_4} }
\Big| = \Bigl( \prod_{j_3=0}^{a_3-1}(R_{\overline{\ga}}-j_3I) \Bigr) 
R_{\overline{\gb}}^{a_4}
\Bigl( \prod_{j_1=0}^{a_1-1}(R_\ga-j_1I)\Bigr), $$
$$ {\partial^{b_1+b_2}\over\partial\gb^{b_1}\partial\overline{\gb}^{b_2}}\Big|
= R_\gb^{b_1} R_{\overline{\gb}}^{b_2}
+b_1 b_2 R_\gb^{b_1-1}R_{\overline{\ga}}R_{\overline{\gb}}^{b_2-1}
+\hbox{lower order terms}, $$
where the lower order terms are zero
when either of $b_1$ or $b_2$ takes the value $0$ or $1$,
and
$$ l.o.t = b_1(b_1-1){\partial\over\partial\overline{\ga}^2\partial\gb^{b_1-2}}\Big|, \qquad b_2=2, $$
$$ l.o.t = 3b_1(b_1-1){\partial\over\partial\overline{\ga}^2\partial\gb^{b_1-2}
\partial\overline{\gb}} \Big|
- 2b_1(b_1-1)(b_1-2){\partial\over\partial\overline{\ga}^3\partial\gb^{b_1-3}}\Big|
, \qquad b_2=3. $$
However, a general formula has yet to be obtained.
\end{example}

\section{Inner products on the quaternionic sphere}

There are two natural (unitarily invariant) inner products defined on 
polynomials from $\Hd\to\CC$ that we consider.
Let
$$ \SS=\SS(\Fd):=\{x\in\Fd:\norm{x}=1\} = \{x\in\RR^{md}:\norm{x}=1\} $$
be the unit sphere in $\Fd$, and
$\gs$ be the surface area measure on $\SS$,
normalised so that $\gs(\SS)=1$.
We note that surface area measure invariant under unitary maps on $\Fd$,
i.e., for $U$ unitary
$$ \int_{\SS(\Fd)} f(Ux)\,d\gs(x)=\int_{\SS(\Fd)} f(x)\,d\gs(x), \qquad\forall f. $$
The first inner product we consider is defined on complex-valued functions restricted to the 
{\bf quaternionic sphere} $\SS=\SS(\Hd)$ by
\begin{equation}
\label{firstinpro}
\inpro{f,g}=\inpro{f,g}_\SS := \int_{\SS(\Hd)} \overline{f(x)}g(x)\, d\gs(x).
\end{equation}
This can be calculated
from the
well known integrals of the monomials in $z,w,\overline{z},\overline{w}\in\Cd$
(polynomials in $2d$ complex variables)
$$ \int_{\SS(\Hd)} z^{\ga_1}w^{\gb_1}\overline{z}^{\ga_2}\overline{w}^{\gb_2}\,
d\gs
= \int_{\SS(\CC^{2d})} z^{\ga_1}w^{\gb_1}\overline{z}^{\ga_2}\overline{w}^{\gb_2}\,
d\gs(z,w), $$ 
which are zero for $(\ga_1,\gb_1)\ne(\ga_2,\gb_2)$, and otherwise
\begin{equation}
\label{intSformula}
\int_{\SS(\Hd)} z^{\ga_1}w^{\gb_1}\overline{z}^{\ga_2}\overline{w}^{\gb_2}\, d\gs
= {(2d-1)!\ga_1!\gb_1!\over(2d-1+|\ga_1|+|\gb_1|)!}
={\ga_1!\gb_1!\over(2d)_{|\ga_1|+|\gb_1|}}
, 
\qquad (\ga_1,\gb_1)=(\ga_2,\gb_2).
\end{equation}
Here $(x)_n:=x(x+1)\cdots(x+n-1)$ is the Pochhammer symbol.

For a polynomial 
$f=\sum_\ga f_\ga z^{\ga_1}w^{\ga_2}\overline{z}^{\ga_3} \overline{w}^{\ga_4}$ 
mapping $\Hd\to\CC$,
let $\tilde f$ be the polynomial obtained by replacing the coefficient 
$f_\ga\in\CC$ by its
conjugate $\overline{f_\ga}$, and $f(\partial)$ be the 
differential operator obtained replacing $z$ by ${\partial\over\partial z}$, 
etc, i.e.,
$$ \tilde{f}=\sum_\ga\overline{f_\ga}z^{\ga_1}w^{\ga_2}\overline{z}^{\ga_3}
\overline{w}^{\ga_4}, \qquad
f(\partial)=\sum_\ga f_\ga{\partial^{\ga_1+a_2+\ga_3+\ga_4} f\over
\partial z^{\ga_1} \partial w^{\ga_2} \partial\overline{z}^{\ga_3}
\partial \overline{w}^{\ga_4}}. $$
The second inner product is given by
\begin{equation}
\label{secondinpro}
\dinpro{f,g}:=\tilde{f}(\partial)g (0)
=\sum_\ga \ga! \overline{f_\ga} g_\ga .
\end{equation}

The inner products (\ref{firstinpro}) and
(\ref{secondinpro}) are both prominent in the theory of spherical harmonics.
The first is natural for 
Fourier expansions on the sphere, and the second, which is variously known as the
{\bf apolar} \cite{V00}, {\bf Bombieri} \cite{Z94} or {\bf Fischer inner product}
\cite{Ghent14},
 is also widely used. 
Not withstanding the fact that they are defined on different spaces, 
these inner products are different, since the monomials are orthogonal in the second, but not in the first in general, e.g.,
$$ \inpro{z_1\overline{z_1},w_1\overline{w_1}}
=\int_{\SS(\Hd)} |z_1w_1|^2\,d\gs
= {1\over 2d(2d+1)}\ne0, \qquad 
\dinpro{z_1\overline{z_1},w_1\overline{w_1}}=0. $$
Nevertheless, these inner product are scalar multiples of each other 
in the following sense
$$ \dinpro{f,g} = (2d)_k \inpro{f,g}, \qquad f\in\Harm_k(\Hd,\CC), \ g\in\Hom_k(\Hd,\CC), $$
which follows from \cite{FX13} (Theorem 1.1.8) as presented in 
\cite{BSW17} (Lemma 2).

The homogeneous polynomials of different degrees are orthogonal to each other 
for both inner products, giving the orthogonal direct sums
$$ \bigoplus_{k\ge0} \Hom_k(\Hd,\CC)\Big|_{\SS(\Hd)}, \qquad
\bigoplus_{k\ge0} \Hom_k(\Hd,\CC), $$
respectively. For simplicity, we will primarily consider the further
decomposition of $\Harm_k(\Hd,\CC)$, with it being understood that this
leads to a corresponding refinement of the direct sums
\begin{equation}
\label{Fischerdecomp} 
\bigoplus_{k\ge0}\Hom_k(\Hd,\CC)
= \bigoplus_{k\ge0}\bigoplus_{0\le j\le{k\over2}} 
\norm{\cdot}^{2j}\Harm_{k-2j}(\Hd,\CC),
\end{equation}
\begin{equation}
\label{HomkHarmdecomp}
\Hom_k(\Hd,\CC)\big|_\SS = \bigoplus_{0\le j\le{k\over2}} \Harm_{k-2j}(\Hd,\CC),
\end{equation}
of the polynomials $\Hd\to\CC$ into irreducibles for the action of $SO(4d)$.
The direct sum (\ref{Fischerdecomp}) is 
sometimes referred to as the {\bf Fischer decomposition} \cite{Ghent14}.

The adjoints of $R$ and $L$ are the same for both of these inner products.

\begin{lemma} 
\label{Rstarisadjoint}
The operators $R^*$ and $L^*$ are the adjoints of $R$ and $L$ with
respect to both the inner products (\ref{firstinpro}) and
(\ref{secondinpro}) defined on $\Hom_k(\Hd,\CC)$.
\end{lemma}

\begin{proof} This is by direct computation. See Section \ref{appendixnumber}.
\end{proof}

\noindent
This result for $R$ was given in \cite{Ghent14} Lemma 5 for the inner product
(\ref{secondinpro}). 

The adjoint can also be calculated using the following property.

\begin{example}
\label{RLconjugateidentities}
An elementary calculation shows the identities
\begin{equation}
\label{RLconjugateformula}
\overline{Rf}= -R^*(\overline{f}), \qquad
\overline{Lf}= -L^*(\overline{f}),
\end{equation}
and so, on subspaces $V$, we have
\begin{equation}
\label{Rconjids}
\overline{R^\ga V}= (R^*)^\ga \overline{V},
\qquad \overline{\ker R^*|_V} = \ker R|_{\overline{V}}.
\end{equation}
\end{example}

The next result follows from the fact that scalar 
multiplication by $\HH^*$ is in $O(\RR^{4d})$, and hence maps harmonic 
polynomials to harmonic polynomials. 

\begin{lemma}
\label{LapRandLcommute}
The operators $R,R^*,L$ and $L^*$ commute
with the Laplacian $\Delta$, and so map harmonic functions
to harmonic functions.
\end{lemma}

\begin{proof}
A direct proof is given in Section \ref{appendixnumber}.
\end{proof}

It follows from (\ref{Lmatscalarcommute/associative}) 
that the action of $R$ and $R^*$ commutes with that of $U\in U(\Hd)$.
In this regard, recall from (\ref{actiononpolys}) and
(\ref{Raf}) that
$$ (U\cdot f)(z+jw) = f(U(z+jw)). $$

\begin{lemma} 
\label{RandU(Hd)commute}
The operators $R$ and $R^*$ commute with the action of $U(\Hd)$.
\end{lemma}

\begin{proof} We will show, more generally, that the operators 
$R_a$ of (\ref{Raf}) commute with the action of $U(\Hd)$.
Let $U\in U(\Hd)$. Then for $f=f(z+jw)$, we have
\begin{align*}
(U\cdot R_a f )
&= {\partial^{a_1+a_2+a_3+a_4}\over \partial\ga^{a_1} \partial\gb^{a_2} 
\partial\overline{\ga}^{a_3} \partial\overline{\gb}^{a_4}}
 f\left(U(z+jw)(\ga+j\gb)\right)\Big|_{\ga=1,\gb=0} \cr
&= {\partial^{a_1+a_2+a_3+a_4}\over \partial\ga^{a_1} \partial\gb^{a_2} 
\partial\overline{\ga}^{a_3} \partial\overline{\gb}^{a_4}}
 (U\cdot f)\left((z+jw)(\ga+j\gb)\right)\Big|_{\ga=1,\gb=0} \cr
&= R_a(U\cdot f).
\end{align*}
i.e., $R_a$ commutes with the action of $U(\Hd)$.
%
\end{proof}

We note that, by the same reasoning, 
the operators $L$ and $L^*$ do not commute with the (left) action of $U(\Hd)$.

\vfill

\section{The action of $R$ and $L$ on polynomials}

Using (\ref{Rformula}) to apply $R$ to a univariate monomial 
$f=z^{a_1}w^{a_2}\overline{z}^{a_3}\overline{w}^{a_4}$ gives
\begin{align*}
R f
&= \overline{w}{\partial\over\partial z}
(z^{a_1}w^{a_2}\overline{z}^{a_3}\overline{w}^{a_4})
-\overline{z}{\partial\over\partial w}
(z^{a_1}w^{a_2}\overline{z}^{a_3}\overline{w}^{a_4}) \cr
&= a_1 z^{a_1-1}w^{a_2}\overline{z}^{a_3}\overline{w}^{a_4+1}
-a_2 z^{a_1}w^{a_2-1}\overline{z}^{a_3+1}\overline{w}^{a_4},
\end{align*}
which is a sum of monomials in which the degree in $z$ and $w$ has decreased by $1$,
whilst the degree in $\overline{z}$ and $\overline{w}$ has increased by $1$.
This type of phenomenon occurs for all of the operators $R,R^*,L,L^*$
(in every dimension), and we 
now make definitions which allow us to account for these changes in degrees.
With standard multi-index notation, we have
\begin{align*}
\Hom_H(p,q) & :=\spam\{
z^{\ga_1} w^{\ga_2} \overline{z}^{\ga_3} \overline{w}^{\ga_4}:
|\ga_1|+|\ga_2|=p,|\ga_3|+|\ga_4|=q\}, \cr
\Hom_K(p,q) & :=\spam\{
z^{\ga_1} w^{\ga_2} \overline{z}^{\ga_3} \overline{w}^{\ga_4}:
|\ga_1|+|\ga_4|=p,|\ga_2|+|\ga_3|=q\}, \cr
\Hom_k^{(a,b)}(\Hd) & := \Hom_K(k-a,a)\cap \Hom_H(k-b,b) \cr
&\,=\spam\{z^{\ga_1} w^{\ga_2} \overline{z}^{\ga_3} \overline{w}^{\ga_4}:
|\ga_1|+|\ga_4|=k-a,|\ga_2|+|\ga_3|=a
, \cr
& \qquad\qquad \qquad \qquad \qquad\ \ \,
|\ga_1|+|\ga_2|=k-b,|\ga_3|+|\ga_4|=b
\}.
\end{align*}
We observe that
$$ \overline{\Hom_H(p,q)}=\Hom_H(q,p), \qquad
\overline{\Hom_K(p,q)}=\Hom_K(q,p). $$
$$ \Delta \Hom_H(a,b) = \Hom_H(a-1,b-1), \qquad
 \Delta \Hom_K(a,b) = \Hom_K(a-1,b-1). $$
The subspaces of harmonic polynomials contained in these are denoted
\begin{align*}
H(p,q) & := \Harm_k(\Hd,\CC) \cap \Hom_H(p,q), \qquad p+q=k, \cr
K(p,q) & := \Harm_k(\Hd,\CC) \cap \Hom_K(p,q), \qquad p+q=k, \cr
H_k^{(a,b)}(\Hd) & :=\Harm_k(\Hd,\CC)\cap\Hom_k^{(a,b)}(\Hd) \cr
&\, = K(k-a,a)\cap H(k-b,b).
\end{align*}
When either $p$ or $q$ above is negative, then we have, by 
definition, the zero subspace.
The subspaces $H(p,q)$ are the irreducible subspaces of
$\Harm_k(\CC^{2d},\CC)\cong\Harm_k(\Hd,\CC)$ 
under the action of (left) multiplication
by $U(\CC^{2d})$, e.g., see \cite{R80}, from where we borrow the
notation $H(p,q)$.
Since $U(\Hd)$ is a subgroup of $U(\CC^{2d})$, the decomposition
of $\Harm_k(\Hd,\CC)$ into $U(\Hd)$-irreducibles 
is obtained by decomposing each $H(p,q)$.

The following dimensions are easily calculated
\begin{equation}
\label{HomHpqdim}
\dim_\CC(\Hom_H(p,q)) = \dim_\CC(\Hom_K(p,q))
= {p+2d-1\choose 2d-1}{q+2d-1\choose 2d-1},
\end{equation}
\begin{equation}
\label{Hpqdim}
\dim_\CC(H(p,q)) = \dim_\CC(K(p,q))
= (p+q+2d-1) {(p+2d-2)!(q+2d-2)!\over p! q! (2d-1)! (2d-2)!},
\end{equation}
whilst those of 
$\Hom_k^{(a,b)}(\Hd)$ and $H_k^{(a,b)}(\Hd)$
are more complicated
(Lemmas \ref{Homabdecomplemma} and \ref{Habdecomplemma}).

\begin{example}
\label{enlargefield}
	Let $H(p,q)^\RR$ be $H(p,q)\subset \Harm_k(\CC^{2d},\CC)$ 
	viewed as a real vector space, which is invariant under the action of $U(\CC^{2d})$.
Complexifying this space gives
$$ \CC H(p,q)^\RR = H(p,q)\oplus H(q,p). $$
Thus we observe that enlarging the field, which preserves the invariance of subspaces, 
	does not always preserve irreducibility.
\end{example}

\begin{lemma} 
\label{hardtounderstandlemma}
For all integers $a$ and $b$, we have
\label{kerRstarRorthogdecomp}
\begin{align*}
 R \Hom_H(a,b) \subset \Hom_H(a-1,b+1), & \quad
R^* \Hom_H(a,b) \subset \Hom_H(a+1,b-1), \cr
 L \Hom_K(a,b) \subset \Hom_K(a-1,b+1), & \quad
L^* \Hom_K(a,b) \subset \Hom_K(a+1,b-1),
\end{align*}
and $\Hom_K(a,b)$ is invariant under right multiplication by $\HH^*$, 
$\Hom_H(a,b)$ is invariant under left multiplication by $\HH^*$, 
and more generally by $U(\Hd)$,
which gives
\begin{align*}
 R \Hom_K(a,b) \subset \Hom_K(a,b), & \quad
R^* \Hom_K(a,b) \subset \Hom_K(a,b), \cr
 L \Hom_H(a,b) \subset \Hom_H(a,b), & \quad
L^* \Hom_H(a,b) \subset \Hom_H(a,b).
\end{align*}
In particular, for $\ga,\gb\ge0$, we have
\begin{align}
L^\ga R^\gb \Hom_k^{(a,b)}(\Hd,\CC) &\subset 
\Hom_k^{(a+\ga,b+\gb)}(\Hd,\CC), \label{LRtoHomab} \\
(L^*)^\ga (R^*)^\gb \Hom_k^{(a,b)}(\Hd,\CC)& \subset 
\Hom_k^{(a-\ga,b-\gb)}(\Hd,\CC). 
\label{LRstartoHomab}
\end{align}
Moreover, for the inner products (\ref{firstinpro}) and (\ref{secondinpro}) 
we have the orthogonal direct sums
\begin{align}
\Hom_H(k-a,a) &= (\Hom_H(k-a,a)\cap \ker R^*) \oplus R \Hom_H(k-a+1,a-1),
\label{HomHkerRandR*} \\
\Hom_H(k-a,a) &= (\Hom_H(k-a,a)\cap \ker R) \oplus R^* \Hom_H(k-a-1,a+1),
\label{HomHkerR*andR} 
\end{align}
\end{lemma}

\begin{proof}
The inclusions follow by (elementary) direct calculations.

Let $f\in \Hom_H(k-a,a)$. Since $R^*f\in \Hom_H(k-a+1,a-1)$, we have
\begin{align*}
f\in\ker R^* & \Iff R^*f=0
\Iff \inpro{R^*f,g}=0,\quad\forall g\in \Hom_H(k-a+1,a-1) \cr
& \Iff \inpro{f,Rg}=0,\quad\forall g\in \Hom_H(k-a+1,a-1) \cr
& \Iff f\in (R\Hom_H(k-a+1,a-1))^\perp,
\end{align*}
so that
\begin{align*}
\Hom_H(k-a,a) &= (R\Hom_H(k-a+1,a-1))^\perp \oplus R\Hom_H(k-a+1,a-1) \cr
&= (\Hom_H(k-a,a)\cap \ker R^*) \oplus R \Hom_H(k-a+1,a-1),
\end{align*}
which gives (\ref{HomHkerRandR*}). The proof of (\ref{HomHkerR*andR}) is 
similar.
\end{proof}

By restricting (\ref{LRtoHomab}),  (\ref{LRstartoHomab}), (\ref{HomHkerRandR*}), 
and (\ref{HomHkerR*andR}) 
to the harmonic polynomials, we have
\begin{align}
L^\ga R^\gb H_k^{(a,b)}(\Hd,\CC) \subset 
H_k^{(a+\ga,b+\gb)}(\Hd,\CC), \label{LRtoHab} \\
(L^*)^\ga (R^*)^\gb H_k^{(a,b)}(\Hd,\CC) \subset 
H_k^{(a-\ga,b-\gb)}(\Hd,\CC),
\label{LRstartoHab}
\end{align} 
\begin{align}
H(k-a,a) &= (H(k-a,a)\cap \ker R^*) \oplus R H(k-a+1,a-1),
\label{HkerRandR*} \\
H(k-a,a) &= (H(k-a,a)\cap \ker R) \oplus R^* H(k-a-1,a+1),
\label{HkerR*andR} 
\end{align}

Henceforth, all ``orthogonal'' direct sum decompositions will hold for both
the inner products (\ref{firstinpro}) and (\ref{secondinpro}), 
unless stated otherwise. 

We now give some technical results, related to
the following commutativity relations.

\begin{lemma}
\label{LRL*R*commute}
The operators $L$ and $L^*$ commute with $R$ and $R^*$, and we have
\begin{equation}
\label{RR*commute}
R^*R-RR^* = \sum_j \Bigl( z_j{\partial\over\partial z_j}
+w_j{\partial\over\partial w_j}
-\overline{z_j}{\partial\over\partial\overline{z_j}}
-\overline{w_j}{\partial\over\partial\overline{w_j}} \Bigr),
\end{equation}
\begin{equation}
\label{LL*commute}
L^*L - LL^* = \sum_{j} \Bigl( 
z_j {\partial\over\partial z_j} -w_j{\partial\over\partial w_j}
-\overline{z_j}{\partial\over\partial\overline{z_j}} 
+\overline{w_j}{\partial\over\partial\overline{w_j}} \Bigr). 
\end{equation}
\end{lemma}

\begin{proof} 
This is by direct computation. See Section \ref{appendixnumber}.
\end{proof}

Clearly, the right hand side of 
(\ref{RR*commute}) and of (\ref{LL*commute}) maps the monomial $m_a$ 
of (\ref{monomialma}) to a scalar multiple of itself, and so we obtain
\begin{equation}
\label{RR*commf}
R^* R f =  R R^* f +(a-b)f, \qquad f\in \Hom_H(a,b),
\end{equation}
\begin{equation}
\label{LL*commf}
L^* L f =  L L^* f +(a-b)f, \qquad f\in \Hom_K(a,b).
\end{equation}
We can iterate these to obtain formulas which interchange $R$ and $R^*$,
and $L$ and $L^*$.

\begin{lemma}
\label{R*Rbeta}
We have
\begin{equation}
\label{RR*betacommf}
R^* R^\gb f =  R^\gb R^* f +\gb(a-b-\gb+1)R^{\gb-1}f, \quad f\in \Hom_H(a,b), 
\end{equation}
\begin{equation}
\label{LL*betacommf}
L^* L^\gb f =  L^\gb L^* f +\gb(a-b-\gb+1)L^{\gb-1}f, \quad f\in \Hom_K(a,b), 
\end{equation}
which also holds for $\gb=0$ (in the obvious way).
\end{lemma}

\begin{proof} We now prove the first equation, using 
induction on $\gb$. The case $\gb=0$ is trivial,
and the case $\gb=1$ is
(\ref{RR*commf}). 
Suppose the formula holds for $\gb-1\ge0$, 
then $Rf\in\Hom_H(a-1,b+1)$, and so, using (\ref{RR*commf}), we have
\begin{align*} 
R^* R^\gb f
&= R^* R^{\gb-1}(Rf)
= R^{\gb-1} R^* Rf +(\gb-1)\bigl((a-1)-(b+1)-(\gb-1)+1\bigr)R^{\gb-2} Rf \cr
&= R^{\gb-1} (R^* Rf) +(\gb-1)(a-b-\gb)R^{\gb-1}f \cr
& = R^{\gb-1}\bigl( R R^* f +(a-b)f \bigr) +(\gb-1)(a-b-\gb)R^{\gb-1}f \cr
& = R^\gb R^* f +\gb(a-b-\gb+1)R^{\gb-1}f,
\end{align*}
which completes the induction.
The proof of the second equation is very similar.
\end{proof}



Here are the most general formulas, which we will use.

\begin{lemma}
\label{RR*commutegeneral}
For all choices of $\ga$ and $\gb$, we have
\begin{equation}
\label{R*alphaRbeta}
(R^*)^\ga R^\gb f = \sum_{c=0}^\ga {\ga\choose c} (-\gb)_c (b-a+\gb-\ga)_c 
R^{\gb-c} (R^*)^{\ga-c}f, \quad f\in\Hom_H(a,b), 
\end{equation}
\begin{equation}
\label{RalphaR*beta}
R^\ga (R^*)^\gb f = \sum_{c=0}^\ga {\ga\choose c} (-\gb)_c (a-b+\gb-\ga)_c 
(R^*)^{\gb-c} R^{\ga-c}f, \quad f\in\Hom_H(a,b), 
\end{equation}
\begin{equation}
\label{L*alphaLbeta}
(L^*)^\ga L^\gb f = \sum_{c=0}^\ga {\ga\choose c} (-\gb)_c (b-a+\gb-\ga)_c 
L^{\gb-c} (L^*)^{\ga-c}f, \quad f\in\Hom_K(a,b),
\end{equation}
\begin{equation}
\label{LalphaL*beta}
L^\ga (L^*)^\gb f = \sum_{c=0}^\ga {\ga\choose c} (-\gb)_c (a-b+\gb-\ga)_c 
(L^*)^{\gb-c} L^{\ga-c}f, \quad f\in\Hom_K(a,b). 
\end{equation}
Here the terms involving a negative power of an operator 
have a zero coefficient.
\end{lemma}

\begin{proof} We now prove (\ref{R*alphaRbeta}), by
induction on $\ga$, with $\gb$ fixed. 
The case $\ga=0$ is immediate. 
Suppose that (\ref{R*alphaRbeta}) 
holds for $\ga-1\ge0$, then 
$$ (R^*)^\ga R^\gb f = \sum_{c'=0}^{\ga-1} 
{\ga-1\choose c'} (-\gb)_{c'} (b-a+\gb-\ga+1)_{c'} R^* R^{\gb-c'} (R^*)^{\ga-1-c'}f. $$
Since $g=(R^*)^{\ga-1-c'}f\in H(a+\ga-1-c',b-(\ga-1-c'))$, 
(\ref{RR*betacommf}) of Lemma \ref{R*Rbeta} gives
\begin{align*}
R^* R^{\gb-c'} g
&= R^{\gb-c'}R^*g+(\gb-c')(a+\ga-1-c'-(b-(\ga-1-c'))-(\gb-c')+1)R^{\gb-c'-1}g \cr
& = R^{\gb-c'}(R^*)^{\ga-c'}f
+(-\gb+c')(b-a-2\ga+1+c'+\gb)R^{\gb-c'-1} (R^*)^{\ga-1-c'}f,
\end{align*}
and we have
\begin{align*}
(R^*)^\ga R^\gb f = \sum_{c'=0}^{\ga-1} &
{\ga-1\choose c'} (-\gb)_{c'} (b-a+\gb-\ga)_{c'} \Bigl\{R^{\gb-c'}(R^*)^{\ga-c'}f \cr
&+(-\gb+c')(b-a-2\ga+1+c'+\gb)R^{\gb-c'-1} (R^*)^{\ga-1-c'}f\Bigr\}.
\end{align*}
The coefficient of $R^{\gb-c}(R^*)^{\ga-c}f$ in the above formula for 
$(R^*)^\ga R^\gb f$ is
\begin{align*}
& {\ga-1\choose c} (-\gb)_{c} (b-a+\gb-\ga+1)_{c} \cr
& +{\ga-1\choose c-1} (-\gb)_{c-1} (b-a+\gb-\ga+1)_{c-1}
(-\gb+c-1)(b-a-2\ga+1+(c-1)+\gb) \cr
&\qquad = {(\ga-1)!\over c! (\ga-c)!} (-\gb)_c (b-a+\gb-\ga+1)_{c-1}
\bigl\{ \quad
\bigr\}
= {\ga!\over c!(\ga-c)!}(-\gb)_c (b-a+\gb-\ga)_c, 
\end{align*}
where
$$
\bigl\{\quad \bigr\} = (\ga-c)(b-a+\gb-\ga+c)+c(b-a-2\ga+c+\gb)
= \ga(b-a+\gb-\ga). $$
Thus we obtain the desired formula for $(R^*)^\ga R^\gb f$, 
which completes the induction.

The other formulas follow in a similar fashion.
\end{proof}

\section{$U(\Hd)$-invariant subspaces}

It follows from Lemma \ref{RR*commutegeneral} that 
$R$ and $R^*$ inverses of each other in some sense.


\begin{lemma} 
\label{RRstarcancelcor}
For $\ga\le\gb$, and $\gb>\ga+a-b$ or $\gb\le a-b$, 
we have
\begin{align}
\label{RRstarcancelI}
(R^*)^\ga R^\gb (\ker R^*\cap\Hom_H(a,b))
&= R^{\gb-\ga} (\ker R^*\cap\Hom_H(a,b)), \\
\label{RRstarcancelII}
R^\ga (R^*)^\gb (\ker R\cap\Hom_H(b,a))
&= (R^*)^{\gb-\ga} (\ker R\cap\Hom_H(b,a)), 
\end{align}
otherwise
\begin{equation}
\label{RRstarcancelIII}
(R^*)^\ga R^\gb (\ker R^*\cap\Hom_H(a,b))=0, \qquad
R^\ga (R^*)^\gb (\ker R\cap\Hom_H(b,a))=0.
\end{equation}
\end{lemma}

\begin{proof} For $f\in\ker R^*\cap\Hom_H(a,b)$, $R^*f=0$, 
and so (\ref{R*alphaRbeta}) reduces to
$$ (R^*)^\ga R^\gb f = (-\gb)_\ga(b-a+\gb-\ga)_\ga R^{\gb-\ga}f, $$
The condition for the constant above to be nonzero is $\ga\le\gb$, and
the $\ga$ factors
$$ b-a+\gb-\ga,\quad b-a+\gb-\ga+1, \quad \ldots \quad b-a+\gb-1 $$
of $(b-a+\gb-\ga)_\ga$ are not zero,  
i.e., $b-a+\gb-\ga>0$ or $b-a+\gb-1<0$.
This gives the first case, with the other following by 
the same argument.
\end{proof}

By repeated applications of (\ref{HomHkerRandR*}) and (\ref{HomHkerR*andR}),
we obtain the following.

\begin{lemma}
\label{2orthoexps}
We have the orthogonal direct sums
\begin{align} 
\label{orthdirectsumI}
\Hom_H(k-b,b) &= \bigoplus_{j=0}^{b}
R^{b-j} \bigr(\ker R^*\cap \Hom_H(k-j,j)\bigl), \quad b\le k-b, \\
\label{orthdirectsumII}
\Hom_H(k-b,b) &= \bigoplus_{j=0}^{k-b}
(R^*)^{k-b-j} \bigr(\ker R\cap \Hom_H(j,k-j)\bigl),
\quad k-b\le b.
\end{align} 
Further
\begin{enumerate}[\rm(i)]
\item For $a>b$, $R$ is 1-1 on $\Hom_H(a,b)$.
\item For $a\le b$, $R$ maps $\Hom_H(a,b)$ onto $\Hom_H(a-1,b+1)$.
\end{enumerate}
\end{lemma}

\begin{proof}
Apply 
(\ref{HomHkerRandR*}) and (\ref{HomHkerR*andR}) repeatedly.
For $b\le k-b$, we have
\begin{align*}
& \Hom_H(k-b,b) 
= ( \ker R^* \cap \Hom_H(k-b,b))\oplus R\Hom_H(k-b+1,b-1) \cr
& \quad = (\ker R^*\cap\Hom_H(k-b,b)) \cr
& \qquad\qquad 
\oplus R\{(\ker R^*\cap\Hom_H(k-b+1,b-1)) \oplus R \Hom_H(k-b+2,b-2)\} \cr
& \quad = (\ker R^*\cap \Hom_H(k-b,b)) \oplus R(\ker R^*\cap\Hom_H(k-b+1,b-1))\cr
&\qquad\qquad\oplus R^2(\ker R^*\cap\Hom_H(k-b+2,b-2))\oplus \cdots
\oplus R^b(\ker R^*\cap\Hom_H(k,0)).
\end{align*}
Similarly, for $k-b\le b$, we have
\begin{align*}
\Hom_H(k-b,b) 
&= (\ker R\cap \Hom_H(k-b,b)) 
\oplus R^*(\ker R\cap \Hom_H(k-b-1,b+1)) \cr
& \qquad \oplus (R^*)^2(\ker R\cap \Hom_H(k-b-2,b+2))\oplus\cdots \cr
& \qquad\qquad \cdots\oplus (R^*)^{k-b} (\ker R\cap \Hom_H(0,k)),
\end{align*}
which gives (\ref{orthdirectsumI}) and (\ref{orthdirectsumII}).

To show the injectivity of (i), 
it suffices to show that for $k-b>b$, i.e., $b+1\le k-b$,
$R$ is 1-1 on each summand in (\ref{orthdirectsumI}), i.e., 
$$ R^{b-j} \bigr(\ker R^*\cap \Hom_H(k-j,j)\bigl), \qquad 0\le j\le b. $$
This follows from
$$ R^*R R^{b-j} \bigr(\ker R^*\cap \Hom_H(k-j,j)\bigl)
= R^{b-j} \bigr(\ker R^*\cap \Hom_H(k-j,j)\bigl), $$
which is (\ref{RRstarcancelI}) of Lemma \ref{RRstarcancelcor}
for $\ga=1$, $\gb=b-j+1$, $a=k-j$, $b=j$, since
$$ b+1\le k-b, \quad j\le b 
\Implies b+j+1\le k, \quad\hbox{i.e., the condition $\gb\le a-b$ holds}. $$

For $a\le b$, a similar argument shows that $R^*$ is 1-1 on $\Hom_H(a-1,b+1)$.
Here, when $a=0$, $\Hom_H(a-1,b+1)=0$. Therefore, 
$(R^*|_{\Hom_H(a-1,b+1)})^* = R|_{\Hom_H(a,b)}$ is onto, and we have (ii). 
\end{proof}

The following result says that the $j$-terms in the expansions
of Lemma \ref{2orthoexps} 
(only one of which holds for a given $b$, $2b\ne k$) are in fact equal. 
This then allows for a single expansion for both cases
(Lemma \ref{HomkRdecomp}).

\begin{lemma} 
\label{RowMovementsLemma}
(Row movements) For $0\le j\le{k\over2}$, we have
\begin{align}
R^{k-2j+1} (\ker R^*\cap \Hom_H(k-j,j))&=0,
\label{Rpowerzero} \\
(R^*)^{k-2j+1} (\ker R\cap \Hom_H(j,k-j))&=0,
\label{Rstarpowerzero}
\end{align}
and for $j\le a\le k-j$, we have
\begin{equation}
\label{RtoRstargeneral}
R^{a-j} (\ker R^*\cap \Hom_H(k-j,j))
= (R^*)^{k-a-j}(\ker R\cap \Hom_H(j,k-j)).
\end{equation}
Furthermore, all of the results above hold with 
$\Hom_H(p,q)$ replaced by $H(p,q)$.
\end{lemma}

\begin{proof} 
The equations (\ref{Rpowerzero}) and (\ref{Rstarpowerzero}) follow
from Lemma \ref{RRstarcancelcor} for the choice $\ga=0$, $\gb=k-2j+1$, $a=k-j$, 
$b=j$.
These give the inclusions
$$ R^{k-2j} (\ker R^*\cap \Hom_H(k-j,j)) \subset \ker R\cap \Hom_H(j,k-j), $$
$$ (R^*)^{k-2j} (\ker R\cap \Hom_H(j,k-j)) \subset \ker R^*\cap \Hom_H(k-j,j). $$

We now prove the cases $a=k-j$ and $a=j$ in (\ref{RtoRstargeneral}), 
i.e.,
\begin{align}
R^{k-2j} (\ker R^*\cap \Hom_H(k-j,j)) &= \ker R\cap \Hom_H(j,k-j), 
\label{kerRstartokerR} \\
(R^*)^{k-2j} (\ker R\cap \Hom_H(j,k-j)) &= \ker R^*\cap \Hom_H(k-j,j),
\label{kerRtokerRstar}
\end{align}
Taking $\ga=\gb=k-2j$ in Lemma \ref{RRstarcancelcor} gives
\begin{equation}
\label{R*Rfix}
(R^*)^{k-2j} R^{k-2j} (\ker R^*\cap \Hom_H(k-j,j)) 
= \ker R^*\cap \Hom_H(k-j,j),
\end{equation}
$$ R^{k-2j} (R^*)^{k-2j} (\ker R\cap \Hom_H(j,k-j))
= \ker R\cap \Hom_H(j,k-j). $$
Thus we have
\begin{align*}
\ker R\cap \Hom_H(j,k-j)
& = R^{k-2j} (R^*)^{k-2j} (\ker R\cap \Hom_H(j,k-j)) \cr
& \subset R^{k-2j} (\ker R^*\cap \Hom_H(k-j,j)) \cr
& \subset \ker R\cap \Hom_H(j,k-j),
\end{align*}
which gives
(\ref{kerRstartokerR}), with (\ref{kerRtokerRstar}) following
similarly. Now (\ref{R*Rfix}) and (\ref{kerRstartokerR}) give
\begin{align}
\label{halfofit}
R^{a-j}(\ker R^*\cap \Hom_H(k-j,j)) 
& = R^{a-j} (R^*)^{k-2j}R^{k-2j}(\ker R^*\cap \Hom_H(k-j,j)) \cr
& = R^{a-j} (R^*)^{k-2j}(\ker R\cap \Hom_H(j,k-j)).
\end{align}
Taking $\ga=a-j$, $\gb=k-2j$ in Lemma \ref{RRstarcancelcor} gives
$$ R^{a-j} (R^*)^{k-2j}(\ker R\cap \Hom_H(j,k-j)) 
= (R^*)^{k-j-a} (\ker R\cap \Hom_H(j,k-j)), $$
which together with (\ref{halfofit}) gives (\ref{RtoRstargeneral}).
\end{proof}

We now present a key technical result.

\begin{lemma}
\label{HomkRdecomp}
We have the orthogonal direct sum decompositions
\begin{align}
\label{HomkUorthogdecomp}
\Hom_k(\Hd,\CC) &= \bigoplus_{0\le j\le {k\over2}}
\bigoplus_{j\le b\le k-j} \Hom_H(k-b,b)_{k-2j},  \\
\label{HkUorthogdecomp}
\Harm_k(\Hd,\CC) &= \bigoplus_{0\le j\le {k\over2}}
\bigoplus_{j\le b\le k-j} H(k-b,b)_{k-2j},
\end{align}
into $U(\Hd)$-invariant subspaces,
where
\begin{align}
\label{Irreducibledefn}
\Hom_H(k-b,b)_{k-2j} 
&:=  R^{b-j} \bigr(\ker R^*\cap \Hom_H(k-j,j)\bigl) \cr
&\ = (R^*)^{k-b-j} \bigr(\ker R\cap \Hom_H(j,k-j)\bigl) \cr
&\ \subset 
\Hom_H(k-b,b), \\
\label{Irreducibledefn2}
H(k-b,b)_{k-2j} 
&:= R^{b-j} \bigr(\ker R^*\cap H(k-j,j)\bigl) \cr
&\ = (R^*)^{k-b-j} \bigr(\ker R\cap H(j,k-j)\bigl) \cr
&\ \subset H(k-b,b).
\end{align}
\end{lemma}

\begin{proof} 
Since the Laplacian operator $\gD$ commutes with $R$ and $R^*$ (Lemma \ref{LapRandLcommute}),  
the decomposition
(\ref{HkUorthogdecomp}) follows from (\ref{HomkUorthogdecomp}) 
by taking the intersection with the harmonic polynomials. 
We therefore consider just the decomposition of $\Hom_H(k-b,b)$.

Since $\Hom_H(k-b,b)$ and $H(k-b,b)$ are invariant under $U(\CC^{2d})$,
they are invariant under $U(\Hd)$. 
Moreover, the action of $U(\Hd)$ commutes with $R$ and $R^*$ 
(Lemma \ref{RandU(Hd)commute}), and so the summands
in (\ref{HomkUorthogdecomp}) and (\ref{HkUorthogdecomp}) are $U(\Hd)$-invariant.
As an indicative calculation, let $U\in U(\Hd)$, then
$$ f\in\ker R^* \Iff 
U\cdot(R^*f)=0
\Iff R^*(U\cdot f)=0 \Iff U\cdot f\in\ker R^*, $$
and so
\begin{align*}
U\cdot H(k-b,b)_{k-2j} 
&= U\cdot R^{b-j} \bigr(\ker R^*\cap H(k-j,j)\bigl)
= R^{b-j} \bigr(U\cdot\ker R^*\cap U\cdot H(k-j,j)\bigl) \cr
&= R^{b-j} \bigr(\ker R^*\cap H(k-j,j)\bigl) 
=H(k-b,b)_{k-2j}.
\end{align*}

Since 
$$ j\le b\le k-j \Iff j\le b,\quad j\le k-b \Iff
j\le\min\{b,k-b\}, $$
the direct sum (\ref{HomkUorthogdecomp}) can be rearranged as
$$ \Hom_k(\Hd,\CC) 
= \bigoplus_{0\le b\le k} \bigoplus_{j=0}^{\min\{b,k-b\}}
\Hom_H(k-b,b)_{k-2j}. $$
By Lemma \ref{RowMovementsLemma},
$$ R^{b-j} \bigr(\ker R^*\cap \Hom_H(k-j,j)\bigl)
= (R^*)^{k-b-j} \bigr(\ker R\cap \Hom_H(j,k-j)\bigl),$$
which gives the equalities in (\ref{Irreducibledefn}) and (\ref{Irreducibledefn2}),
and so   
it suffices to show the orthogonal direct sums
\begin{align*}
\Hom_H(k-b,b) &= \bigoplus_{j=0}^{b}
R^{b-j} \bigr(\ker R^*\cap \Hom_H(k-j,j)\bigl), \quad b\le k-b, \cr
\Hom_H(k-b,b) &= \bigoplus_{j=0}^{k-b}
(R^*)^{k-b-j} \bigr(\ker R\cap \Hom_H(j,k-j)\bigl),
\quad k-b\le b.
\end{align*}
These are given by Lemma \ref{2orthoexps}.
\end{proof}


To calculate the dimensions of various irreducibles,
we will need the following.

\begin{lemma}
\label{H(k-j,j)capKerR*dim} 
Let $0\le j\le {k\over2}$. For $d=1$, we have the following dimensions
$$  \dim \bigr(\ker R^*\cap \Hom_H(k-j,j)\bigl)=k-2j+1, \quad
 \dim \bigr(\ker R^*\cap H(k-j,j)\bigl)
=\begin{cases}
1, & j=0; \cr
0, & j\ne 0.
\end{cases} $$
For $d\ge 2$, we have
\begin{equation}
\label{HomH(k-b,b)k-2jdim}
\dim \bigr(\ker R^*\cap \Hom_H(k-j,j)\bigl)
= (k-2j+1) {(k-j+2d-1)!(j+2d-2)!\over(k-j+1)!j!(2d-1)!(2d-2)!},
\end{equation}
\begin{equation}
\label{H(k-b,b)k-2jdim}
\dim \bigr(\ker R^*\cap H(k-j,j)\bigl)
= (k-2j+1) (k+2d-1) {(k-j+2d-2)!(j+2d-3)!\over(k-j+1)!j!(2d-1)!(2d-3)!}.
\end{equation}
\end{lemma}

\begin{proof}
From (\ref{HomHkerRandR*}), and the fact $R$ is 1-1 on $\Hom_H(k-j+1,j-1)$
(Lemma \ref{2orthoexps}), we have
\begin{align*}
\dim(\Hom_H(k-j,j)\cap \ker R^*) 
&= \dim\left(\Hom_H(k-j,j)\right)-\dim\left(R \Hom_H(k-j+1,j-1)\right) \cr
&=\dim\left(\Hom_H(k-j,j)\right)-\dim\left(\Hom_H(k-j+1,j-1)\right).
\end{align*}
Using this and (\ref{HomHpqdim}), 
with $p=k-j$, $q=j$,
we calculate (\ref{HomH(k-b,b)k-2jdim}) for $d\ge1$
\begin{align*}
& \dim( \ker R^* \cap \Hom_H(k-j,j)) \cr
& \qquad = {1\over(2d-1)!^2} \Bigl\{ {(p+2d-1)!(q+2d-1)!\over p! q!}
- {(p+2d)!(q+2d-2)!\over (p+1)! (q-1)!} \Bigr\} \cr
& \qquad = {(p+2d-1)!(q+2d-2)!\over (p+1)!q!(2d-1)!^2}
\{ (q+2d-1)(p+1)-(p+2d)q\} \cr
& \qquad = {(p+2d-1)!(q+2d-2)!\over (p+1)!q!(2d-1)!^2}
(p-q+1)(2d-1). 
\end{align*}
The other formula follows in a similar way, from 
\begin{align*}
\dim(\ker R^* & \cap H(k-j,j)) = \dim(H(k-j,j))-\dim(H(k-j+1,j-1)) \cr
&= {k-j+2d-1\choose 2d-1}{j+2d-1\choose 2d-1}
- {k-j+2d\choose 2d-1}{j+2d-2\choose 2d-1},
\end{align*}
with the $d=1$ case calculated separately,
which completes the proof.
\end{proof}

\begin{example} For $j=0$, (\ref{H(k-b,b)k-2jdim}) reduces to
$$ \dim \bigr(\ker R^*\cap H(k,0)\bigl)
= {(k+2d-1)!\over k! (2d-1)!}
= \dim\bigl(H(k,0)\bigl), $$
so that $\ker R^*\cap H(k,0)$ is the holomorphic polynomials, i.e.,
$$ \ker R^*\cap H(k,0) = H(k,0) 
=\bigoplus_{|\ga+\gb|=k} \spam\{ z^\ga w^\gb \} \quad
\hbox{(orthogonal direct sum)}. $$
We also observe, from the proof of Lemma \ref{H(k-j,j)capKerR*dim}, that
for $0\le j\le {k\over 2}$, we have
$$ \ker R^*\cap H(k-j,j)
=\{f\in H(k-j,j) : f\perp \bigoplus_{0\le a<j} H(k-a,a)\}, $$
so the by applying Gram-Schmidt to a spanning sequence ordered so that its
elements are in $H(k,0),H(k-1,1), \ldots H(k-j,j)$, successively, 
the corresponding elements are an orthonormal basis for
$\ker R^*\cap H(k,0), \ldots, \ker R^*\cap H(k-j,j)$.
\end{example}

The results of this section can found or deduced from those of
the work of \cite{Ghent14}. Their variables $z_1,\ldots, z_{2p}$ 
correspond to ours via
$$ z_1,\ldots,z_{2p}  \quad\longleftrightarrow\quad
z_1,w_1,\ldots,z_p,w_p, $$
and they define operators
$$ \gep=R^*, \quad \gep^\dagger= R. $$
The decomposition (\ref{HkUorthogdecomp}) for $H(k-b,b)$ of Lemma \ref{HomkRdecomp} 
is presented as the two cases in Lemma \ref{2orthoexps} (Theorems 5.1 and 5.2 of \S 5 \cite{Ghent14}).

\section{Visualising the action of $L$ and $R$ on subspaces}

The action of $L$ and $R$ given in Lemma \ref{hardtounderstandlemma} leads to
the following orthogonal direct sums.

\begin{lemma}
\label{Homabdecomplemma}
We have the orthogonal direct sum decomposition into $(k+1)^2$ subspaces
\begin{equation}
\label{Homabdecomp}
\Hom_k(\Hd,\CC) = \bigoplus_{0\le a,b\le k} \Hom_k^{(a,b)}(\Hd).
\end{equation}
For $0\le a,b\le k$, let $m_a=m_a^{(k)}:=\min\{a,k-a\}$, $m_b=m_b^{(k)}:=\min\{b,k-b\}$,
and
\begin{equation}
\label{mMdefn}
m=m_{a,b}^{(k)}:=\min\{m_a,m_b\}, \quad 
 M=M_{a,b}^{(k)}:=\max\{m_a,m_b\}, \quad 
  c:=\min\{a,b\}.
\end{equation}
Then
\begin{equation}
\label{Homabstructure}
\Hom_k^{(a,b)}(\Hd) =\spam\{ z^{\ga_1}w^{\ga_2}\overline{z}^{\ga_3}
\overline{w}^{\ga_4}\}_{(\ga_1,\ga_2,\ga_3,\ga_4)\in A}, 
\end{equation}
where $A=A_{a,b,k}$ is given by 
$$ A:=\{\ga:|\ga_1|=k-(a+b-c)-j,|\ga_2|=a-c+j, 
|\ga_3|=c-j,|\ga_4|=b-c+j,0\le j\le m\}. $$
In particular, we have
\begin{align}
\label{Homabdim}
\dim\bigl(\Hom_k^{(a,b)}(\Hd,\CC)\bigr)
&= \sum_{j=0}^m {k-M-j+d-1\choose d-1} {j+d-1\choose d-1} \cr
& \qquad\times {m-j+d-1\choose d-1} {M-m+j+d-1 \choose d-1}.
\end{align}
\end{lemma}

\begin{proof} To establish (\ref{Homabdecomp}),
it suffices to show that the direct sums
$$ \Hom_k(\Hd,\CC) = \bigoplus_{a+b=k} \Hom_H(a,b)=
\bigoplus_{p+q=k} \Hom_K(p,q), $$
are orthogonal, which follows immediately since the monomials are orthogonal 
for (\ref{secondinpro}).


We now consider (\ref{Homabstructure}).
Let $f=z^{\ga_1}w^{\ga_2}\overline{z}^{\ga_3}\overline{w}^{\ga_4}
\in\Hom_k^{(a,b)}(\Hd,\CC)$, i.e.,
$$ 
|\ga_1|+|\ga_4|=k-a, \quad |\ga_2|+|\ga_3|=a
, \qquad
|\ga_1|+|\ga_2|=k-b, \quad |\ga_3|+|\ga_4|=b
. $$
The above equations imply that once an allowable value of 
$|\ga_1|,|\ga_2|,|\ga_3|,|\ga_4|$ is specified,
then the others are uniquely determined. The allowable
values are determined by an equation where the right hand
side is $m=\min\{a,k-a,b,k-b\}$, and so we must 
treat (four) cases. First consider the case $a\le b$, i.e., 
$m=a,k-b$, for which we have
$$ 0\le j=|\ga_2|\le m\in\{a,k-b\}, \quad
|\ga_3|=a-j, \quad |\ga_1|=k-b-j, \quad 
|\ga_4| 
= j+b-a, $$
and hence
\begin{align*}
\Hom_k^{(a,b)}(\Hd,\CC)
& =\spam\{ z^{\ga_1}w^{\ga_2}\overline{z}^{\ga_3}\overline{w}^{\ga_4}: 
|\ga_1|=k-b-j,|\ga_2|=j, |\ga_3|=a-j, \cr
&\qquad\qquad\qquad\qquad\qquad\qquad
 |\ga_4|= j+b-a,0\le j\le m \}.
\end{align*}
The corresponding condition for the case $a\ge b$ is
$$ 0\le j=|\ga_4|\le m\in\{b,k-a\}, \quad
|\ga_1|=k-a-j,\quad |\ga_2|=a-b+j, \quad |\ga_3|=b-j, $$
and so we obtain (\ref{Homabstructure}).

It follows from symmetries of the space $\Hom_k^{(a,b)}(\Hd,\CC)$,
or by direct calculation, that its dimension, the cardinality of $A$,
depends only on $m,M$ (and $k$). Therefore, by 
the case $m=a=c$, $M=b$, and the fact $\ga_1,\ldots,\ga_4\in\ZZ_+^d$, 
we obtain (\ref{Homabdim}).
\end{proof}

\begin{example}
For $d=1$, we have
\begin{equation}
\label{Homkabonedim}
\dim(\Hom_k^{(a,b)}(\HH))=m+1, \qquad m:=\min\{a,k-a,b,k-b\}.
\end{equation}
\end{example}

\begin{example}
For $k=1$, we have
\begin{align*} 
\Hom_1^{(0,0)}(\Hd)&=\spam\{z_1,\ldots,z_d\}, \qquad
\Hom_1^{(0,1)}(\Hd) =\spam\{\overline{w_1},\ldots,\overline{w_d}\}, \cr
\Hom_1^{(1,0)}(\Hd)& =\spam\{w_1,\ldots,w_d\}, \qquad
\Hom_1^{(1,1)}(\Hd)=\spam\{\overline{z_1},\ldots,\overline{z_d}\}.
\end{align*}
\end{example}

The corresponding result for harmonic polynomials is the following.

\begin{lemma}
\label{Habdecomplemma}
We have the orthogonal direct sum decomposition into $(k+1)^2$ subspaces
\begin{equation}
\label{Habdecomp}
\Harm_k(\Hd,\CC) = \bigoplus_{0\le a,b\le k} H_k^{(a,b)}(\Hd).
\end{equation}
where
\begin{equation}
\label{dimHkabdiff}
\dim(H_k^{(a,b)}(\Hd))
= \dim(\Hom_k^{(a,b)}(\Hd))
-\dim(\Hom_{k-2}^{(a-1,b-1)}(\Hd)).
\end{equation}
In particular, for $d=1$, we have
\begin{equation}
\label{Hkabonedim}
\dim(H_k^{(a,b)}(\HH))=1, \qquad 0\le a,b\le k,
\end{equation}
and for $d>1$, with $m$ and $M$ given by (\ref{mMdefn}),
we have
\begin{equation}
\label{Hkabdim}
\dim(H_k^{(a,b)}(\Hd,\CC)) = F(k,m,M,d), \qquad 0\le a,b\le k,
\end{equation}
where
\begin{align}
F(k,m,M,d)
&:= \sum_{j=0}^{m}
{j+d-1\choose d-1}{M-m+j+d-1 \choose d-1} \label{FkmMddef} \\
&\quad\times { (m-j+1)_{d-2}(k-M-j+1)_{d-2} \over (d-1)!(d-2)! }
 (k-M+m-2j+d- 1).  \nonumber
\end{align}
\end{lemma}

\begin{proof}
The dimension formula (\ref{dimHkabdiff}) follows
since $H_k^{(a,b)}(\Hd)$ is the kernel of $\gD$ restricted
to $\Hom_k^{(a,b)}(\Hd)$, and
$\Delta \Hom_k^{(a,b)}(\Hd) = \Hom_{k-2}^{(a-1,b-1)}(\Hd)$.

To develop an explicit formula from (\ref{Homabdim}), we need
to take account of the case when $a\in\{0,k\}$ or $b\in\{0,k\}$,
i.e., $m=0$,
in which case
$\dim(H_{k-2}^{(a-1,b-1)}(\Hd))=0$.

For $d=1$, (\ref{Hkabonedim}) holds for $m=0$,
and otherwise (\ref{Homkabonedim}) gives
\begin{align*}
\dim(\Hom_{k-2}^{(a-1,b-1)}(\HH))
&= \min\{a-1,k-2-(b-1),b-1,k-2-(a-1)\}+1 \cr
&= \min\{a,k-b,b,k-a\}
= \dim(\Hom_k^{(a,b)}(\HH))-1.
\end{align*}
A similar (elementary) calculation gives
$$ m_{a-1,b-1}^{(k-2)}= m_{a,b}^{(k)}-1, \qquad
M_{a-1,b-1}^{(k-2)}= M_{a,b}^{(k)}-1. $$
Hence, with $m:=m_{a,b}^{(k)}$ and $M:=M_{a,b}^{(k)}$,
(\ref{Homabdim}) gives
\begin{align*}
\dim\bigl(\Hom_{k-2}^{(a-1,b-1)}(\Hd,\CC)\bigr)
&= \sum_{j=0}^{m-1} {(k-2)-(M-1)-j+d-1\choose d-1} {j+d-1\choose d-1} \cr
&\qquad {(m-1)-j+d-1\choose d-1} {(M-1)-(m-1)+j+d-1 \choose d-1},
\end{align*}
where for $d>1$ the ``$j=m$'' term above is zero by virtue of
${(m-1)-m+d-1\choose d-1}=0$.
Thus (\ref{Homabdim}) gives
$$ \dim(H_k^{(a,b)}(\Hd,\CC)) = \sum_{j=0}^{m}
{j+d-1\choose d-1}{M-m+j+d-1 \choose d-1}\Bigl\{\quad\Bigr\}, $$
where
\begin{align*}
\Bigl\{\quad\Bigr\}
&= {k-M-j+d-1\choose d-1}{m-j+d-1\choose d-1}  \cr
&\qquad - {k-1-M-j+d-1\choose d-1} {m-1-j+d-1\choose d-1}, 
\end{align*}
which simplifies to give (\ref{FkmMddef}).
\end{proof}

\begin{example}
For $k=2$, we have three cases $(m,M)=(0,0),(0,1),(1,1)$, giving
\begin{align*}
\dim(H_2^{(a,b)}(\Hd)) &=F(2,0,0,d) = \hbox{${1\over 2}d(d+1)$}, \qquad
(a,b)\in\{(0,0),(2,0),(0,2),(2,2)\}, \cr
\dim(H_2^{(a,b)}(\Hd)) &=F(2,0,1,d) = d^2, \, \quad\qquad\qquad
(a,b)\in\{(1,0),(0,1),(1,2),(2,1)\}, \cr
\dim(H_2^{(a,b)}(\Hd)) &=F(2,1,1,d) = 2 d^2-1, \quad\qquad
(a,b)\in\{(1,1)\}.
\end{align*}
These formulas also hold for $d=1$. We also have
$$ \dim(H_k^{(a,b)}(\Hd)) 
= {k+d-1\choose d-1}
=\dim(\Hom_k^{(a,b)}(\Hd)), 
\quad
(a,b)\in\{(0,0),(k,0),(0,k),(k,k)\}, $$
with the corresponding spaces given by $A=A_{a,b,k}$ of Lemma \ref{Homabdecomplemma}, e.g.,
$$ H_k^{(0,0)}(\Hd) =\bigoplus_{|\ga|=k} \spam\{z^\ga\},
\quad\hbox{(orthogonal direct sum)}. $$
\end{example}


The results of this section for $\Harm_k(\Hd,\CC)$ can be summarised
as follows.

\begin{schematic}
\label{schematicHarmk}
The orthogonal decomposition (\ref{Habdecomp}) of $\Harm_k(\Hd,\CC)$
can be displayed as a square matrix/array/table
\begin{equation}
\label{schematicpict}
\mat{&\cr K(k,0) \cr K(k-1,1) \cr \vdots\cr K(0,k) }
 \mat{ \mat{H(k,0)\quad\ H(k-1,1)\ \, \cdots \quad H(0,k)} \cr
\bmat{ H_k^{(0,0)}(\Hd) & H_k^{(0,1)}(\Hd) & \cdots & H_k^{(0,k)}(\Hd) \cr
H_k^{(1,0)}(\Hd) & H_k^{(1,1)}(\Hd) & \cdots & H_k^{(1,k)}(\Hd) \cr
\vdots & \vdots & & \vdots \cr
H_k^{(k,0)}(\Hd) & H_k^{(k,1)}(\Hd) & \cdots & H_k^{(k,k)}(\Hd) } }
\qquad \mat{ &&\cr &\moveLb&\cr &&\cr \moveRb&&\quad\moveR\cr&&\cr &\moveL &}
\end{equation}
where the rows are indexed by $K(k-a,a)$ and the columns by $H(k-b,b)$.
Here
\begin{itemize}
\item $\Harm_k(\Hd,\CC)$ is the orthogonal direct sum of the $(k+1)^2$
entries of the matrix.
\item The subspace in the $K(k-a,a)$ row and $H(k-b,b)$ column is
$$ H_k^{(a,b)}(\Hd)=K(k-a,a)\cap H(k-b,b). $$
\item $H(k-a,b)$ is the orthogonal direct sum of the entries of its column.
\item $K(k-a,a)$ is the orthogonal direct sum of the entries of its row.
\item Multiplication by $L$ moves down the columns, and $L^*$ up them.
\item Multiplication by $R$ moves right along the rows, and $R^*$ to the left of them.
Further, multiplication by $R$ is 1-1 on the left hand side (half) of the table, and is
onto on the right hand side.
\item Left multiplication by $\HH^*$ (and more generally by $U(\Hd)$) moves within the columns.
\item Right multiplication by $\HH^*$ moves within the rows.
\item Multiplication by $L_\ga$, $L_{\overline{\ga}}$, $R_\ga$ and
$R_{\overline{\ga}}$ does not move the entries of the matrix.
\end{itemize}
\end{schematic}

\noindent
There is a similar ``square'' for the decomposition (\ref{Homabdecomp})
of $\Hom_k(\Hd,\CC)$.

There are also ``symmetries'' which permute the entries of the square,
such as
$$ \overline{H_k^{(a,b)}(\Hd)}=H_k^{(k-a,k-b)}(\Hd). $$
It is convenient to imagine zero subspaces outside of the square matrix, which
then encodes properties such as
$$ R^{k+1}\Harm_k(\Hd,\CC) =0, \qquad R^* H(k,0)=0, \qquad L^2 K(1,k-1)=0. $$

\section{The one variable case}

We now consider the $d=1$ case in great detail.
Though this is somewhat degenerate, and usually not considered,
it provides motivation and illustrates the main features
of the general case.

By Lemma \ref{Habdecomplemma}, $\dim(\Harm_k(\HH,\CC))=(k+1)^2$,
and the square matrix/table (\ref{schematicpict}) for $\Harm_k(\HH,\CC)$
consists of the one-dimensional
subspaces $\{H_k^{(a,b)}(\HH)\}_{0\le a,b\le k}$.
Since the polynomial $z^k$ is holomorphic, it is harmonic,
and so
$$ H_k^{(0,0)}(\HH)=\spam_\CC\{z^k\}. $$
We consider what are the other harmonic monomials in $\Hom_k(\HH,\CC)$.

\begin{example}
The monomial $z^{\ga_1}w^{\ga_2}\overline{z}^{\ga_3}\overline{w}^{\ga_4}$,
$|\ga|=k$,
is harmonic if and only if
$$ \left({\partial^2\over\partial\overline{z}\partial z}
+{\partial^2\over\partial\overline{w}\partial w}\right)
z^{\ga_1}w^{\ga_2}\overline{z}^{\ga_3}\overline{w}^{\ga_4}
= \ga_1\ga_3 z^{\ga_1-1}w^{\ga_2}\overline{z}^{\ga_3-1}\overline{w}^{\ga_4}
+ \ga_2\ga_4 z^{\ga_1}w^{\ga_2-1}\overline{z}^{\ga_3}\overline{w}^{\ga_4-1}
= 0, $$
i.e., $\ga_1\ga_3=\ga_2\ga_4=0$. This gives $4k$ harmonic monomials of degree $k$.
\end{example}

The harmonic monomials of degree $k$ lie on the four ``edges'' (of length $k+1$) of
the square table, 
which are given by
$\ga_2=\ga_3=0$ (top edge),
$\ga_3=\ga_4=0$ (left edge),
$\ga_1=\ga_2=0$ (right edge),
$\ga_1=\ga_4=0$ (bottom edge),
with
$$ H_k^{(\ga_2+\ga_3,\ga_3+\ga_4)}(\HH)=\spam_\CC\{z^{\ga_1}w^{\ga_2}
\overline{z}^{\ga_3}\overline{w}^{\ga_4} \},
\qquad |\ga|=k, \quad \ga_1\ga_3=\ga_2\ga_4=0. $$
An elementary calculation shows that
\begin{equation}
\label{onedimanchor}
L^k R^k (z^k) = (-1)^k k!^2 \overline{z}^k,
\end{equation}
and so it follows from the Schematic \ref{schematicHarmk},
that by applying
$L$ and $R$ to the upper left corner $z^k \in H_k^{(0,0)}(\HH)$,
that we can ``fill out the table'' with nonzero polynomials
in the subspaces, which (in this case) gives a basis for them,
e.g., for $k=2$, we have
\begin{equation}
\label{univariatesquaresimple}
\mat{ z^2 & \hskip-0.2truecm \moveR\cr \hskip-0.4truecm\moveL &}
\qquad \mat{&\cr K(2,0)\cr K(1,1)\cr K(0,2)}
 \mat{ \mat{H(2,0)&H(1,1)&H(0,2)} \cr \bmat{ z^2 & 2z\overline{w} & 2\overline{w}^2 \cr
2zw & 2w\overline{w}-2z\overline{z} & -4\overline{z}\overline{w} \cr
2w^2 & -4\overline{z}w & 4\overline{z}^2} }.
\end{equation}
Since $L$ and $R$ commute, it makes no difference how one fills out the
table by applying $L$ and $R$, e.g., the middle entry can be obtained as either of
$$ RL(z^2)=R(2zw) = 2w\overline{w}-2z\overline{z}, \qquad
LR(z^2) =L(2z\overline{w}) = 2w\overline{w}-2z\overline{z}. $$
Even in this simple example, one can observe the following features of the
general case:
\begin{itemize}
\item The harmonic functions on the edges of the square have the simplest
description, with the formulas becoming more complicated as one moves
towards the centre. 
\item There are symmetries of the polynomials given by
certain permutations of
$z,w,\overline{z},\overline{w}$.
\item One can move around the square table by
applying $L$ and $L^*$ (down and up)
and $R$ and $R^*$ (across and back).
\end{itemize}

We now show that $L^aR^b (z^k)$ has an increasingly
complicated formula as one moves towards the centre of the table.

\begin{lemma}
\label{pablemma}
The unique harmonic polynomial $p_k^{(a,b)}$  in $H_k^{(a,b)}(\HH)$ is given by
\begin{align}
\label{pabformula}
p_k^{(a,b)}
&= \sum_{j=0}^m {(-1)^j\over j!} { (-c)_j (a+b-c-k)_j
\over (|b-a|+1)_j} z^{k-(a+b-c)-j}w^{a-c+j}\overline{z}^{c-j}\overline{w}^{b-c+j} \cr
&=  {(k-a-b+c)!\over k!(c-a-b)_c} L^aR^b(z^k),
\end{align}
where
$$ m=\min\{a,b,k-a,k-b\}, \quad c=\min\{a,b\}={1\over2}(a+b-|b-a|). $$
\end{lemma}

\begin{proof}
We consider the case $a\le b$, i.e., $m=\min\{a,k-b\}$, the other being similar.

By Lemma \ref{Habdecomplemma},
there is a unique (up to a scalar multiple) harmonic polynomial
in $\Hom_k^{(a,b)}(\HH,\CC)$, which by (\ref{Homabstructure}) has the form
$$ f= \sum_{j=0}^m c_j z^{k-b-j}w^j\overline{z}^{a-j}\overline{w}^{b-a+j}. $$
The condition that $f$ be harmonic, i.e., $\nabla f=0$, gives
\begin{align*}
& \sum_{j=0}^{m-1} c_j (k-b-j)(a-j) z^{k-b-j-1}w^j\overline{z}^{a-j-1}\overline{w}^{b-a+j} \cr
& \qquad + \sum_{j=0}^{m-1} c_{j+1} (j+1) (b-a+j+1) z^{k-b-j-1}w^{j}\overline{z}^{a-j-1}\overline{w}^{b-a+j}= 0,
\end{align*}
and equating coefficients of the monomials gives
$$ c_j (k-b-j)(a-j)  + c_{j+1} (j+1) (b-a+j+1) =0 ,
\qquad 0\le j \le m-1, $$
so that
$$ c_{j+1}= - c_j { (k-b-j)(a-j)\over (j+1) (b-a+j+1)}
\Implies
c_j = {(-1)^j\over j!} {(k-b+1-j)_j (a+1-j)_j \over (b-a+1)_j} c_0, $$
which gives the desired formula.
\end{proof}

The indices $\{(a,b)\}_{0\le a,b\le k}$ for
the polynomials $p_k^{(a,b)}\in H_k^{(a,b)}(\HH)$ in the square table
can be partitioned into nested squares
\begin{equation}
\label{Smdefn}
S_m:=\{(a,b):\min\{a,b,k-a,k-b\}=m\},
\qquad 0\le m\le {k\over2},
\end{equation}
with $S_0$ giving the ``edges of the table''. These have size
\begin{equation}
|S_m|=
\begin{cases}
4(k-2m), & 0\le m<{k\over 2}; \cr
1, & m={k\over2}.
\end{cases}
\end{equation}
Here is an illustration for the case $k=4$
$$ \smallmat{
\boxed{0} & \boxed{0} & \boxed{0} & \boxed{0} & \boxed{0} \cr
\boxed{0} & \boxed{1} & \boxed{1} & \boxed{1} & \boxed{0} \cr
\boxed{0} & \boxed{1} & \boxed{2} & \boxed{1} & \boxed{0} \cr
\boxed{0} & \boxed{1} & \boxed{1} & \boxed{1} & \boxed{0} \cr
\boxed{0} & \boxed{0} & \boxed{0} & \boxed{0} & \boxed{0} },
\qquad
S_0=\{\boxed{0}\}, \quad S_1=\{\boxed{1}\}, \quad S_2=\{\boxed{2}\}. $$
From the formula (\ref{pabformula}), we have
the first instance of a general phenomenon:
\begin{itemize}
\item The polynomials $p_k^{(a,b)}$, $(a,b)\in S_m$, have $m+1$ terms,
i.e., the complexity of the formula for $p_k^{(a,b)}$ increases as one
gets closer to the center of the square array.
\end{itemize}
We now consider the decomposition of $\Harm_k(\Hd,\CC)$ into
irreducibles.

The irreducible representations of
the simply connected compact nonabelian Lie group
$\HH^*=U(\HH)=\Sp(1)=SU(2)$ are well known \cite{H15}.
For now, we need only that there is precisely one irreducible
representation $W_k$ of dimension $k+1$, for each $k\ge0$.

\begin{theorem}
\label{leftrightHdecomp}
For left multiplication by $U(\HH)=\HH^*$, we have the
following orthogonal direct sum of irreducibles
$$ \Harm_k(\HH,\CC) = \bigoplus_{0\le a\le k} H(k-a,a)\cong (k+1)W_k, $$
and for right multiplication by $\HH^*$,
we have the direct sum of irreducibles
$$ \Harm_k(\HH,\CC) = \bigoplus_{0\le b\le k} K(k-b,b)\cong (k+1)W_k, $$
\end{theorem}

\begin{proof} From the Schematic \ref{schematicHarmk} for $\Harm_k(\HH,\CC)$,
it follows that by taking columns (respectively rows) of the table gives
an orthogonal direct sum of
invariant subspaces for action given by left (respectively right)
multiplication by $\HH^*$ (Ansatz \ref{ansatzhopetoprove}), 
and so it remains only to show that these
$(k+1)$-dimensional subspaces are irreducible.

We now show $K(k-a,a)$ is irreducible for the action given by
right multiplication.
The other case is similar, and can be found in \cite{F95} Theorem 5.37.
We have
$$ K(k-a,a)=\spam_\CC\{p_k^{(a,b)}\}_{0\le b\le k}
=\spam\{ R^b p_k^{(a,0)}\}_{0\le b\le k}. $$
Consider a nonzero polynomial
$$f=\sum_{0\le b\le b^*} c_b p_k^{(a,b)}\in K(k-a,a), \qquad c_{b^*}\ne 0.$$
Let $V$ be an invariant subspace of $K(k-a,a)$ containing $f$.
Since $R^*$ maps nonzero polynomials left across the table,
we have that $(R^*)^{b^*}f$ is a nonzero multiple of $p_k^{(a,0)}$.
Thus, $V$ contains $p_k^{(a,0)}$, and hence
$R p_k^{(a,0)},\ldots,R^k p_k^{(a,0)}$,
giving $V=K(k-a,a)$, i.e., $V$ is irreducible.
\end{proof}

In both cases, there is a single homogeneous component
corresponding to $W_k$.

\begin{example} For left multiplication by $\HH^*$, i.e., the
action given by
$$ (\ga+j\gb)(z+jw) = (\ga z-\overline{\gb}w) +j (\overline{\ga}w+\gb z), $$
we have the irreducible representation
$$ H(k,0) =\spam_\CC\{z^k,z^{k-1}w,z^{k-2}w^2,\ldots,w^k\}, $$
which is given by Folland \cite{F95} for the
action of $SU(2)\cong\HH^*$ 
given by
$$ \pmat{\ga &-\overline{\gb}\cr\gb&\overline{\ga}} \pmat{z\cr w} = \pmat{\ga z-\overline{\gb}w\cr\overline{\ga}w+\gb z}. $$
Here $L$ and $L^*$ reduce to $L=w{\partial\over\partial z}$
and $L^*=z{\partial\over\partial w}$.
\end{example}

We now consider the combined action given by both left and right
multiplication by $\HH^*$, i.e., the action of
$\Sp(1)\times\Sp(1)=U(\HH)\times \HH^*$
given by
$$\bigl( (q_1,q_2)\cdot f\bigr) (q) := f(q_1 q\overline{q_2}). $$
The invariant subspaces for this action are invariant under both
$L$ and $R$ (and their adjoints). This leads to the following.

\begin{theorem} The action of $\Sp(1)\times\Sp(1)$
given by left and right multiplication by $\HH^*$
is irreducible on $\Harm_k(\HH,\CC)$, i.e.,
for all nonzero $f\in\Harm_k(\HH,\CC)$, we have
\begin{equation}
\label{sptimesspH}
\spam_\CC\{q\mapsto f(q_1 q\overline{q_2}):q_1,q_2\in\HH^*\}
= \Harm_k(\HH,\CC).
\end{equation}
\end{theorem}

We consider the special case of the linear polynomials ($k=1$).


\begin{example} The linear polynomial
$$ f(q):=q=x_1+ix_2+jx_3+kx_4, \qquad x_1,x_2,x_3,x_4\in\RR, $$
is in $\Harm_1(\HH,\HH)=\Hom_1(\HH,\HH)$, as
are the coordinate maps $q\mapsto x_\ell$, which are also
in $\Harm_1(\HH,\RR)=\Hom_1(\HH,\RR)$.
These can be written explicitly
in the form (\ref{sptimesspH}) as follows
\begin{align}
\label{txyzeqns}
x_1 &= {1\over4} ( q - iqi - jqj - kqk ), \qquad
x_2 = {1\over4i}( q - iqi + jqj + kqk ), \cr
x_3 &= {1\over4j}( q + iqi - jqj + kqk ), \qquad
x_4 = {1\over4k}( q + iqi + jqj - kqk ).
\end{align}
The formula (\ref{txyzeqns}) is used by \cite{S79} to show that
the ``polynomials of degree $k$ in $q$'',
i.e., sums of 
the ``monomials''
$$ q\mapsto a_0 q a_1 q a_2\cdots q a_{k-1}q a_k, \qquad a_0,a_1,\ldots, a_k\in\HH, $$
is precisely $\Hom_k(\HH,\HH)$ as we have defined it, or, equivalently,
the $\HH$-linear combinations
of the monomials in real variables $x_1,x_2,x_3,x_4$.
\end{example}

\section{The irreducible representations of $\Harm_k(\Hd,\CC)$}

We now consider the irreducible representations of $\Harm_k(\Hd,\CC)$
for $d\ge2$ (the case usually considered in the literature).
Here, there is more than just the one irreducible $W_k$ involved.
Our method is to construct rectangular 
arrays, like that in Schematic
\ref{schematicHarmk},
corresponding to a given irreducible $W_k,W_{k-2},W_{k-4},\ldots$.
We will say that these are {\bf commuting arrays} if we can move over them
using $L,R,L^*,R^*$, as in the $d=1$ case.
They can be visualised
as the ``layers on (square) wedding cake''.

We follow the development of Bachoc and Nebe \cite{BN02}.
For the action given by right multiplication by $\HH^*$,
let $I(W_p)^{(k)}$ be the homogeneous component of $\Harm_k(\Hd,\CC)$
corresponding to the irreducible $W_p$ (of dimension $p+1$), which gives the
orthogonal decomposition
$$ \Harm_k(\Hd,\CC) = \bigoplus_{p\ge0} I(W_p)^{(k)}. $$
The values of $p$ involved in this sum are $p=k-2j$, $0\le j\le {k\over2}$,
which is observed in \cite{BN02}, and follows from our explicit
decomposition (Theorem \ref{HarmkIWdecomp}).

There is also the well known decomposition \cite{IS68}, \cite{R80}  (Chapter 12, \S12.2)
into irreducibles for
the action of left multiplication by $U(\CC^{2d})$
$$ \Harm_k(\Hd,\CC) = \bigoplus_{a+b=k} H(a,b). $$
Since $U(\Hd)$ is the subgroup of $U(\CC^{2d})\subset O(\RR^{4d})$
characterised as those elements of $U(\CC^{2d})$ which commute with
right multiplication by $\HH^*$ (in the group $O(\RR^{4d})$), we have
the orthogonal direct sum of invariant $U(\Hd)$-modules
\begin{equation}
\label{UHdinvariantdecomp}
\Harm_k(\Hd,\CC) = \bigoplus_{0\le j\le{k\over2}} \bigoplus_{j\le b\le k-j}
H(k-b,b)\cap I(W_{k-2j})^{(k)}.
\end{equation}
This is in fact an orthogonal direct sum of $U(\Hd)$-irreducibles
\begin{equation}
\label{areUirreducible}
R_{k-2j}^{(k)}\cong H(k-b,b)\cap I(W_{k-2j})^{(k)}, 
\end{equation}
(see \cite{II87} \S1.2, \cite{BN02} Theorem 4.1, for $k$ even).



We now give the irreducibles for right multiplication by $\HH^*$.
For $d=1$, these
were obtained by taking a row of the square array (\ref{schematicpict}),
i.e., by choosing a (particular) nonzero element
$f\in H(k,0)=H(k,0)\cap\ker R^*$,
and applying $R$ to it $k$ times. Since $R^*$ moves back in the opposite
direction to $R$, it followed that
\begin{equation}
\label{Wkorbitf}
\spam_\CC\{f,Rf,R^2f,\ldots,R^k f\} \cong W_k
\end{equation}
was an irreducible. Exactly the same argument holds for $d\ge2$, i.e.,
for a nonzero $f\in H(k,0)=H(k,0)\cap\ker R^*$ the subspace (\ref{Wkorbitf})
is irreducible. These are all the irreducibles for $W_k$, and for $d=1$ this
is the end of the story (Theorem \ref{leftrightHdecomp}).
For $d\ge2$, there are other irreducibles, constructed
in a similar way: starting with a nonzero $f$ in the second column,
which does not give the irreducible $W_k$, i.e., $f\in H(k-1,1)\cap\ker R^*$,
one obtains the irreducible subspaces
$$  \spam_\CC\{f,Rf,R^2f,\ldots,R^{k-2} f\} \cong W_{k-2},$$
and so forth.

\begin{theorem}
\label{HarmkIWdecomp}
For the action on $\Harm_k(\Hd,\CC)$ given by
right multiplication by $\HH^*$, the homogeneous component
corresponding to the irreducible $W_{k-2j}$, $0\le j\le{k\over2}$, is
\begin{align*}
I(W_{k-2j})^{(k)} &= \sum_{f\in H(k-j,j)\cap\ker R^*} \spam_\CC\{f,Rf,\ldots,R^{k-2j} f\}
\qquad\hbox{(sum of irreducibles)}, \cr
& = \bigoplus_{j\le b\le k-j} R^{b-j} (H(k-j,j)\cap\ker R^*) \qquad \hbox{(orthogonal direct sum)}.
\end{align*}
Moreover, these are the only irreducibles that appear, i.e., we have
$$ \Harm_k(\Hd,\CC) = \bigoplus_{0\le j\le{k\over2}} I(W_{k-2j})^{(k)}\qquad \hbox{(orthogonal direct sum)}, $$
where the summands above are all nonzero for $d\ge2$,
and $\Harm_k(\HH,\CC)=I(W_k)^{(k)}$.
\end{theorem}

\begin{proof}
From Lemma \ref{HomkRdecomp},
we have the orthogonal direct sum decomposition
$$ \Harm_k(\Hd,\CC) = \bigoplus_{0\le j\le {k\over2}}
\bigoplus_{j\le b\le k-j} R^{b-j} (H(k-j,j)\cap\ker R^*), $$
and so it suffices to show that every nontrivial irreducible subspace
$$ V\subset \bigoplus_{0\le a\le k-2j} R^a (H(k-j,j)\cap\ker R^*)
=\bigoplus_{j\le b\le k-j} R^{b-j} (H(k-j,j)\cap\ker R^*)  $$
has the form
$$ V=\spam_\CC\{f,Rf,\ldots,R^{k-2j} f\}, \quad {f\in H(k-j,j)\cap\ker R^*},$$
so that $V\cong W_{k-2j}$. 

Choose a nonzero $g\in V$, and write
$$ g=\sum_{0\le a\le k-2j} R^a f_a, \qquad f_a\in H(k-j,j)\cap\ker R^*. $$
Let $a^*$ be the largest value of $a$ for which $f_a\ne0$. 
Then, by (\ref{R*alphaRbeta}) of Lemma \ref{RR*commutegeneral},
$$ f:=(R^*)^{a^*} g = (R^*)^{a^*} R^{a^*} f_{a^*}
= (-a^*)_{a^*} (2j-k)_{a^*} f_{a^*}, $$
which is a nonzero scalar multiple of $f_{a^*}$  
(for $2j-k=0$, $a^*=0$),
and
$$ W= \spam_\CC\{f,Rf,\ldots,R^{k-2j} f\}\subset V, \quad R^{k-2j} f\ne0. $$
Hence $\dim(V)\ge\dim(W)= k-2j+1$. By construction, $W$
is invariant under the action of $R$ and $R^*$. Supposing each such $W$ were 
invariant under right multiplication, i.e., $V=W$, 
then our decomposition would have the same multiplicity of each irreducible 
$W_{k-2j}$ in $\Harm_k(\Hd,\CC)$ as 
as the abstract decomposition of \cite{II87} and \cite{BN02}. 
This must indeed be the case, since otherwise a union of some of the $W$'s would be 
an irreducible, giving a contradiction.
\end{proof}

The last part of the proof above gives the following.

\begin{corollary} 
\label{Rrightmultequiv}
The irreducibles for right multiplication by $\HH^*$ and for
multiplication by $R$ and $R^*$ are the same.
\end{corollary}

This is effectively the Ansatz \ref{ansatzhopetoprove}, 
which we had hoped to prove by elementary means.
We will call a sequence
$$ f,Rf,\ldots, R^{k-2j}f, \qquad 0\le j\le\hbox{${k\over2}$}, $$
or any nonzero scalar multiples of it, an {\bf $R$-orbit} (for $W_{k-2j}$) if
$$ f\in\Hom_H(k-j,j), \qquad R^*f =0. $$
It follows from Theorem \ref{HarmkIWdecomp} that $R^{k-2j}f\ne0$,
and
\begin{equation}
\label{Rorbitdescript}
R\{f\}:=\spam_\CC\{f,Rf,\ldots, R^{k-2j}f\},
\end{equation}
is an irreducible subspace (of dimension $k+1-2j$) for right multiplication by $\HH^*$.

We can now give an explicit form for the $U(\Hd)$-irreducibles.

\begin{theorem}
\label{HarmkUirredthm}
Let $d\ge2$.
For the action on $\Harm_k(\Hd,\CC)$ given by $U(\Hd)=\Sp(d)$,
we have the
following orthogonal direct sum of irreducibles
\begin{align}
\label{HarmkUirrdecomp}
\Harm_k(\Hd,\CC) &= \bigoplus_{0\le j\le {k\over2}}
\bigoplus_{j\le b\le k-j} H(k-b,b)_{k-2j}  \cr
& \cong \bigoplus_{0\le j\le {k\over2}} (k-2j+1) R_{k-2j}^{(k)},
\end{align}
where
\begin{align}
\label{H(k-b,b)_(k-2j)form}
H(k-b,b)_{k-2j}
&:= R^{b-j} \bigr(\ker R^*\cap H(k-j,j)\bigl) \cr
&\ = (R^*)^{k-b-j} \bigr(\ker R\cap H(j,k-j)\bigl) \cr
&\ = H(k-b,b) \cap I(W_{k-2j})^{(k)} \cr
&\ \cong R_{k-2j}^{(k)} :=  \ker R^*\cap H(k-j,j),
\end{align}
and
\begin{equation}
\label{Rkirrdim}
\dim(R_{k-2j}^{(k)})
= (k-2j+1) (k+2d-1) {(k-j+2d-2)!(j+2d-3)!\over(k-j+1)!j!(2d-1)!(2d-3)!}.
\end{equation}
\end{theorem}

\begin{proof} 
By Lemma \ref{HomkRdecomp}, we already have that 
(\ref{HarmkUirrdecomp}) is an orthogonal direct sum of
$U(\Hd)$-invariant subspaces, with $H(k-b,b)_{k-2j}\subset H(k-b,b)$,
given by the first two formulas in (\ref{H(k-b,b)_(k-2j)form}). 
We therefore need only show that they are $U(\Hd)$-irreducible, 
i.e., given by the formula (\ref{UHdinvariantdecomp}), i.e., the 
third formula, with (\ref{areUirreducible}) holding (the fourth formula).
By Theorem \ref{HarmkIWdecomp}, we have
$$ I(W_{k-2j})^{(k)}
 = \bigoplus_{j\le a\le k-j} R^a (H(k-j,j)\cap\ker R^*). $$
Since $R^a(H(k-j,j)\cap\ker R^*)\subset H(k-j-a,j+a)$,
the only contribution to the intersection
with $ H(k-b,b)$ is when $b=j+a$,
which gives the third formula, i.e.,
$$ H(k-b,b)\cap I(W_{k-2j})^{(k)} =
R^{b-j} \bigr(\ker R^*\cap H(k-j,j)\bigl). $$

We now show, that for $j$ fixed, the $H(k-b,b)_{k-2j}$ are 
isomorphic $U(\Hd)$-irreducibles.  
Taking $\ga=\gb=b-j$ in Lemma \ref{RRstarcancelcor} gives
$$ (R^*)^{b-j}H(k-b,b)_{k-2j} 
= (R^*)^{b-j} R^{b-j} \bigr(\ker R^*\cap H(k-j,j)\bigl)
= \ker R^*\cap H(k-j,j).  $$
This implies the subspaces have the same dimension as 
$\ker R^*\cap H(k-j,j)$, which is given by 
equation (\ref{H(k-b,b)k-2jdim}) of Lemma \ref{H(k-j,j)capKerR*dim}. 
Finally, since the action of $U(\Hd)$ commutes with the action 
of $R$ (and its powers), these subspaces are all $U(\Hd)$-isomorphic
to $R_{k-2j}^{(k)} :=  \ker R^*\cap H(k-j,j)$.
\end{proof}

This decomposition is given in \cite{BN02} Theorem 4.1
(for $k$ even, the summands not given explicitly),
and in \cite{Ghent14} (Theorems 1 and 2).
The presentation of \cite{Ghent14} involves two separate cases
for the decomposition of $H(a,b)$, namely
$$ H(k-b,b)_{k-2j} =
\begin{cases}
R^{b-j} (H(k-j,j)\cap\ker R^*),   & k-b\ge b; \cr
(R^*)^{k-b-j} (H(j,k-j)\cap\ker R),  & k-b\le b.
\end{cases}
$$


We now consider the irreducibles for the action on $\Harm_k(\Hd,\CC)$
given by both left multiplication
by $U\in U(\Hd)$ and right multiplication by $q^*\in\HH^*$, i.e.,
$$ ((U,q^*)\cdot f)(q) := f(U q \overline{q^*}). $$

\begin{theorem}
\label{Qkdecompthm}
For the action on $\Harm_k(\Hd,\CC)$
given by $U(\Hd)\times\HH^*=\Sp(d)\times\Sp(1)$, $d\ge2$, 
we have the
following orthogonal direct sum of irreducibles
\begin{equation}
\label{Qkdecompformula}
\Harm_k(\Hd,\CC) = \bigoplus_{0\le j\le {k\over2}} Q_{k-2j}^{(k)},
\end{equation}
where
\begin{align}
\label{Qk-2jdef}
Q_{k-2j}^{(k)} &:= \bigoplus_{j\le b\le k-j} H(k-b,b)_{k-2j} \cr
&\ = \bigoplus_{j\le b\le k-j}R^{b-j} \bigr(\ker R^*\cap H(k-j,j)\bigl)
= I(W_{k-2j})^{(k)},
\end{align}
and
\begin{equation}
\label{Qk-2jirreduc}
\dim(Q_{k-2j}^{(k)})
= (k-2j+1)^2 (k+2d-1) {(k-j+2d-2)!(j+2d-3)!\over(k-j+1)!j!(2d-1)!(2d-3)!}.
\end{equation}
\end{theorem}

\begin{proof} The subspace $Q_{k-2j}^{(k)}$ is invariant
under the actions of $U(\Hd)$ and $\HH^*$, as it is a sum
of irreducibles for each of these actions.
We now show that it is irreducible.

Suppose $V\subset Q_{k-2j}^{(k)}$ is irreducible under
the action of $U(\Hd)\times\HH^*$.
By Theorem \ref{HarmkIWdecomp},
$V\subset I(W_{k-2j})^{(k)}$, and $V$
contains an irreducible for the action of $\HH^*$ of the
form
$$ \spam_\CC\{f,Rf,\ldots, R^{k-2j}f\}, \qquad
0\ne R^{b-j} f \in H(k-b,b)_{k-2j}, \quad j\le b\le k-j. $$
Since each $H(k-b,b)_{k-2j}$ is $U(\Hd)$-irreducible,
we have that $H(k-b,b)_{k-2j}\subset V$, and hence
$V=Q_{k-2j}^{(k)}$ is irreducible.
\end{proof}

In other words, the $\Sp(d)\times\Sp(1)$-irreducibles $Q_{k-2j}^{(k)}$
are precisely the homogeneous components $I(W_{k-2j})$
for right multiplication by $\HH^*$.

The decomposition (\ref{Qkdecompformula}) of $\Harm_k(\Hd,\CC)$
into $\Sp(d)\times\Sp(1)$-irreducibles
is given in \cite{S74} Theorem 2.4,
and \cite{ACMM20} Proposition 2.1 (as the joint eigenfunctions of operators
$\gD_\SS$ and $\gG$),
where the following notations are used (respectively)
$$ Q_{k-2j}^{(k)}
= \begin{cases}
H_{j,{k\over2}-j}, & \hbox{($k$ even)}; \cr
{\tilde H}_{j,{k-1\over2}-j}, & \hbox{($k$ odd)},
\end{cases}
\qquad\qquad
Q_{k-2j}^{(k)} = \cH_{k,j}. $$
Both observe that $Q_{k-2j}^{(k)}$ is invariant
under conjugation, and so has a basis of real-valued polynomials,
and a real-valued zonal function (a function invariant under the
subgroup of $\Sp(d)\times\Sp(1)$ that fixes a point).
The structural form of this zonal is given in
\cite{S74} Proposition 2.8,
and it is given explicitly in
\cite{ACMM20} Proposition 3.1.

The invariance of $Q_{k-2j}^{(k)}$ under conjugation follows 
directly from (\ref{Rconjids}), i.e.,
\begin{align}
\label{irreducsymmetries}
\overline{H(k-b,b)_{k-2j}}
& = \overline{ R^{b-j} \bigr(\ker R^*\cap \Hom_H(k-j,j)\bigl)} \cr
& = (R^*)^{b-j} \bigr(\overline{\ker R^*}\cap \overline{\Hom_H(k-j,j)}\bigl) \cr
& = (R^*)^{b-j} \bigr(\ker R\cap \Hom_H(j,k-j)\bigl)  \cr
&= H(b,k-b)_{k-2j}.
\end{align}

\begin{schematic}
\label{WeddingcakeRk-2j}
(Wedding cake) The orthogonal decomposition
(\ref{HarmkUirrdecomp}) of $\Harm_k(\Hd,\CC)$ into $U(\Hd)$-irreducibles can
be displayed as layers of a ``wedding cake''
\begin{center}
${\footnotesize \begin{array}{ r|lllclll }
& H(k,0) & H(k-1,1) & H(k-2,2) & \cdots & H(2,k-2) & H(1,k-1) & H(0,k) \\
& \\[-8pt]
\hline
& \\[-8pt]
& & \moveRb\qquad  & & {\tiny\vdots} & & \qquad\moveR & \cr
R_{k-4}^{(k)}: \qquad & & & \grey H(k-2,2)_{k-4} & \cdots & H(2,k-2)_{k-4} & & \cr
R_{k-2}^{(k)}: \qquad  &
& \grey H(k-1,1)_{k-2} & H(k-2,2)_{k-2} & \cdots & H(2,k-2)_{k-2} & H(1,k-1)_{k-2} & \cr
R_k^{(k)}: \qquad &
\grey H(k,0)_k & H(k-1,1)_k \ \ \, & H(k-2,2)_k \ \ \, & \cdots & H(2,k-2)_k \ \ \,
& H(1,k-1)_k \ \ \, & H(0,k)_k \cr
\end{array}}
$
\end{center}
where the layers (rows) correspond to the irreducible $R_{k-2j}^{(k)}$
(the bottom layer is $R_k^{(k)}$), and the slices (columns) correspond
to the decomposition of a given $H(k-b,b)$ into $\min\{b,k-b\}+1$ different
irreducibles.
One can move along the
layers using $R$ and $R^*$, as indicated. Therefore, the left most irreducibles
(shaded in grey), i.e.,
$$ H(k-j,j)_k=H(k-j,j)\cap\ker R^*, \qquad 0\le j\le {k\over 2}, $$
are a distinguished copy of each irreducible, 
from which the other summands in the layer 
can be obtained by applying $R$.
Further, in view of the symmetries (\ref{irreducsymmetries}), i.e., that conjugation
reflects the cake around its centre, only half of these summands need be calculated,
in practice.
Similarly, the right most entries are distinguished, and give the other summands 
by applying $R^*$.
\end{schematic}


\begin{example}
We consider $\Harm_2(\Hd,\CC)$, for which (\ref{HarmHomdims}) gives
$$ \dim(\Harm_2(\Hd,\CC))=2d(4d+1)-1 =  (2 d + 1) (4 d - 1). $$
For $d=2$, we have the following table, where each line is an $R$-orbit, as in
(\ref{Rorbitdescript}).
\begin{center}
$\begin{array}{ c|ccc }
&& \\[-10pt]
& H(2,0) & H(1,1) & H(0,2) \\
&& \\[-10pt]
\hline
&& \\[-10pt]
\multirow{4}{4em}{$K(2,0)\Bigl\{$}
& z_1^2 & z_1\overline{w_1} & \overline{w_1}^2 \\
& z_2^2 & z_2\overline{w_2} & \overline{w_2}^2 \cr
& z_1z_2 & z_1\overline{w_2}+z_2\overline{w_1} & \overline{w_1}\overline{w_2}\cr
&& \\[-10pt]
& & z_1\overline{w_2}-z_2\overline{w_1} \\
&& \\[-10pt]
\hline
&& \\[-10pt]
\multirow{7}{4em}{$K(1,1)\Bigl\{$}
&z_1w_1 & z_1\overline{z_1}-w_1\overline{w_1} & \overline{z_1}\overline{w_1}\cr
&z_1w_2 & \overline{w_1}w_2-z_1\overline{z_2} & \overline{w_1}\overline{z_2}\cr
&z_2w_1 & \overline{w_2}w_1-z_2\overline{z_1} & \overline{w_2}\overline{z_1}\cr
&z_2w_2 & z_2\overline{z_2}-w_2\overline{w_2} & \overline{z_2}\overline{w_2}\cr
&& \\[-10pt]
&& \overline{z_1}z_2+w_1\overline{w_2} \cr
&& z_1\overline{z_2}+\overline{w_1}w_2 \cr
&& z_2\overline{z_2}+w_2\overline{w_2} -z_1\overline{z_1}-w_1\overline{w_1} \cr
&& \\[-10pt]
\hline
&& \\[-10pt]
\multirow{4}{4em}{$K(0,2)\Bigl\{$}
& w_1^2 & \overline{z_1}w_1&\overline{z_1}^2 \cr
& w_2^2 & \overline{z_2}w_2&\overline{z_2}^2 \cr
& w_1w_2 & \overline{z_1}w_2+\overline{z_2}w_1 & \overline{z_1}\overline{z_2} \cr
&& \\[-10pt]
& & \overline{z_1} w_2 - \overline{z_2} w_1 \\
&& \\[-10pt]
\hline
\end{array}
$
\end{center}
For example, we have the 
decomposition into irreducibles for
right multiplication by $\HH^*$
$$ K(2,0)=(R\{z_1^2\}\oplus R\{z_2^2\}\oplus R\{z_1z_2\})
\oplus R\{z_1\overline{w_2}-z_2\overline{w_1}\}
\cong 3 W_2 \oplus W_0, $$
where
$$ R\{z_1^2\}=\spam\{z_1^2,z_1\overline{w_1},\overline{w_1}^2\}\cong W_2, \quad
R\{z_1\overline{w_2}-z_2\overline{w_1}\}=\spam\{z_1\overline{w_2}-z_2\overline{w_1}\}
\cong W_0, $$
etc.
This calculation was done for $\Hom_2(\HH^2)$, which has a dimension $1$ higher.
Apart from applying $R$ to fill out the rows, the only other calculation
done
was solving $Rf=0$ or $R^*f$ for $f\in H_2^{(1,1)}(\HH^2)=K(1,1)\cap H(1,1)$
gives a 4-dimensional space spanned by
$$ z_1\overline{z_1}+w_1\overline{w_1}, \quad
z_2\overline{z_2}+w_2\overline{w_2}, \quad
\overline{z_1}z_2+w_1\overline{w_2}, \quad
z_1\overline{z_2}+\overline{w_1}w_2, $$
The first two have nonzero constant Laplacian, so their difference is harmonic,
and the second two are harmonic.
Similar calculations give the general decomposition
\begin{align*}
K(2,0) &=\Bigl( \bigoplus_{|\ga|=2}R\{z^\ga\}\Bigr)\oplus
\Bigl( \bigoplus_{1\le j<k\le d}R\{z_j\overline{w_k}-z_k\overline{w_j}\}\Bigr)
\cong {1\over2}d(d+1)W_2\oplus {1\over2}d(d-1)W_0, \cr
K(1,1) &= \Bigl( \bigoplus_{1\le j,k\le d} R\{z_jw_k\}\Bigr)
\oplus\Bigl( \bigoplus_{j\ne k}R\{z_j\overline{z_k}+\overline{w_j}w_k\}
\oplus \bigoplus_{2\le j\le d} R\{ z_j\overline{z_j}+w_j\overline{w_j}
-z_1\overline{z_1}-w_1\overline{w_1}\}\Bigr) \cr
& \cong d^2 W_2\oplus(d^2-1)W_0, \cr
K(0,2) &=\Bigl( \bigoplus_{|\ga|=2}R\{w^\ga\}\Bigr)\oplus
\Bigl( \bigoplus_{1\le j<k\le d}R\{\overline{z_j}w_k-\overline{z_k}w_j\}\Bigr)
\cong {1\over2}d(d+1)W_2\oplus {1\over2}d(d-1)W_0,
\end{align*}
into irreducibles (sums of $R$-orbits).
In particular, the homogeneous components, i.e., the
$\Sp(d)\times\Sp(1)$-irreducibles, are
$$ \Harm_2(\Hd) = Q_2^{(k)}\oplus Q_0^{(k)}=  I(W_2)^{(2)}\oplus I(W_0)^{(2)}
\cong d(2d+1)W_3\oplus (d-1)(2d+1) W_0. $$
\end{example}

\begin{example} Since $H(k,0)\cap\ker R^*=H(k,0)$, we have
$$
I(W_k)^{(k)} = \bigoplus_{|\ga+\gb|=k} R\{z^\ga w^\gb\}
\cong {k+2d-1\choose k} W_k \quad \hbox{(orthogonal direct sum)}. $$
\end{example}


\section{Zonal polynomials}

Here we consider the ``zonal polynomials'' for our irreducible representations
of the groups $G=U(\Hd),U(\Hd)\times\HH^*$ on $\Harm_k(\Hd,\CC)$.
There are two common notions of zonal functions:
\begin{itemize}
\item The functions fixed by the action of the subgroup $G_{q'}$ which
fixes a point $q'$.
\item The Riesz representer of point evaluation at a point $q'$
(the reproducing kernel).
\end{itemize}
When $G_{q'}$ is a maximal compact subgroup of $G$ these are equivalent.
We will consider the first notion. For a group $G$ acting on $\Hd$,
we define the stabliser (or isotrophy) subgroup of $q'\in\Hd$ to be
those elements which fix $q'$, i.e.,
$$ G_{q'} := \{ g\in G:g\cdot q'=q'\}. $$
A function $\Hd\to\CC$ which is fixed by the action of $G_{q'}$ is 
said to {\bf zonal} (with pole $q'$).
We denote the subspace of zonal functions in a space $V$ of polynomials by 
$$ V^{G_{q'}} :=\{ f\in V: g\cdot f=f\}. $$
We now condsider the zonal polynomials for the group $G=U(\Hd)$.

Recall $\inpro{v,w}=v^* w$ is the Euclidean inner product (\ref{Innerproductdefn}).
For vectors $q=z+jw$, $q'=z'+jw'$ in $\Hd$, 
we define two inner products
\begin{equation}
\label{Hdinnerprods}
\inpro{q',q}_{\Hd}:=\inpro{q',q}\in\HH, \qquad
\inpro{q',q}_{\CC^{2d}}:= \inpro{\pmat{z'\cr w'},\pmat{z\cr w}}\in\CC.
\end{equation}

\begin{lemma}
\label{zonalformqq'}
For $q,q'\in\Hd$, we have
\begin{equation}
\label{Hdinprorelationship}
\inpro{q',q}_{\Hd} = \inpro{q',q}_{\CC^{2d}} + j \inpro{q'j,q}_{\CC^{2d}}.
\end{equation}
For the action of $U(\Hd)$ the following are zonal polynomials $\Hd\to\CC$
with pole $q'$ 
$$ q\mapsto \inpro{q',q}_{\CC^{2d}}=\overline{z_1'}z_1+\cdots+\overline{z_d'}z_d
+\overline{w_1'}w_1+\cdots+\overline{w_d'}w_d ,$$
$$ q\mapsto \inpro{q'j,q}_{\CC^{2d}}=z_1'w_1+\cdots+z_d'w_d
-w_1'z_1-\cdots-w_d'z_d. $$
When $q'=e_1$, the zonal polynomials above are
$$ z+jw\mapsto z_1, \qquad z+jw\mapsto w_1. $$
\end{lemma}

\begin{proof} 
Using (\ref{jzcommute}), we calculate
$$ q'j=(z'+jw')j=z'j+jw'j= -\overline{w'}+ j\overline{z'}, $$
and so
\begin{align*}
\inpro{q',q}_{\Hd} 
&= (z'+jw')^*(z+jw)
= ((z')^* -(w')^*j) (z+jw) \cr
&= (z')^*z+ (w')^*w + j(\overline{z'})^*w -j(\overline{w'})^*z \cr
&= \inpro{\pmat{z'\cr w'},\pmat{z\cr w}}
+j \inpro{\pmat{-\overline{w'}\cr \overline{z'}},\pmat{z\cr w}}
= \inpro{q',q}_{\CC^{2d}} + j \inpro{q'j,q}_{\CC^{2d}},
\end{align*}
which is (\ref{Hdinprorelationship}).
Let $U\in U(\Hd)$ with $Uq'=q'$, then we have
$$ \inpro{q',q}_{\Hd}
= \inpro{Uq',Uq}_{\Hd}
= \inpro{q',Uq}_{\Hd}
= \inpro{q',Uq}_{\CC^{2d}} + j \inpro{q'j,Uq}_{\CC^{2d}}, $$
so that 
$$ \inpro{q',Uq}_{\CC^{2d}}=\inpro{q',q}_{\CC^{2d}}, \qquad
\inpro{q'j,Uq}_{\CC^{2d}}=\inpro{q'j,q}_{\CC^{2d}}, $$
which shows that the linear polynomials given are zonal.
\end{proof}

We note that $z_1$ and $w_1$ are zonal polynomials in the $U(\Hd)$-irreducible
subspace 
$$ H(1,0)_1=\spam\{z_1,\ldots,z_d,w_1,\ldots w_d\}, $$
and so the space of zonal polynomials in a given $U(\Hd)$-irreducible 
is not $1$-dimensional, as it is in the real and complex cases.

\begin{example}
The quadratic polynomial $q\mapsto\norm{q}^2=\inpro{q,q}_{\Hd}$ is zonal
(for any $q'$). 
By folk law (the real and complex cases), the zonal polynomials should 
be a function of this and the quaternionic inner product
$q\mapsto \inpro{q',q}_{\Hd}=(q')^* q$.
Using (\ref{txyzeqns}), we have the
explicit formulas:
$$ \inpro{q',q}_{\CC^{2d}} = {1\over2}
(\inpro{q',q}_{\Hd}-i\inpro{q',q}_{\Hd}i), \qquad
\inpro{q'j,q}_{\CC^{2d}} = {1\over2j}(\inpro{q',q}_{\Hd}
+i\inpro{q',q}_{\Hd}i). $$
\end{example}

Using the zonal polynomials above,
which commute, since they are complex-valued, 
\cite{BN02} define zonal polynomials in $\Hom_k(\Hd,\CC)$ by 
\begin{equation}
\label{zonalq'def}
[\ga_1,\ga_2,\ga_3,\ga_4,r]_{q'}(q):= 
\inpro{q',q}_{\CC^{2d}}^{\ga_1}
\inpro{q'j,q}_{\CC^{2d}}^{\ga_2}
\overline{\inpro{q',q}}_{\CC^{2d}}^{\ga_3}
\overline{\inpro{q'j,q}}_{\CC^{2d}}^{\ga_4}
\norm{q}^{2r}, 
\end{equation}
where $\ga_1+\ga_2+\ga_3+\ga_4+2r=k$.
These span and hence are a basis for the zonal polynomials in $\Hom_k(\Hd,\CC)$ 
(Proposition 4.2, \cite{BN02}).

\begin{example} For a general $q'$, we have
$$ [\ga_1,\ga_2,\ga_3,\ga_4,r]_{q'} \in \Hom_H(\ga_1+\ga_2+r,\ga_3+\ga_4+r), $$
and for $q'=e_1$, we have
\begin{equation}
\label{Homkzonalbasis}
[\ga_1,\ga_2,\ga_3,\ga_4,r]_{e_1} = z_1^{\ga_1}w_1^{\ga_2}\overline{z_1}^{\ga_3}
\overline{w_1}^{\ga_4}\norm{z+jw}^{2r}, \quad \ga_1+\ga_2+\ga_3+\ga_4+2r=k,
\end{equation}
so that 
$$ [\ga_1,\ga_2,\ga_3,\ga_4,r]_{e_1} 
\in \Hom_k^{(\ga_2+\ga_3+r,\ga_3+\ga_4+r)}(\Hd). $$
\end{example}

We can take advantage of (\ref{Homkzonalbasis}) to simplify the
proof and presentation of results, since if $U$ is unitary with $Uq'=e_1$, then
we have the following correspondence between zonal polynomials
with poles $q'$ and $e_1$
$$ [a_1,a_2,a_3,a_4,r]_{q'} = [a_1,a_2,a_3,a_4,r]_{e_1}(V\cdot), $$
This follows from the calculation
$$ \inpro{q',q}_\Hd=\inpro{Uq',Uq}_\Hd=\inpro{e_1,Uq}_\Hd, $$
and the fact that such a $U$ can always be constructed, since
$U(\Hd)$ is transitive on the quaternionic sphere.
In effect, a zonal polynomial for $q'$ can be obtained from one with pole $e_1$
by making the substitution
\begin{equation}
\label{e1toq'subs}
z_1\mapsto\inpro{q',q}_{\CC^{2d}}, \qquad
w_1\mapsto\inpro{q'j,q}_{\CC^{2d}}.
\end{equation}


The number of zonal functions given by (\ref{zonalq'def}) 
is independent of the dimension $d$.
For $d=1$, these zonal polynomials have linear dependencies, e.g.,
$$ [1,0,1,0,0]+[0,1,0,1,0]= z_1\overline{z_1}+w_1\overline{w_1} = [0,0,0,0,1], $$
and for $d>1$ they are linearly dependent. Thus we obtain the following
dimensions.

\begin{lemma}
For $Z:=U(\Hd)_{q'}$, $d\ge2$, 
the zonal polynomials have dimensions
\begin{align}
\label{HomkZdim}
\dim(\Hom_k(\Hd,\CC)^Z)
&= \sum_{0\le j\le {k\over 2}} {k-2j+3\choose3},  \\
\label{HarmkZdim}
\dim(\Harm_k(\Hd,\CC)^Z)
&= {k+3\choose3} = \sum_{0\le j\le {k\over 2}} (k-2j+1)^2.
\end{align}
Further, if $q'=z'\in\Cd$, e.g., $q'=e_1$, then 
\begin{align}
\dim(\Hom_k^{(a,b)}(\Hd)^Z) &= {1\over2}(m+1)(m+2), \\
\dim(H_k^{(a,b)}(\Hd)^Z) &= m+1,
\end{align}
where 
 $$ m:=\min\{a,k-a,b,k-b\}. $$
\end{lemma}

\begin{proof}
Since the zonal polynomials in (\ref{Homkzonalbasis}) are clearly linearly 
independent and span $\Hom_k(\Hd,\CC)^Z$ (see \cite{BN02} Proposition 4.2),
it suffices to count them, 
which gives
$$ \dim(\Hom_k(\Hd,\CC)^Z) = \sum_{0\le r \le{k\over2}}
\dim(\Hom_{k-2r}(\HH,\CC))
= \sum_{0\le r\le {k\over 2}} {k-2r+3\choose3}. $$
When $q'=z'$ ($w'=0$), 
each of these zonal polynomials is 
in some $\Hom_k^{(a,b)}(\Hd)$,
so that
\begin{equation}
\label{HomkabZdecomp}
\Hom_k^{(a,b)}(\Hd)^Z
= \bigoplus_{m\le r \le {k\over2}}
\spam_\CC\bigr\{[\ga_1,\ga_2,\ga_3,\ga_4,r]:
\mat{\ga_1+\ga_4=k-a-r,\ \ga_2+\ga_3=a-r\cr
\ga_1+\ga_2=k-b-r,\ \ga_3+\ga_4=b-r } \bigl\},
\end{equation}
and counting again, using (\ref{Homabdim}) and $m_{a-r,b-r}^{(k-2r)}=m+1-r$, gives
\begin{align*}
\dim(\Hom_k^{(a,b)}(\Hd)^Z)
&= \sum_{m\le r\le {k\over2}} \dim(\Hom_{k-2r}^{(a-r,b-r)}(\HH)) \cr
&= 1+2+\cdots+m+(m+1)
= {1\over2}(m+1)(m+2).
\end{align*}
Since the Laplacian maps 
$\Hom_k^{(a,b)}(\Hd)$ onto $\Hom_{k-1}^{(a-1,b-1)}(\Hd)$
and zonal polynomials to zonal polynomials (Lemma \ref{LaplacianZonal}), 
we have
$$ \dim(\Harm_k(\Hd,\CC)^Z) 
= \dim(\Hom_k(\Hd,\CC)^Z)-\dim(\Hom_{k-2}(\Hd,\CC)^Z)
= {k+3\choose3}, $$
\begin{align*}
\dim(H_k^{(a,b)}(\Hd)^Z)
&= \dim(\Hom_k^{(a,b)}(\Hd)^Z)
-\dim(\Hom_{k-2}^{(a-1,b-1)}(\Hd)^Z) \cr
&= {1\over2}(m+1)(m+2)-{1\over2}(m-1+1)(m-1+2)
= m+1,
\end{align*}
which completes the proof.
\end{proof}

We will give a simple example first, which motivates 
the general and constructive result to follow. 

\begin{example}
\label{exshowsymms}
For $q'=e_1$, the unique zonal polynomial
in $H_{k}^{(0,0)}(\Hd)$ is
$$ [k,0,0,0,0] = z_1^k. $$
We may apply $L$ (down) and $R$ (right) to this, as in the univariate case
(Schematic \ref{schematicHarmk} and Lemma \ref{pablemma}),
to obtain $(k+1)^2$ zonal polynomials in $I(W_k)^{(k)}$.
$$ \mat{
z_1^k & k z_1^{k-1} \overline{w_1}
& \cdots & k! \overline{w_1}^k \cr
k z_1^{k-1}w_1 & k(k-1) z_1^{k-2} w_1 \overline{w_1} -kz_1^{k-1}\overline{z_1}
&\cdots  & -k!k \overline{z_1} \overline{w_1}^{k-1} \cr
k(k-1) z_1^{k-2}w_1^2 &
k(k-1) \{(k-2) z_1^{k-3} w_1^2 \overline{w_1} -2z_1^{k-2}w_1\overline{z_1}\}
& \cdots & k!k(k-1) \overline{z_1}^2 \overline{w_1}^{k-2} \cr
\vdots & \vdots  & & \vdots \cr
k! z_1w_1^{k-1} & k!w_1^{k-1}\overline{w_1} -k!(k-1)z_1w_1^{k-2}\overline{z_1}
& \cdots & (-1)^{k-1} k!^2 \overline{z_1}^{k-1}\overline{w_1} \cr
k! w_1^k  & -k!k w_1^{k-1}\overline{z_1}
& \cdots & (-1)^k k!^2 \overline{z_1}^k \cr
 } $$
\end{example}

\begin{theorem} 
\label{HarmonicZonalconst}
Let $q'=e_1$.
For $d\ge2$, there is a unique harmonic zonal polynomial
$$ P_{k-2j}^{(k)}=P_{k-2j,d}^{(k)} \in \ker L^*\cap\ker R^*\cap H_k^{(j,j)}(\Hd), \qquad 0\le j\le{k\over2}, $$
given by 
\begin{equation}
\label{Zk-2j(k)explicitform}
P_{k-2j}^{(k)} := \sum_{b+c+r=j} {(-1)^r \over b!c!r!}
{(k+2-j-r)_r\over(k+2d-1-r)_{r}}[k-j-b-r,b,c,b,r],
\end{equation}
which has ${1\over2}(j+1)(j+2)$ terms.  Let
\begin{equation}
\label{Zk-2jab(k)explicitform}
P_{k-2j,a,b}^{(k)} := L^{a-j} R^{b -j} P_{k-2j}^{(k)}, \qquad
j\le a,b\le k-j.
\end{equation}
Then the zonal polynomials (with pole $e_1$) in $\Harm_k(\Hd,\CC)$ 
have the following orthogonal direct sum decomposition
into one-dimensional subspaces
\begin{equation}
\label{HarmkPonedimdecomp}
\Harm_k(\Hd,\CC)^Z = \bigoplus_{0\le j\le {k\over2}}
\bigoplus_{j\le a,b\le k-j} \spam\{ P_{k-2j,a,b}^{(k)}\}.
\end{equation}
\end{theorem}

\begin{proof} By (\ref{HomkabZdecomp}), 
a general zonal polynomial in $\Hom_k^{(j,j)}(\Hd)$ has the form
$$ f:=\sum_{b+c+r=j} C_{bcr} [k-j-b-r,b,c,b,r], \qquad  C_{bcr}\in\CC, $$
which involves ${1\over2}(j+1)(j+2)$ terms. 
By Lemma \ref{LaplacianZonal}, the condition for $f$ to be harmonic is
\begin{align*}
{1\over 4} \Delta f 
= \sum_{b+c+r=j} C_{bcr} \bigl\{  &
(k-j-b-r)c[k-j-b-r-1,b,c-1,b,r] \cr
&\quad +b^2 [k-j-b-r,b-1,c,b-1,r] \cr
&\quad +r(k+2d-1-r)[k-j-b-r,b,c,b,r-1] \bigr\} = 0,
\end{align*}
which gives ${1\over2}j(j+1)$ equations, and hence 
$j+1=\dim((H_k^{(j,j)})^Z)$
free parameters.
Hand calculations indicated that $\Delta f=0$ together with 
the conditions $R^*f=0$ and $L^*f=0$ leads to
a unique (one parameter) solution $f$.
From these special cases, we ``guessed'' the formula 
(\ref{Zk-2j(k)explicitform}). We will now verify directly 
that $f$ defined by (\ref{Zk-2j(k)explicitform})
has the desired properties, 
and then conclude that it is unique (by a cardinality argument).

By Lemma \ref{LaplacianZonal} and Lemma \ref{Rmapszonaltozonal}, we have
\begin{align*} \Delta ([k-j-b-r,b,c,b,r])
&= (k-j-b-r)c[k-j-b-r-1,b,c-1,b,r] \cr
&\qquad 
+b^2[k-j-b-r,b-1,c,b-1,r] \cr
&\qquad +r(k+2d-1-r)[k-j-b-r,b,c,b,r-1], \cr
R^* ([k-j-b-r,b,c,b,r])
&=  -c[k-j-b-r,b+1,c-1,b,r] \cr
&\qquad +b[k-j-b-r+1,b,c,b-1,r], \cr
L^* ([k-j-b-r,b,c,b,r])
&= b[k-j-b-r+1,b-1,c,b,r] \cr
&\qquad -c[k-j-b-r,b,c-1,b+1,r].
\end{align*}
Hence, the $[k-j-b'-r'-1,b',c',b',r']$ coefficient of
$\Delta f$ is
\begin{align*}
&{(-1)^{r'}\over(b')!(c')!} {(k+2-j-r')_{r'}\over(k+2d-1-r')_{r'}}
\Bigl\{ {1\over c'+1}(k-j-b'-r')(c'+1) \cr
& + {1\over b'+1}(b'+1)^2-{1\over r'+1}{(k+2-j-r'-1)\over(k+2d-1-r'-1)}
(r'+1)(k+2d-1-r'-1) \Bigr\}=0, 
\end{align*}
the $[k-j-b'-r+1,b',c',b'-1,r]$, $b'\ne0$, coefficient of $R^*f$ is
$$ {(-1)^r\over r!} {(k+2-j-r)_r\over(k+2d-1-r)_{r}}
\Bigl( {1\over(b'-1)!(c'+1)!}(-(c'+1))+{1\over(b')!(c')!} b'\Bigr)=0, $$
and the 
$ [k-j-b'-r+1,b'-1,c',b',r] $, $b'\ne0$, coefficient of $L^*f$ is
$$ {(-1)^r\over r!} {(k+2-j-r)_r\over(k+2d-1-r)_{r}}
\Bigl\{ {1\over(b')!(c')!}(b')-{1\over(b'-1)!(c'+1)!} (c'+1)\Bigr\}=0, $$
so that $f=P_{k-2j}^{(k)} \in \ker(L^*)\cap\ker(R^*)\cap H_k^{(j,j)}(\Hd)^Z$.
Since $P_{k-2j}^{(k)}\in H_k^{(j,j)}(\Hd)$, 
by Lemma \ref{RowandColMovementsLemma} and 
Lemma \ref{Rmapszonaltozonal},
we have the orthogonal direct sum decomposition
$$ \bigoplus_{0\le j\le {k\over2}}
\bigoplus_{j\le a,b\le k-j} \spam\{ L^{a-j} R^{b -j} P_{k-2j}^{(k)}\}
\subset \Harm_k(\Hd,\CC)^Z, $$
and by a dimension count using (\ref{HarmkZdim}),
we obtain (\ref{HarmkPonedimdecomp}), and hence 
the uniqueness of $P_{k-2j}^{(k)}$ up to a scalar multiple.
\end{proof}

It follows from Theorem \ref{HarmonicZonalconst} (also see \cite{BN02}) the 
zonal functions satisfy
$$ \dim ( (I(W_{k-2j})^{(k)})^Z) = (k-2j+1)^2, $$
$$ \dim ( H(k-b,b)_{k-2j}^Z) = k-2j+1, \quad j\le b\le k-j, $$
and for $q'=e_1$, we have
\begin{equation}
\label{1dimzonal}
\dim ( H(k-b,b)_{k-2j}\cap H_k^{(a,b)}(\Hd)^Z)
=\begin{cases}
1, & j\le a,b\le k-j; \cr
0, & \hbox{otherwise}.
\end{cases} 
\end{equation}

Let $Z_{k-2j,a,b}^{(k)}$ be the zonal polynomial with pole $q'$
obtained from $P_{k-2j,a,b}^{(k)}$ 
by making the substitution (\ref{e1toq'subs}).

\begin{corollary}
The zonal polynomials in $\Harm_k(\Hd,\CC)$
have the following orthogonal direct sum decomposition
into one-dimensional subspaces
\begin{equation}
\label{HarmkZPonedimdecomp}
\Harm_k(\Hd,\CC)^Z = \bigoplus_{0\le j\le {k\over2}}
\bigoplus_{j\le a,b\le k-j} \spam\{ Z_{k-2j,a,b}^{(k)}\}.
\end{equation}
\end{corollary}

\begin{proof}
Apply the substitution (\ref{e1toq'subs}) to
the orthogonal direct sum 
(\ref{HarmkPonedimdecomp}).
\end{proof}

The existence of the zonal polynomials $Z_{k-2j,a,b}^{(k)}$ 
in (\ref{HarmkPonedimdecomp}) is proved
inductively in \cite{BN02}, where they are denoted by $\cZ_{p,w,w'}^{(k)}$. 
We now outline how the two are related.
Here $p=k-2j$, and the ``weight'' parameters $w,w'$ are related to $(a,b)$, 
as follows
\begin{equation}
\label{ww'abidentification}
a={k-w'\over 2}, \quad b={k-w\over 2}, 
\qquad w'=k-2a, \quad w=k-2b,
\end{equation}
which gives the correspondence between indices
$$ (a,b)\in\{0,1,\ldots,k\}^2 \Iff (w,w')\in\{-k,-k+2,\ldots,k-2,k\}^2. $$ 
We note that for $k$ even (the case considered in \cite{BN02}) the
weights $w$ and $w'$ are even, and for $k$ odd, they are odd.
They define the space of zonal polynomials
\begin{align}
\label{Ewwpkdefn}
E_{w,w'}^{(k)} &:=\spam\{[\ga_1,\ga_2,\ga_3,\ga_4,r]_{q'}:\ga_1+\ga_2+\ga_3+\ga_4+2r=k, \cr
&\qquad\qquad\qquad \ga_1+\ga_2-\ga_3-\ga_4=w, \ga_1-\ga_2-\ga_3+\ga_4=w'\},
\end{align}
which satisfies
$$ 
E_{w,w'}^{(k)}=\Hom_k^{(a,b)}(\Hd)^Z, \qquad\hbox{for $q'=z'\in\Cd$}, $$
and the space 
$$ V_w^{(k)} = H\bigl({k+w\over2},{k-w\over2}\bigr) = H(k-b,b). $$
In \cite{BN02} (Proposition 4.5), the zonal polynomials $\cZ_{p,w,w'}^{(k)}$
are characterised by the following properties:
\begin{itemize}
\item $\cZ_{p,w,w'}\in E_{w,w'}^{(k)}$, i.e., $Z_{k-2j,a,b}^{(k)}$
has the structural form given by (\ref{Zk-2j(k)explicitform})
and (\ref{Zk-2jab(k)explicitform}).
\item  $\{\cZ_{p,w,w'}\}_{w'\in\{-p,\ldots,p-2,p\}}$ is a 
basis of (the zonal polynomials in) $I(W_p)^{(k)}\cap V_w^{(k)}$, i.e.,
$\{Z_{k-2j,a,b}^{(k)}\}_{j\le a\le k-j}$ is a basis 
of the zonal polynomials in $I(W_{k-2j})^{(k)}\cap H(k-b,b)$.
\item $\{\cZ_{p,w,w'}\}_{w\in\{-p,\ldots,p-2,p\}}$ is a basis of 
zonal polynomials for an irreducible subspace for right multiplication by $\HH^*$,
(which is isomorphic to $W_p$), i.e., 
$\{Z_{k-2j,a,b}^{(k)}\}_{j\le b\le k-j}$ is an $R$-orbit for a 
$W_{k-2j}$.
\end{itemize}

These follow from our construction, and the observation 
(by Lemma \ref{Rmapszonaltozonal}) that
$$  Z_{k-2j,a,b}^{(k)} = R^{b-j} Z_{k-2j,a,0}^{(k)}, \qquad j\le b\le k-j. $$

\begin{example}
The first three polynomials $Z_{k-2j}^{(k)}=Z_{k-2j,0,0}^{(k)}$ given by 
(\ref{Zk-2j(k)explicitform}) are
\begin{align*}
Z_k^{(k)} &=[k,0,0,0,0], \cr 
Z_{k-2}^{(k)} &= [k-2,1,0,1,0] + [k-1,0,1,0,0] -{k\over k+2d-2}[k-2,0,0,0,1], \cr
Z_{k-4}^{(k)} &= [k-4,2,0,2,0]
+[k-2,0,2,0,0]
+2[k-3,1,1,1,0] 
-{2(k-1)\over k+2d-2} [k-4,1,0,1,1] \cr
&\qquad -{2(k-1)\over k+2d-2} [k-3,0,1,0,1]
+ {(k-1)(k-2)\over(k+2d-2)(k+2d-3)} [k-4,0,0,0,2]. 
\end{align*}
We observe that, except for the first, these depend on the dimension $d$.
\end{example}

\begin{example} 
For $k=1$, the zonal polynomials in (\ref{HarmkZPonedimdecomp}) 
are
$$ Z_1^{(1)}:=[1,0,0,0,0]=z_1, \qquad
R Z_1^{(1)}=[0,0,0,1,0]=\overline{w_1}, $$
$$ L Z_1^{(1)}=[0,1,0,0,0]=w_1, \qquad
-LRZ_1^{(1)}=[0,0,1,0,0]=\overline{z_1}. $$
and for $k=2$, they are given by the schematic
$$ \mat{ & H(1,1)_0 \cr
K(1,1) & [1,0,1,0,0]+[0,1,0,1,0]-{1\over d}[0,0,0,0,1] } $$
$$ \mat{ & H(2,0)_2 & H(1,1)_2 & H(0,2)_2 \cr
K(2,0) & [2,0,0,0,0] & [1,0,0,1,0] & [0,0,0,2,0] \cr
K(1,1) & [1,1,0,0,0] & [0,1,0,1,0]-[1,0,1,0,0] & [0,0,1,1,0] \cr
K(0,2) & [0,2,0,0,0] & [0,1,1,0,0] & [0,0,2,0,0] } $$
with the indexing of rows and columns as before.
\end{example}


Similarly to the Schematic \ref{WeddingcakeRk-2j}, 
the summands $\{Z_{k-2j,a,b}\}_{j\le a,b\le k-j}$, $0\le j\le{k\over2}$, 
in (\ref{HarmkZPonedimdecomp}) 
can be arranged
as the layers of a ``wedding cake'' (see Figure \ref{Weddingfig}).

\def\nosides{\multicolumn{1}{>{$}c<{$}}}
\begin{table}[h]
\begin{center}
\begin{tabular}{|>{$}c<{$}|>{$}c<{$}|>{$}c<{$}|>{$}c<{$}|>{$}c<{$}|>{$}c<{$}|}
\nosides{} & \nosides{} & \nosides{} & \nosides{H(2,2)_0} & \nosides{}  & \nosides{} \\
\cline{4-4}
\nosides{K(2,2)} & \nosides{} & \multicolumn{1}{>{$}c<{$}|}{} & P_0=P_{000}^{(4)} & \nosides{} & \nosides{} \\
\cline{4-4}
\nosides{} & \nosides{} & \nosides{H(3,1)_2} & \nosides{H(2,2)_2} & \nosides{H(1,3)_2}  & \nosides{} \\
\cline{3-5}
\nosides{K(3,1)} & \multicolumn{1}{>{$}c<{$}|}{} & P_2=P_{200}^{(4)} & RP_2  & R^2P_2 & \nosides{} \\
\cline{3-5}
\nosides{K(2,2)} & \multicolumn{1}{>{$}c<{$}|}{} & LP_2 & LRP_2 & LR^2P_2 & \nosides{} \\
\cline{3-5}
\nosides{K(1,3)} & \multicolumn{1}{>{$}c<{$}|}{} & L^2P_2 & L^2RP_2 & L^2R^2P_2 & \nosides{} \\
\cline{3-5}
\nosides{} & \nosides {H(4,0)_4} & \nosides{H(3,1)_4} &
\nosides {H(2,2)_4} & \nosides {H(1,3)_4} & \nosides {H(0,4)_4}  \\
\cline{2-6}
\multicolumn{1}{>{$}c<{$}|}{K(4,0)} & P_4=P_{400}^{(4)} & RP_4 & R^2P_4 & R^3P_4 & R^4P_4 \\
\cline{2-6}
\multicolumn{1}{>{$}c<{$}|}{K(3,1)} & LP_4 & LRP_4 & LR^2P_4 & LR^3P_4 & LR^4P_4 \\
\cline{2-6}
\multicolumn{1}{>{$}c<{$}|}{K(2,2)} & L^2P_4 & L^2RP_4 & L^2R^2P_4 & L^2R^3P_4 & L^2R^4P_4 \\
\cline{2-6}
\multicolumn{1}{>{$}c<{$}|}{K(1,3)} & L^3P_4 & L^3RP_4 & L^3R^2P_4 & L^3R^3P_4 & L^3R^4P_4  \\
\cline{2-6}
\multicolumn{1}{>{$}c<{$}|}{K(0,4)} & L^4P_4 & L^4RP_4 & L^4R^2P_4 & L^4R^3P_4 & L^4R^4P_4  \\
\cline{2-6}
\end{tabular}
\end{center}
\caption{Schematic of the $1^2+3^2+5^2$ zonal functions for $\Harm_4(\Hd,\CC)$
	given by (\ref{Zk-2jab(k)explicitform}). } 
\label{Weddingfig}
\end{table}


We now seek an explicit formula for the zonal polynomial
$L^{\ga} R^{\gb} P_{k-2j}^{(k)}$ of Theorem \ref{HarmonicZonalconst}.
We first determine its structural form.
The Lemma \ref{LaRbstructformlemma}, below,
says that the complexity of the
formula depends on how far the index $(\ga,\gb)$ is from the edges of
the array of indices $A=\{0,1,\ldots,k-2j\}^2$. 
Partition $A$ into ``nested squares'', as in (\ref{Smdefn}),
\begin{equation}
\label{Skjmdefn}
S_{k,j,m}:=\{(\ga,\gb):\min\{j+\ga,j+\gb,k-j-\ga,k-j-\gb\}=m\}, \quad
j\le m \le {k\over2}.
\end{equation}


\begin{lemma}
\label{LaRbstructformlemma}
Let $0\le \ga,\gb \le k-2j$, $0\le j\le {k\over2}$, and
$$ m:=\min\{ j+\ga, k-j-\gb, j+\gb, k-j-\ga \} \qquad
	\hbox{i.e., $(\ga,\gb)\in S_{k,j,m}$}.  $$
Then $L^\ga R^\gb P_{k-2j}^{(k)} \in K(k-j-\ga,j+\ga) \cap H(k-j-\gb,j+\gb) $ 
has the form
$$ L^\ga R^\gb P_{k-2j}^{(k)} = \sum_{0\le r\le j\atop 0\le b \le m-r}
C_{br}^{(\ga,\gb)} [k-j-\gb-b-r,b,\ga+j-r-b,\gb-\ga+b,r], \quad \ga\le\gb, $$
$$ L^\ga R^\gb P_{k-2j}^{(k)} = \sum_{0\le r\le j\atop 0\le b \le m-r}
C_{br}^{(\ga,\gb)}` [k-j-\ga-b-r,\ga-\gb+b,j+\gb-b-r,b,r], \quad \ga\ge\gb, $$
which involves ${1\over2}(j+1)(2m+2-j)$ terms.
\end{lemma}

\begin{proof} A general zonal polynomial of degree $k$ has the form
$$ f=\sum_{a+b+c+d+2r=k} C_{abcdr}[a,b,c,d,r].$$
By Lemma \ref{Rmapszonaltozonal}, $L$ and $R$ applied to $[a,b,c,d,r]$ 
preserves the value of $r$, so that
$$f=L^\ga R^\gb P_{k-2j}^{(k)}$$ 
has the same restriction on $r$ 
as $P_{k-2j}^{(k)}$ does, i.e., $0\le r\le j$. 

The condition that
$f=L^\ga R^\gb P_{k-2j}^{(k)} 
\in K(k-j-\ga,j+\ga) \cap H(k-j-\gb,j+\gb)$ gives
\begin{align}
\label{HKcdn}
\begin{split}
 a+b+r &=k-j-\gb, \qquad c+d+r=j+\gb, \\
 a+d+r &=k-j-\ga, \qquad b+c+r=j+\ga.
\end{split}
\end{align}

First consider the case $\ga\le\gb$, i.e., $m=j+\ga$ or $m=k-j-\gb$.
If $m=j+\ga$, then (\ref{HKcdn}) gives
$$ b+c = m-r, \qquad a=k-j-\gb-b-r, \qquad d=j+\gb-c-r, $$
so that
$$ L^\ga R^\gb P_{k-2j}^{(k)} = \sum_{0\le r\le j\atop b+c=m-r}
C_{bcr} [k-j-\gb-b-r,b,c,j+\gb-c-r,r], \quad m=j+\ga. $$
Using $b+c+r=j+\ga$ to eliminate $c$ above, gives
\begin{equation}
\label{meqjplusalpha}
L^\ga R^\gb P_{k-2j}^{(k)} = \sum_{0\le r\le j\atop 0\le b \le m-r}
C_{br} [k-j-\gb-b-r,b,j+\ga-r-b,\gb-\ga+b,r].
\end{equation}
Now consider $m=k-j-\gb$. Then (\ref{HKcdn}) gives
$$ a+b=m-r, \quad a=k-j-\gb-b-r, \quad c=j+\ga-r-b,
\qquad d=\gb-\ga+b,$$
so that (\ref{meqjplusalpha}) holds for $m=j+\ga$ and 
$m=k-j-\gb$, i.e., $\ga\le\gb$. 

For the case $\ga\ge\gb$, i.e., $m=j+\gb$ or $m=k-j-\ga$, we have,
respectively
$$ c+d=m-r, \quad a=k-j-\ga-d-r, \quad
b=\ga-\gb+d, \quad c=j+\gb-d-r, $$
$$ a+d=m-r, \quad a=k-j-\ga-d-r, \quad
b=\ga-\gb+d, \quad c=j+\gb-d-r, $$
which (replacing $d$ by $b$) gives the second formula.

In the sum, we can have $r=0,1,\ldots,j$ ($j+1$ choices), with
$m+1-r$ choices for $b$, and so the number of terms is
$$ (m+1)+m+(m-1)+\cdots +(m+1-j) = {1\over2}(j+1)(2m+2-j). $$
\end{proof}

An explicit formula for $L^\ga R^\gb P_{k-2j}^{(k)}$ can by found 
by applying (\ref{Rzonal}) and (\ref{Lzonal}).
This gives very complicated coefficients. 
Instead, we used Lemma \ref{pablemma}
and numerous symbolic calculations for low values of $\ga$ and $\gb$,
such as Lemma \ref{RbetaZprop} below,
to conjecture the formulas of Theorems \ref{LaRbZformalphalessbeta}
and \ref{LaRbZgenform}, which were 
then proved for a general $(\ga,\gb)$.

\begin{lemma} 
\label{RbetaZprop}
For $0\le\gb\le k-2j$, $0\le j\le {k\over2}$, 
we have  
$$ R^\gb Z_{k-2j}^{(k)} 
= \sum_{b+c+r=j} {(-1)^r \over b!c!r!}
{(k+2-j-r)_r\over(k+2d-1-r)_{r}}(k-2j-\gb+1)_\gb
[k-2j-\gb+c,b,c,b+\gb,r]. $$
\end{lemma}

\begin{proof} 
Use induction on $0\le\gb\le k-2j$, with $\gb=0$ 
being trivial. Suppose that it holds for $\gb-1\ge0$, then
\begin{equation}
\label{RbetaZkform}
R^\gb Z_{k-2j}^{(k)} 
= \sum_{b+c+r=j} {(-1)^r \over b!c!r!}
{(k+2-j-r)_r\over(k+2d-1-r)_{r}}(k-2j-(\gb-1)+1)_{\gb-1}
R([\quad]), 
\end{equation}
where, 
by (\ref{Rzonal})
of Lemma \ref{Rmapszonaltozonal},
\begin{align*}
R([\quad]) &= R([k-2j-\gb+1+c,b,c,b+\gb-1,r]) \cr
&= (k-2j-\gb+1+c) [k-2j-\gb+c,b,c,b+\gb,r] \cr
& \qquad -b[k-2j-\gb+1+c,b-1,c+1,b+\gb-1,r].
\end{align*}
The $[k-2j-\gb+c',b',c',b'+\gb,r]$, $b'+c'+r=j$, coefficient in 
(\ref{RbetaZkform}) is
$$ {(-1)^r\over r!} 
{(k+2-j-r)_r\over(k+2d-1-r)_{r}}
(k-2j-\gb+2)_{\gb-1} \Bigl\{ \quad \Bigr\}, $$
where
\begin{align*}
\Bigl\{ \quad \Bigr\} 
&= {1\over(b')!(c'!)}(k-2j-\gb+1+c')-{1\over(b'+1)!(c'!-1)} ((b'+1) \cr
&= {1\over(b')!(c'!)}(k-2j-\gb+1+c'-c') 
= {1\over(b')!(c'!)} (k-2j-\gb+1),
\end{align*}
which completes the induction. 
\end{proof}


\begin{theorem} 
\label{LaRbZformalphalessbeta}
For $0\le \ga\le\gb \le k-2j$, $0\le j\le {k\over2}$, 
we have
$$ L^\ga R^\gb Z_{k-2j}^{(k)} 
= \sum_{0\le r\le j\atop 0\le b \le m-r}
C_{br}^{(\ga,\gb)} [k-j-\gb-b-r,b,j+\ga-r-b,\gb-\ga+b,r], $$
where $m=\min\{j+\ga,k-j-\gb\}$ and $ C_{br}^{(\ga,\gb)} = A_{br}^{(\ga,\gb)} B_{br}^{(\ga,\gb)}$, with
\begin{equation}
\label{Aalphabeta}
A_{br}^{(\ga,\gb)}:={(-1)^r \over b!(j+\ga-b-r)!r!}
{(k+2-j-r)_r\over(k+2d-1-r)_{r}}(k-2j-\gb+1)_\gb,
\end{equation}
\begin{align}
\label{Balphabetahyp}
B_{br}^{(\ga,\gb)} : &= \sum_{u+v=\ga}
 {\alpha!\over u!v!}(k-2j-\alpha+1)_{u}(b-u+1)_{u}
(j-r+1)_{v}(-\gb)_{v} \cr
&= (1+j-r)_\ga (-\gb)_\ga\,
{}_3F_2\Bigl(\mat{-\ga,-b,k-2j-\ga+1\cr r-j-\ga,\gb+1-\ga};1\Bigl).
\end{align}
The constant $B_{br}^{(\ga,\gb)}$ can also be calculated from 
$B_{br}^{(0,\gb)}:=1$ and the recurrence
\begin{equation}
\label{Balphabetarecurr}
B_{br}^{(\ga,\gb)}
= (k-j-\gb-b+1-r) b B_{b-1,r}^{(\ga-1,\gb)}
-(\gb-\ga+1+b) (j+\ga-b-r) B_{br}^{(\ga-1,\gb)}.
\end{equation}
\end{theorem}

\begin{proof} We first prove the result for $B_{br}^{(\ga,\gb)}$ given 
by the recurrence relation (\ref{Balphabetarecurr}),
by using induction on $\ga$. This is true for $\ga=0$
and all $\gb$ by Lemma \ref{RbetaZprop} (where $m=j$).
Let $A_{-1,r}^{(\ga,\gb)}$ and $B_{-1,r}^{(\ga,\gb)}$ take some value
(it matters not which). Then we have
\begin{equation}
\label{Arecurrence}
A_{b-1,r}^{(\ga-1,\gb)} = b A_{br}^{(\ga,\gb)}, \quad
A_{br}^{(\ga-1,\gb)} = (j+\ga-b-r) A_{br}^{(\ga,\gb)}, \qquad \ga>0.
\end{equation}
Suppose that $\ga>0$, then by the inductive hypothesis, we have
$$ L^{\ga-1} R^\gb Z_{k-2j}^{(k)} 
= \sum_{0\le r\le j\atop 0\le b' \le m-r}
C_{b'r}^{(\ga-1,\gb)} [k-j-\gb-b'-r,b',j+\ga-1-r-b',\gb-\ga+1+b',r]. $$
We apply $L$ to this, using (\ref{Lzonal}), i.e.,
\begin{align*}
& L([k-j-\gb-b'-r,b',j+\ga-1-r-b',\gb-\ga+1+b',r]) \cr
& \quad =(k-j-\gb-b'-r)[k-j-\gb-b'-r-1,b'+1,j+\ga-1-r-b',\gb-\ga+1+b',r] \cr
& \qquad -(\gb-\ga+1+b')[k-j-\gb-b'-r,b',j+\ga-r-b',\gb-\ga+b',r]
\end{align*}
and (\ref{Arecurrence}), to obtain
\begin{align*} 
C_{br}^{(\ga,\gb)} 
&= (k-j-\gb-(b-1)-r) C_{b-1,r}^{(\ga-1,\gb)}
-(\gb-\ga+1+b) C_{br}^{(\ga-1,\gb)} \cr
&= (k-j-\gb-b+1-r) b A_{br}^{(\ga,\gb)} B_{b-1,r}^{(\ga,\gb)}
-(\gb-\ga+1+b) (j+\ga-b-r) A_{br}^{(\ga,\gb)} B_{br}^{(\ga-1,\gb)}.
\end{align*}
Since $A_{br}^{(\ga,\gb)}\ne0$, we may divide the above by it,
to obtain
$$ B_{br}^{(\ga,\gb)}
= (k-j-\gb-b+1-r) b B_{b-1,r}^{(\ga-1,\gb)}
-(\gb-\ga+1+b) (j+\ga-b-r) B_{br}^{(\ga-1,\gb)}, $$
i.e., (\ref{Balphabetarecurr}), which completes the induction.

Finally, we show that the formula (\ref{Balphabetahyp}) 
for $B_{br}^{(\ga,\gb)}$
involving a ${}_3F_2$ hypergeometric series holds, 
i.e., it satisfies the recurrence. This we do by induction on $\ga$.
The case $\ga=0$ is immediate, and the inductive step follows 
from the contiguous relation
\begin{align*}
 (de)\, {}_3F_2\Bigl(\mat{-n,a,c\cr d,e};1\Bigl) 
& = (a+c-d-e+1-n)(-a)\, {}_3F_2\Bigl(\mat{1-n,a+1,c+1\cr d+1,e+1};1\Bigl) \cr
& \qquad -(e-a)(a-d)\, {}_3F_2\Bigl(\mat{1-n,a,c+1\cr d+1,e+1};1\Bigl).
\end{align*}
for hypergeometric functions, for the choice
$$ n=\ga, \quad a=-b, \quad c=k-2j-\ga+1, \quad d=r-j-\ga, \quad
e=\gb+1-\ga. $$
\end{proof}

The recurrence relation (\ref{Balphabetarecurr}) was determined first. 
It suggests that $(b,\gb)\mapsto B_{br}^{\ga,\gb}$ is a polynomial 
of degree $2\ga$, where in fact it has degree $\ga$, as is indicated 
by (\ref{Balphabetahyp}). We could not prove formula  (\ref{Balphabetahyp})
directly, without recourse to the contiguous relation. To indicate
the complexity of such a calculation, 
we give the inductive step 
for $\ga=1,2$
%
%
%
\begin{align*}
B_{br}^{(1,\gb)} &= (k-j-\gb-b+1-r) b -(\gb+b) (j+1-b-r)
=(k-2j)b-(j-r+1)\gb, \cr
B_{br}^{(2,\gb)}
& = (k-j-\gb-b+1-r) b \{(k-2j)(b-1)-(j-r+1)\gb\} \cr
&\qquad -(\gb-1+b) (j+2-b-r) \{(k-2j)b-(j-r+1)\gb\} \cr
& = (j-r+1)_2(-\gb)_2
+2(k-2j-1)b(j-r+1)(-\gb)
+(k-2j-2)_2(b-1)_2.
\end{align*}

\begin{example}
For $r=j$, the hypergeometric series in (\ref{Balphabetahyp}) can be summed 
using the generalised binomial theorem, or Gauss's summation
for the resulting ${}_2F_1$, to obtain
\begin{align*}
B_{bj}^{(\ga,\gb)}
& = \sum_{u=0}^b {\alpha!\over u!v!}(k-2j-\alpha+1)_{u}
{b!\over(b-u)!} v! (-\gb)_{\ga-b}(-\gb+\ga-b)_{b-u} \cr
&= \ga! (-\gb)_\ga { (\gb-k+2j)_b\over (\gb+1-\ga)_b },
\end{align*}
$$ C_{bj}^{(\ga,\gb)} 
= {(-1)^j \over j!} {(k+2-2j)_j\over(k+2d-1-j)_{j}}(k-2j-\gb+1)_\gb
 (-\gb)_\ga 
{(-1)^b\over b!} { (-\ga)_b (\gb-k+2j)_b\over (\gb+1-\ga)_b }.$$
For the case $j=0$, this further reduces to
$$ C_{b0}^{(\ga,\gb)} 
= (k-\gb+1)_\gb (-\gb)_\ga 
{(-1)^b\over b!} { (-\ga)_b (\gb-k)_b\over (\gb+1-\ga)_b },$$
and we recover the Lemma \ref{pablemma} 
as the particular case
$j=0$ and $d=1$.
\end{example}

The case $\ga\ge\gb$ can easily be obtained in a similar way to 
Theorem \ref{LaRbZformalphalessbeta}.

\begin{theorem}
\label{LaRbZgenform}
For $0\le \ga,\gb \le k-2j$, $0\le j\le {k\over2}$, let
$$ m:=\min\{j+\ga,k-j-\gb,j+\gb,k-j-\ga\}, \qquad c:=\min\{\ga,\gb\}. $$
Then
we have
\begin{equation}
\label{LaRbZgenformula}
L^\ga R^\gb Z_{k-2j}^{(k)} 
= \sum_{0\le r\le j\atop 0\le b \le m-r}
C_{br}^{(\ga,\gb)} [k-j+c-\ga-\gb-b-r,\ga-c+b,j+c-b-r,\gb-c+b,r],
\end{equation}
where
and $C_{br}^{(\ga,\gb)}:=A_{br}^{(\ga,\gb)}B_{br}^{(\ga,\gb)}$ is
given by (\ref{Aalphabeta}) and (\ref{Balphabetahyp}) for $\ga\le\gb$, 
and by 
$$ A_{br}^{(\ga,\gb)}:=A_{br}^{(\gb,\ga)}, \quad
B_{br}^{(\ga,\gb)}:=B_{br}^{(\gb,\ga)}, \qquad \ga\ge\gb. $$
For $\gb\ge\ga$, the constant $B_{br}^{(\ga,\gb)}$ can be calculated from
$B_{br}^{(\ga,0)}:=1$ and the recurrence
\begin{equation}
\label{Bbetalessalpharecurr}
B_{br}^{(\ga,\gb)} 
= (k-j-\ga-b+1-r) b B_{b-1,r}^{(\ga,\gb-1)}
-(\ga-\gb+1+b) (j+\gb-b-r) B_{br}^{(\ga,\gb-1)}.
\end{equation}
\end{theorem}

\begin{proof} In light of Theorem \ref{LaRbZformalphalessbeta},
we need only consider the case $\ga\ge\gb$. 
It follows from (\ref{Balphabetarecurr}) that 
$B_{br}^{(\ga,\gb)}:=B_{br}^{(\gb,\ga)}$, $\ga\ge\gb$,
satisfies (\ref{Bbetalessalpharecurr}). 
The formula (\ref{LaRbZgenformula}) can be proved as in Theorem \ref{LaRbZformalphalessbeta},
using induction on $\gb$ and by applying
$R$ to $B_{br}^{(\ga,\gb-1)}$.
\end{proof}

We now consider an alternative formula for the zonal polynomial of (\ref{Zk-2j(k)explicitform}),
i.e.,
$$ P_{k-2j}^{(k)} := \sum_{b+c+r=j} {(-1)^r \over b!c!r!}
{(k+2-j-r)_r\over(k+2d-1-r)_{r}}[k-j-b-r,b,c,b,r], $$
where 
\begin{align*}
[k-j-b-r,b,c,b,r]
&= z_1^{k-j-b-r}w_1^b\overline{z_1}^c\overline{w_1}^b\norm{z+jw}^{2r}
&= z_1^{k-2j} |z_1|^{2c} |w_1|^{2b}\norm{z+jw}^{2r}.
\end{align*}
By the binomial identity, we have
\begin{align*}
P_{k-2j}^{(k)} &=  \sum_{r=0}^j 
{(-1)^r \over r!} {(k+2-j-r)_r\over(k+2d-1-r)_{r}}
z_1^{k-2j} \norm{z+jw}^{2r}
{1\over (j-r)!}
\sum_{b+c=j-r} {(b+c)!\over b!c!} |z_1|^{2c} |w_1|^{2b} \cr 
&=  z_1^{k-2j} \sum_{r=0}^j {(-1)^r \over r!} {(k+2-j-r)_r\over(k+2d-1-r)_{r}}
\norm{z+jw}^{2r} {1\over (j-r)!} (|z_1|^2+|w_1|^2)^{j-r}. 
\end{align*}
This has the structural form
\begin{equation}
\label{zonalstructform}
P_{k-2j}^{(k)} = z_1^{k-2j} F(\norm{z+jw}^{2},|z_1|^2+|w_1|^2),
\end{equation}
where $F$ is a homogeneous polynomial of degree $j$ with real coefficients.
An elementary calculation shows that the polynomial
$$ |\inpro{z+jw,e_1}|^2= |z_1|^2+|w_1|^2 = z_1\overline{z_1}+w_1\overline{w_1} $$
is in the kernel of $R$, $R^*$, $L$ and $L^*$, e.g.,
$$ R(z_1\overline{z_1}+w_1\overline{w_1}) 
=\overline{w_1}\overline{z_1}-\overline{z_1}\overline{w_1}=0. $$
Hence, by (\ref{RLproductrule}), all polynomials of the form
\begin{equation}
\label{LRkerform}
g= G(z_1\overline{z_1}+w_1\overline{w_1}, \cdots,z_d\overline{z_d}+w_d\overline{w_d}),
\end{equation}
which include $\norm{z+jw}^{2}$ and $|z_1|^2+|w_1|^2$,
are in the kernel of $R$, $R^*$, $L$ and $L^*$,
and hence
\begin{equation}
\label{kerfactoraction}
T(fg) = T(f)g+f T(g)=T(f)g, \qquad T=R,R^*,L,L^*.
\end{equation}
Applying this to (\ref{zonalstructform}) gives the following.

\begin{theorem} 
Let $q'=e_1$.  For $d\ge2$, $0\le j\le {k\over 2}$, 
the zonal polynomials of Theorem \ref{HarmonicZonalconst}
are given by
\begin{equation}
\label{nicezonalformula}
L^\ga R^\gb P_{k-2j}^{(k)} = L^\ga R^\gb (z_1^{k-2j}) F, 
\qquad 0\le\ga,\gb\le k-2j,
\end{equation}
where 
$F=F(\norm{z+jw}^{2},|z_1|^2+|w_1|^2)$	does not depend on $\ga$ and $\gb$, and is given by 
\begin{align}
\label{Fformula}
F &= {(-1)^j (k-2j+2)_j \over(k+2d-1-j)_j j!}
\sum_{s=0}^j (-j)_s {(k-2j+2d-1+j)_s \over (k-2j+2)_s} {1\over s!}
\norm{z+jw}^{2(j-s)} \bigl(|z_1|^2+|w_1|^2\bigr)^s \cr
&= {(-1)^j \over(k+2d-1-j)_j} \norm{z+jw}^{2j}
P_j^{(k-2j+1,2d-3)}\Bigl(1-2 {|z_1|^2+|w_1|^2\over \norm{z+jw}^2}\Bigr),
\end{align}
with $P_j^{(k-2j+1,2d-3)}$ a Jacobi polynomial.
\end{theorem}

\begin{proof} Since the function $F$ of (\ref{zonalstructform}) is of the form 
(\ref{LRkerform}), 
we may apply (\ref{kerfactoraction}) repeatedly, to obtain
$$ L^\ga R^\gb P_{k-2j}^{(k)} = L^\ga R^\gb\bigl(z_1^{k-2j}\bigr)
	F(\norm{z+jw}^{2},|z_1|^2+|w_1|^2), $$ 
where $F$ does not depend on $\ga$ and $\gb$, and is given by
$$ F= {P_{k-2j}^{(k)}\over z_1^{k-2j}}
=  \sum_{r=0}^j {(-1)^r \over r!} {(k+2-j-r)_r\over(k+2d-1-r)_{r}}
\norm{z+jw}^{2r} {1\over (j-r)!} (|z_1|^2+|w_1|^2)^{j-r}.  $$
By making the change of variables $s=j-r$, we obtain
\begin{align*}
F &= {(-1)^j (k-2j+2)_j \over(k+2d-1-j)_j j!} 
\sum_{s=0}^j (-j)_s {(k-2j+2d-1+j)_s \over (k-2j+2)_s} {1\over s!} 
\norm{z+jw}^{2(j-s)} \bigl(|z_1|^2+|w_1|^2\bigr)^s \cr
&= {(-1)^j (k-2j+2)_j \over(k+2d-1-j)_j j!} \norm{z+jw}^{2j}
\sum_{s=0}^j (-j)_s {(k-2j+2d-1+j)_s \over (k-2j+2)_s} {1\over s!} 
\Bigl({|z_1|^2+|w_1|^2\over \norm{z+jw}^2}\Bigr)^s,
\end{align*}
so that $F$ can be expressed in terms of a Jacobi polynomial, i.e..
$$ F = {(-1)^j \over(k+2d-1-j)_j} z_1^{k-2j} \norm{z+jw}^{2j}
P_j^{(k-2j+1,2d-3)}\Bigl(1-2 {|z_1|^2+|w_1|^2\over \norm{z+jw}^2}\Bigr). $$
\end{proof}

The formula (\ref{Fformula}) for the zonal polynomial (reproducing kernel) 
$P_{k-2j}^{(k)}= z_1^{k-2j} F$, involving the Jacobi polynomial, 
appears in \cite{BSW17} (Theorem 8). An explicit formula for the 
factor $L^\ga R^\gb(z_1^{k-2j})$ is given by the formula
(\ref{pabformula}) for the univariate case (replace $z$ by $z_1$, etc).

By writing the zonal polynomials in the form (\ref{nicezonalformula}),
the squares in the table/schematic for the zonal polynomials 
(see Figure \ref{Weddingfig}) become essentially those
for the univariate cases $\Harm_{k-2j}(\HH,\CC)$, as
depicted in (\ref{univariatesquaresimple}).

\section{Symmetries}

The polynomials $L^\ga R^\gb P_{k-2j}^{(k)}$
and the spaces $H_k^{(\ga,\gb)}(\Hd)$
have certain natural symmetries that correspond to the
symmetries of the square (see the array in Example \ref{exshowsymms}).

Let the permutation group $\Sym(4)$ act on functions of four
variables in the natural way, i.e.,  
$$\gs\cdot f(x_1,x_2,x_3,x_4) = f(x_{\gs 1},x_{\gs 2},x_{\gs 3},x_{\gs 4}), $$
and hence on functions 
$f(z,w,\overline{z},\overline{w})\in\Hom_k(\Hd,\CC)$. 
There is a subgroup $G$ of $\Sym(4)$ which maps $\Harm_k(\Hd,\CC)$ to itself,
which is generated by the permutations
\begin{equation}
\label{sigmataudefn}
\gs:=(24), \quad \tau:=(14)(23).
\end{equation}
This group is the dihedral group of symmetries of the square
$$ D_4=\inpro{a,b|a^4=b^2=(ba)^2=1}, \qquad
a=\gs\tau,\quad b=\gs, $$
and hence has order eight.
By considering the action on monomials, one obtains
\begin{equation}
\label{gstauonHk(a,b)}
\gs\cdot H_k^{(\ga,\gb)}(\Hd,\CC) = H_k^{(\gb,\ga)}(\Hd,\CC), \qquad
\tau\cdot H_k^{(\ga,\gb)}(\Hd,\CC) = H_k^{(\ga,k-\gb)}(\Hd,\CC),
\end{equation}
and so $G$ permutes the subspaces $H_k^{(a,b)}(\Hd,\CC)$, where
$$(a,b)\in \{
(\ga,\gb),(\ga,k-\gb),(k-\ga,\gb),(k-\ga,k-\gb),
(\gb,\ga),(\gb,k-\ga-k),(k-\gb,\ga),(k-\gb,k-\ga) \}, $$
via the action on the indices given by
$$ \gs\cdot (\ga,\gb):=(\gb,\ga), \qquad
\tau\cdot (\ga,\gb):=(\ga,k-\gb). $$
It clear from Lemma \ref{Habdecomplemma} that these subspaces do indeed have 
the same dimension. The number of subspaces above can be
$1,2,4,8$, depending on
the position of the index $(\ga,\gb)$ in the square array $\{0,1,\ldots,k\}^2$.
The corresponding symmetries of the zonal polynomials
\begin{equation}
\label{PinHk(ab)}
L^\ga R^\gb P_{k-2j}^{(k)}\in H_k^{(\ga+j,\gb+j)}(\Hd,\CC), 
\qquad 0\le\ga,\gb,\le k-2j,
\end{equation}
are as follows.

\begin{lemma} 
\label{eightsymmslemma}
(Eight symmetries) The zonal polynomials 
$(L^\ga R^\gb P_{k-2j}^{(k)})_{0\le\ga,\gb,\le k-2j}$ 
have the following basic symmetries corresponding to the
	$\gs$ and $\tau$ of (\ref{sigmataudefn})
\begin{align}
\label{Pkbasicsymms}
L^\ga R^\gb P_{k-2j}^{(k)} (z,w,\overline{z},\overline{w})
&= L^\gb R^\ga P_{k-2j}^{(k)} (z,\overline{w},\overline{z},w)
	\qquad \hbox{($\gs$)} \cr
&= c_{\ga,\gb} L^\ga R^{k-2j-\gb} P_{k-2j}^{(k)} 
( \overline{w},\overline{z},w,z) \qquad \hbox{($\tau$)},
\end{align}
where $c_{\ga,\gb}$ is a constant. The 
	identities
for the remaining nontrivial elements of $G$ are
\begin{align} 
\label{Pkfivesymms}
	L^\ga R^\gb P_{k-2j}^{(k)} & (z,w,\overline{z},\overline{w}) \cr
& = c_{\ga,\gb} L^{k-2j-\gb} R^{\ga} 
P_{k-2j}^{(k)} (\overline{w},z,w,\overline{z}) \qquad
\hbox{($\gs\tau$)}  \cr 
& = c_{\gb,\ga} L^\gb R^{k-2j-\ga} P_{k-2j}^{(k)} 
(w,\overline{z},\overline{w},z) \qquad \hbox{($\tau\gs$)} \cr 
& = c_{\gb,\ga} L^{k-2j-\ga} R^\gb P_{k-2j}^{(k)} 
(w,z,\overline{w},\overline{z}) \qquad \hbox{($\gs\tau\gs$)} \cr 
& = c_{\ga,\gb}c_{k-2j-\gb,\ga} L^{k-2j-\gb} R^{k-2j-\ga} 
P_{k-2j}^{(k)} (\overline{z}, w ,z, \overline{w}) 
\qquad \hbox{($\tau\gs\tau$)} 
\cr
& = c_{\ga,\gb}c_{k-2j-\gb,\ga} L^{k-2j-\ga} R^{k-2j-\gb} 
Z_{k-2j}^{(k)} (\overline{z},\overline{w},z,w) 
\qquad \hbox{($\gs\tau\gs\tau$)}.
\end{align} 
\end{lemma}

\begin{proof} 
The permutations $\gs$ and $\tau$ map zonal polynomials to zonal polynomials,
and so, in light of (\ref{gstauonHk(a,b)}) and (\ref{PinHk(ab)},
we obtain (\ref{Pkbasicsymms}). This could also be established, 
with values of the constants, from the formulas for 
$L^\ga R^\gb P_{k-2j}^{(k)}$ given in Theorem \ref{LaRbZformalphalessbeta}, 
or by using identities such as
	$$ \gs \cdot (Rf) = Lf, \quad \gs \cdot (Lf) = Rf, \qquad
\tau\cdot (Rf)=R^*f, \quad \tau\cdot (Lf)=-Lf, $$
together with Lemma \ref{RR*commutegeneral}.
\end{proof}


The formulas in Lemma \ref{eightsymmslemma} 
for $P_{k-2j,a,b}^{(k)}=L^{a-j} R^{b-j} P_{k-2j}^{(k)}$, $\ga=a-j$, $\gb=b-j$, 
do not have
a simple formula for the constants, as the normalisation of 
$L^\ga R^\gb P_{k-2j}^{(k)}$ is biased towards the (starting) polynomial
$P_{k-2j}^{(k)}\in H_k^{(j,j)}(\Hd,\CC)$, which corresponds to a corner 
of the array of indices.
For $k$ even, one could start with the ``centre'' polynomial
$$ C_{k-2j}^{(k)} = P_{k-2j,{k\over2},{k\over2}}^{(k)}
=L^{{k\over2}-j}R^{{k\over2}-j} P_{k-2j}^{(k)} \in H_k^{({k\over2},{k\over2})}
(\Hd,\CC), $$
to obtain zonal polynomials
$$ 
  L^{\max\{\ga,{k\over2}-j\}}
  (L^*)^{\max\{{k\over2}-j-\ga,0\}}
  R^{\max\{\gb,{k\over2}-j\}}
  (R^*)^{\max\{{k\over2}-j-\gb,0\}}
C_{k-2j}^{(k)}. $$
By using Lemma \ref{RR*commutegeneral}, these can be written as
$$ P_{k,k-2j}^{(\ga,\gb)} :=
 {M_\ga!\over \ga!}
 {M_\gb!\over \gb!}
R^\ga L^\gb P_{k-2j}^{(k)},  $$
where
$$ M_\ga:=\max\{\ga,k-2j-\ga\}, \qquad
M_\gb:=\max\{\gb,k-2j-\ga\}. $$
The $\gs$ and $\tau$ symmetries of Lemma \ref{eightsymmslemma} then become
\begin{align}
\label{Pkbasicsymmsnice}
P_{k,k-2j}^{(\ga,\gb)} (z,w,\overline{z},\overline{w})
&= P_{k,k-2j}^{(\gb,\ga)} (z,\overline{w},\overline{z},w)
        \qquad \hbox{($\gs$)} \cr
	&= (-1)^\ga P_{k,k-2j}^{(\ga,k-2j-\gb)}
( \overline{w},\overline{z},w,z) \qquad \hbox{($\tau$)}.
\end{align}

\section{The fine scale decomposition for left and right multiplication by $\HH^*$}

By taking the intersection of the decomposition into
irreducibles for right multiplication by $\HH^*$ (Theorem \ref{HarmkIWdecomp})
with the corresponding one for left multiplication,
we obtain the following decomposition into low dimensional subspaces.
All of our decompositions, and others, can be built up from this.

\begin{theorem} 
\label{LRorthdecomp} (Fine scale decomposition)
Let
$$ V_k^{(j_1,j_2)}(\Hd) 
:=\ker L^* \cap \ker R^* \cap H_{k}^{(j_1,j_2)}(\Hd),\qquad
0\le j_1,j_2\le{k\over2}. $$
Then for $d\ge2$, we have the orthogonal direct sum
\begin{align}
\label{HomkVdecomp}
\Hom_k(\Hd,\CC)
&= \bigoplus_{0\le j_1,j_2\le {k\over2}} 
\bigoplus_{ j_1\le a\le k-j_1\atop j_2\le b\le k-j_2 } 
\bigoplus_{i=0}^{\min\{j_1,j_2\}}
\norm{\cdot}^{2i} L^{a-j_1} R^{b-j_2} V_{k-2i}^{(j_1-i,j_2-i)}(\Hd),
\end{align}
and in particular
\begin{equation}
\label{HarmkVdecomp}
\Harm_k(\Hd,\CC)
= \bigoplus_{0\le j_1,j_2\le {k\over2}} 
\bigoplus_{ j_1\le a\le k-j_1\atop j_2\le b\le k-j_2 } 
 L^{a-j_1} R^{b-j_2} V_{k}^{(j_1,j_2)}(\Hd),
\end{equation}
where
$$ \dim(L^{a-j_1} R^{b-j_2} V_{k}^{(j_1,j_2)}(\Hd))
=\dim(V_{k}^{(j_1,j_2)}(\Hd)). $$
\end{theorem}

\begin{proof} 
We note that for $j\le {k\over2}$, $j\le k-j$, so that
$\min\{j,k-j\}=j$, and
\begin{equation}
\label{HomHFischer}
\Hom_H(k-j,j)
= \bigoplus_{i=0}^j \norm{\cdot}^{2i} H(k-j-i,j-i).
\end{equation}
We observe that Lemma \ref{Rmapszonaltozonal}
implies multiplication of polynomials by $\norm{\cdot}^2$ 
commutes with the action of $R,L,R^*,L^*$. 
Thus from (\ref{HomHFischer}), we obtain
\begin{align*}
\Hom_H(k-j,j)_{k-2j}
&:=\ker R^*\cap \Hom_H(k-j,j) \cr
&= \bigoplus_{i=0}^j \norm{\cdot}^{2i} (\ker R^*\cap  H(k-j-i,j-i)),
\end{align*}
so that Lemma \ref{HomkRdecomp} gives the orthogonal direct sum decomposition
$$ \Hom_k(\Hd,\CC)= \bigoplus_{0\le j\le {k\over2}} \bigoplus_{j\le b\le k-j} 
 \bigoplus_{i=0}^j \norm{\cdot}^{2i}
R^{b-j} (\ker R^* \cap H(k-j-i,j-i)). $$
Similarly, we obtain the orthogonal direct sum decomposition
$$ \Hom_k(\Hd,\CC) = \bigoplus_{0\le j\le {k\over2}} \bigoplus_{j\le a\le k-j} 
 \bigoplus_{i=0}^j \norm{\cdot}^{2i}
L^{a-j} (\ker L^* \cap K(k-j-i,j-i)). $$
Thus $\Hom_k(\Hd,\CC)$ is an orthogonal direct sum of subspaces
$$ \norm{\cdot}^{2i_1}
 L^{a-j_1} (\ker L^* \cap K(k-j_1-i_1,j_1-i_1))
\cap \norm{\cdot}^{2i_2}
 R^{b-j_2} (\ker R^* \cap H(k-j_2-i_2,j_2-i_2)). $$
In view of the uniqueness of the Fischer decomposition, these 
can be nonzero only if $i_1=i_2=i\le\min\{j_1,j_2\}$. 
Since $L,R$ and $\norm{\cdot}^2$ commute,
the intersection above can be written
$$ \norm{\cdot}^{i} L^{a-j_1} R^{b-j_2}
 (\ker L^* \cap \ker R^*
\cap K(k-j_1-i_1,j_1-i_1))
 \cap H(k-j_2-i_2,j_2-i_2)), $$
which gives (\ref{HomkVdecomp}), with the $i=0$ terms giving
(\ref{HarmkVdecomp}). The dimension formula follows from
Lemma \ref{RowandColMovementsLemma}
\end{proof}

Theorem \ref{LRorthdecomp} also holds for $d=1$, in a 
degenerate way, with 
$$ V_k^{(j_1,j_2)}(\HH)=0, \qquad (j_1,j_2)\ne(0,0). $$

\begin{corollary}
The decomposition of zonal polynomials for $Z=U(\Hd)_{q'}$,
$q'=z'\in\Cd$,
corresponding to (\ref{HarmkVdecomp}) is
\begin{equation}
\label{HarmkVdecompZonal}
\Harm_k(\Hd,\CC)^Z
= \bigoplus_{0\le j\le {k\over2}} 
\bigoplus_{ j\le a,b\le k-j} 
	\bigl(L^{a-j} R^{b-j} V_{k}^{(j,j)}(\Hd)\bigr)^Z, 
\end{equation}
where
$$ \dim(L^{a-j} R^{b-j} V_{k}^{(j,j)}(\Hd)^Z)=1. $$
Moreover, we have
\begin{equation}
\label{HkabVdecom}
H_k^{(a,b)}(\Hd)
= \bigoplus_{0\le j_1\le m_a^{(k)}\atop 0\le j_2\le m_b^{(k)}}
 L^{a-j_1} R^{b-j_2} V_{k}^{(j_1,j_2)}(\Hd).
\end{equation}
\end{corollary}

\begin{proof}
The decomposition (\ref{HarmkVdecompZonal}) is given in 
Theorem \ref{HarmonicZonalconst}, and the decomposition
(\ref{HkabVdecom}) follows
from (\ref{HarmkVdecomp}) by grouping the terms
$L^{a-j_1} R^{b-j_2} V_{k}^{(j_1,j_2)}(\Hd)\in H_k^{(a,b)}(\Hd)$.
\end{proof}


The dimension of $V_k^{(a,b)}(\Hd)$ is as follows
(see \S\ref{appendixnumber} for the proof).

\begin{lemma}
\label{dimVkablemma}
For $0\le a,b\le{k\over2}$, we have that
\begin{align}
\label{dimVkabF}
\dim(V_k^{(a,b)}(\Hd)) 
&= F(k,m,M,d)+F(k,m-1,M-1,d) \cr
&\quad-F(k,m-1,M,d)-F(k,m,M-1,d),
\end{align}
where $F$ is given by (\ref{FkmMddef}), and $m=\min\{a,b\}$, $M=\max\{a,b\}$.
In particular,
$$ \dim(V_k^{(a,b)}(\HH^2)) = (m+1)(k-2M+1).  $$
The zonal polynomials in $V_k^{(a,b)}(\Hd)^{U_q}$ is given by
$$ \dim(V_k^{(a,b)}(\Hd)^{U_q}) = 
\begin{cases}
1, & a=b; \cr
0, & a\ne b.
\end{cases}
$$
\end{lemma}



For $d\ge2$, it follows from Lemma \ref{dimVkablemma} that
all the summands in (\ref{HomkVdecomp}) and (\ref{HarmkVdecomp}) are nontrivial.

\begin{example} We have
$$ V_k^{(0,0)}(\Hd)=H_k^{(0,0)}(\Hd)=\spam\{z^\ga:|\ga|=k\}, \quad
\dim(V_k^{(0,0)}(\Hd)) = {k+d-1\choose d-1}. $$
For $k=2$, $d=2$,
$V_2^{(0,0)}(\HH^2)=\spam\{z_1^2,z_1z_2,z_2^2\}$, and
$$ V_2^{(0,1)}(\HH^2)=\spam\{z_1\overline{w_2}-z_2\overline{w_1}\}, \quad
V_2^{(1,0)}(\HH^2)=\spam\{z_1{w_2}-z_2{w_1}\}, $$
$$ V_2^{(1,1)}(\HH^2)=\spam\{
z_1\overline{z_1}+w_1\overline{w_1}
-z_2\overline{z_2}-w_2\overline{w_2},
z_1\overline{z_2} + \overline{z_1}z_2
+w_1\overline{w_2} + \overline{w_1}w_2
\}. $$
Here we can see explicitly, that the zonal polynomials for $Z=U(\HH^2)_{e_1}$ are 
given by 
$$ V_2^{(0,0)}(\HH^2)^Z=\spam\{z_1^2\}, \quad V_2^{(0,1)}(\HH^2)^Z=0, \quad
V_2^{(1,0)}(\HH^2)^Z=0, $$
$$ V_2^{(1,1)}(\HH^2)^Z
=\spam\{2(z_1\overline{z_1}+w_1\overline{w_1})-\norm{(z,w)}^2\}. $$
The decomposition of (\ref{HarmkVdecompZonal}) involves subspaces of
dimensions $3$ (nine), $2$ (one) and $1$ (six), with 
$$ \dim(\Harm_2(\HH^2))=9\cdot 3+1\cdot2+6\cdot1=35. $$
\end{example}


\section{Conclusion}

The fine scale decomposition of Theorem \ref{LRorthdecomp} refines all the decompositions 
of the harmonic polynomials $\Harm_k(\Hd,\CC)$ under the action of a group 
$G\subset U(\RR^{4d})$ that we have given or described. 
These can be summarised as follows:
$$ \mat{ U(\RR^{4d})\cr |\cr U(\CC^{2d}) \cr
|\cr U(\Hd)\times U(\HH)\cr |\cr U(\Hd)\cr|\cr U(\HH) \cr |\cr 1} \qquad
\begin{matrix*}[l] 
\hbox{(a single irreducible)} \ \cite{ABR01} \cr \cr 
	\hbox{(the spaces $H(k-b,b)$)} \ \cite{K73}, \cite{F75}, \cite{R80}, \cite{AS13} \cr \cr 
\hbox{(Theorem \ref{Qkdecompthm})} \ \cite{S74},\cite{ACMM20} \cr \cr 
\hbox{(Theorem \ref{HarmkUirredthm})} \ \cite{BN02},\cite{Ghent14} \cr \cr
\hbox{(Theorem \ref{HarmkIWdecomp})} \ \cite{BN02} \cr\cr
\hbox{(a single homogeneous component).}
\end{matrix*} $$
Here the action of $U(\Hd)\times U(\HH)=\Sp(d)\times\Sp(1)$, 
and similar products, is not faithful, since the real unitary scalar matrix
$-I$ belongs to $U(\Hd)$ and $U(\HH)$, where it has the same action.
One can naturally obtain irreducible decompositions by applying 
the given group to components of the fine scale decomposition.  
For example, we have the following.


\begin{corollary}
\label{leftrightmultirred}
Let $d\ge 2$.
For the action 
given by left and right multiplication by $\HH^*=\Sp(1)$, 
i.e., the group $G=\Sp(1)\times\Sp(1)$, we have
the following orthogonal direct sum of 
homogeneous components
\begin{align}
\label{leftrightmultirredhomo}
\Harm_k(\Hd,\CC)
&=\bigoplus_{0\le j_1,j_2\le{k\over2}} \Bigl\{
		\bigoplus_{ j_1\le a\le k-j_1\atop j_2\le b\le k-j_2 }
	L^{a-j_1} R^{b-j_2} V_{k}^{(j_1,j_2)}(\Hd)\Bigr\} 
= \bigoplus_{0\le j_1,j_2\le{k\over2}} I(W_{k-2j_1,k-2j_2} )^{(k)} \cr
&\cong \bigoplus_{0\le j_1,j_2\le{k\over2}} \dim(V_{k}^{(j_1,j_2)}(\Hd))\cdot
	W_{k-2j_1,k-2j_2},
\end{align}
for the irreducibles
\begin{equation}
\label{leftrightmultirredW}
W_{k-2j_1,k-2j_2} \cong 
\spam_\CC\{L^{a-j_1} R^{b-j_2} f 
	\}_{j_1\le a\le k-j_1\atop j_2\le b\le k-j_2}, \quad f\ne0,
	\quad f\in V_{k}^{(j_1,j_2)}(\Hd).
\end{equation}
\end{corollary}

\begin{proof} The orthogonal direct sums in (\ref{leftrightmultirredhomo})
are immediate, and the $I(W_{k-2j_1,k-2j_2} )^{(k)}$ defined is a sum of
the subspaces in (\ref{leftrightmultirredW}).
It is easily seen from Theorem \ref{HarmkIWdecomp}, 
and its analogue for $L$, 
that these subspaces, i.e.,
$$ \spam_\CC\{L^{\ga} R^{\gb} f 
\}_{0\le\ga\le k-2j_1\atop 0\le\gb\le k-2j_2}, \quad f\ne0,
\quad f\in V_{k}^{(j_1,j_2)}(\Hd), $$
are invariant under left and right multiplication 
by $\HH^*$, that the action is irreducible,
and they are isomorphic $\CC G$-modules.
\end{proof}


Finally, we observe that the central idea underlying our development
is Corollary \ref{Rrightmultequiv}, 
which can be stated  as follows
\begin{itemize}
\item A subspace of $\Hom_k(\Hd,\CC)$ is invariant under the
group action given by right multiplication by $\HH^*$ if and only if
it is invariant under $R$ and $R^*$.
\end{itemize}
This was taken as an Ansatz, then proved indirectly.
A simple constructive proof has eluded us, i.e., finding an explicit formula
$$ f\bigl((z+jw)(\ga+j\gb)\bigr) 
= \sum_{|a|=k} \ga^{a_1}\gb^{a_2}\overline{\ga}^{a_3}\overline{\gb}^{a_4} p_a(R,R^*), 
\qquad \forall f\in\Hom_k(\Hd,\CC), $$
where $p_a(x,y)$ is a polynomial of degree $k$ in the noncommuting
variables $x$ and $y$.
Such a formula would give additional insight and simplify the theory.

\section{Appendix: technical details}
\label{appendixnumber}

Here we give proofs of some technical results used, which are routine, but not insightful.

\proclaim{Lemma \ref{Rstarisadjoint}}
The operators $R^*$ and $L^*$ are the adjoints of $R$ and $L$ with respect to both the inner products (\ref{firstinpro}) and (\ref{secondinpro}) defined on $\Hom_k(\Hd,\CC)$.
\endproclaim

\begin{proof} This is by direct computation. We will just consider $R^*$
(the case for $L^*$ is similar).
Without loss of generality, let $f,g\in\Hom_k(\Hd,\CC)$ be the monomials
$$ f=z^{\ga_1}w^{\ga_2}\overline{z}^{\ga_3} \overline{w}^{\ga_4}, \qquad
g=z^{\gb_1}w^{\gb_2}\overline{z}^{\gb_3} \overline{w}^{\gb_4}. $$
With $(\ga_i)_j$ denoting the $j$-th entry of the multi-index $\ga_i$, etc,
we have
\begin{align*}
Rg &= \sum_{j=1}^d \bigl\{ (\gb_1)_j z^{\gb_1-e_j}w^{\gb_2}\overline{z}^{\gb_3} 
\overline{w}^{\gb_4+e_j} -(\gb_2)_j z^{\gb_1}w^{\gb_2-e_j}\overline{z}^{\gb_3+e_j} 
\overline{w}^{\gb_4}\bigr\}, \cr
R^* f &= \sum_{j=1}^d \bigl\{-(\ga_3)_j z^{\ga_1}w^{\ga_2+e_j}
\overline{z}^{\ga_3-e_j} \overline{w}^{\ga_4} +(\ga_4)_j z^{\ga_1+e_j}
w^{\ga_2}\overline{z}^{\ga_3} \overline{w}^{\ga_4-e_j}\bigr\}.
\end{align*}

For the first inner product, it follows from (\ref{intSformula}) that
all terms of the inner products
\begin{align*} 
\inpro{f,Rg} &= \sum_j \int_{\SS(\CC^{2d})}
\bigl\{ (\gb_1)_j z^{\ga_3+\gb_1-e_j}w^{\ga_4+\gb_2}\overline{z}^{\ga_1+\gb_3} 
\overline{w}^{\ga_2+\gb_4+e_j} \cr
&\qquad\qquad\qquad -(\gb_2)_j z^{\ga_3+\gb_1}w^{\ga_4+\gb_2-e_j}
\overline{z}^{\ga_1+\gb_3+e_j} \overline{w}^{\ga_2+\gb_4} \bigr\}, \cr
\inpro{R^* f,g} 
&= \sum_j\int_{\SS(\CC^{2d})} \bigl\{ -(\ga_3)_j
z^{\ga_3+\gb_1-e_j} w^{\ga_4+\gb_2} \overline{z}^{\ga_1+\gb_3}
\overline{w}^{\ga_2+\gb_4+e_j} \cr
& \qquad\qquad\qquad +(\ga_4)_j z^{\ga_3+\gb_1} w^{\ga_4+\gb_2-e_j}  
\overline{z}^{\ga_1+\gb_3+e_j}\overline{w}^{\ga_2+\gb_4} \bigr\}, 
\end{align*}are zero, except for when
$\ga_3+\gb_1-e_j=\ga_1+\gb_3$, $\ga_4+\gb_2=\ga_2+\gb_4+e_j$,
in which case
\begin{align*}
\inpro{f,Rg} & = {1\over(2d)_k} \sum_j \bigl\{ (\gb_1)_j 
(\gb_1-e_j+\ga_3)!(\gb_2+\ga_4)!
- (\gb_2)_j (\gb_1+\ga_3)!(\gb_2+\ga_4-e_j)! \bigr\} \cr
&= {1\over(2d)_k} (\gb_1+\ga_3-e_j)!(\gb_2+\ga_4-e_j)!
\{ (\gb_1)_j((\gb_2)_j+(\ga_4)_j)-(\gb_2)_j((\gb_1)_j+(\ga_3)_j)\} \cr
&= {1\over(2d)_k} (\ga_1+\gb_3)!(\ga_2+\gb_4)!
\{ (\gb_1)_j(\ga_4)_j-(\gb_2)_j(\ga_3)_j)\}, \cr
\inpro{R^*f,g} 
&= {1\over(2d)_k} 
\bigl\{ -(\ga_3)_j (\gb_1+\ga_3-e_j)!(\gb_2+\ga_4)!
+(\ga_4)_j (\gb_1+\ga_3)!(\gb_2+\ga_4-e_j)! \bigr\} \cr
& = {1\over(2d)_k} (\gb_1+\ga_3-e_j)!(\gb_2+\ga_4-e_j)!
\{ -(\ga_3)_j ((\gb_2)_j+(\ga_4)_j)+(\ga_4)_j((\gb_1)_j+(\ga_3)_j)\}\cr
& = \inpro{f,Rg}, 
\end{align*}
which shows that $R^*$ is indeed the adjoint of $R$.

We now consider the inner product $\dinpro{f,g}$,
for which the monomials are orthogonal.
The inner products $\dinpro{f,Rg}$ and $\dinpro{R^*f,g}$ have nonzero terms
if and only if either
\begin{align*}
& \ga_1=\gb_1-e_j, \quad \ga_2=\gb_2, \quad \ga_3=\gb_3, \quad \ga_4=\gb_4+e_j, \cr
\hbox{or}\qquad
& \ga_1=\gb_1, \quad \ga_2=\gb_2-e_j, \quad \ga_3=\gb_3+e_j, \quad \ga_4=\gb_4.
\end{align*}
For the first case above, these inner products are
\begin{align*}
\dinpro{f,Rg} 
&= (\gb_1)_j
\dinpro{z^{\ga_1}w^{\ga_2}\overline{z}^{\ga_3}\overline{w}^{\ga_4},
z^{\ga_1}w^{\ga_2}\overline{z}^{\ga_3}\overline{w}^{\ga_4} }
= (\gb_1)_j (\gb_1-e_j)!\gb_2!\gb_3!(\gb_4+e_j)!, \cr
\dinpro{R^*f,g} 
&= (\ga_4)_j \dinpro{z^{\gb_1}w^{\gb_2}\overline{z}^{\gb_3}\overline{w}^{\gb_4},
z^{\gb_1}w^{\gb_2}\overline{z}^{\gb_3}\overline{w}^{\gb_4}}
= ((\gb_4)_j+1) \gb_1!\gb_2!\gb_3!\gb_4! = \dinpro{f,Rg}, 
\end{align*}
and, similarly, for the second case
\begin{align*}
\dinpro{f,Rg} 
&=-(\gb_2)_j \dinpro{z^{\ga_1}w^{\ga_2}\overline{z}^{\ga_3}\overline{w}^{\ga_4},
z^{\ga_1}w^{\ga_2}\overline{z}^{\ga_3}\overline{w}^{\ga_4} }
= -(\gb_2)_j \gb_1!(\gb_2-e_j)!(\gb_3+e_j)!\gb_4, \cr
\dinpro{R^*f,g}
&= -(\ga_3)_j \dinpro{
z^{\gb_1}w^{\gb_2}\overline{z}^{\gb_3}\overline{w}^{\gb_4},
z^{\gb_1}w^{\gb_2}\overline{z}^{\gb_3}\overline{w}^{\gb_4}}
= -((\gb_3)_j+1) \gb_1!\gb_2!\gb_3\gb_4! = \dinpro{f,Rg}. 
\end{align*}
Hence $R^*$ is the adjoint of $R$ for the second inner product also.
\end{proof}

\proclaim{Lemma \ref{LapRandLcommute}}
The operators $R,R^*,L$ and $L^*$ commute
with the Laplacian $\Delta$, and so map harmonic functions
to harmonic functions.
\endproclaim

\begin{proof}
We consider $R$, the others being similar. Since
$$ R=\sum_j\Bigl( \overline{w_j}{\partial\over\partial z_j}
- \overline{z_j}{\partial\over\partial w_j}\Bigr), 
\qquad {1\over4}\Delta
= \sum_k \Bigl( {\partial^2\over\partial\overline{z_k}\partial z_k}
+{\partial^2\over\partial\overline{w_k}\partial w_k}\Bigr), $$
we calculate
$$ {1\over4}R\Delta
= \sum_{j,k} \Bigl\{ \overline{w_j}
\Bigl( 
{\partial^3\over\partial z_j\partial\overline{z_k}\partial z_k}
+{\partial^3\over\partial z_j\partial\overline{w_k}\partial w_k} \Bigr)
- \overline{z_j}
\Bigl(
{\partial^3\over\partial w_j\partial\overline{z_k}\partial z_k}
+{\partial^3\over\partial w_j\partial\overline{w_k}\partial w_k}
\Bigr)\Bigr\}, $$
and 
\begin{align*}
{1\over4}\Delta R
& = \sum_{j,k} \Bigl\{ {\partial^2\over\partial\overline{z_k}\partial z_k}
\Bigl( \overline{w_j}{\partial\over\partial z_j}
- \overline{z_j}{\partial\over\partial w_j} \Bigr)
+{\partial^2\over\partial\overline{w_k}\partial w_k}
\Bigl( \overline{w_j}{\partial\over\partial z_j}
- \overline{z_j}{\partial\over\partial w_j} \Bigr) \Bigr\} \cr
& = \sum_{j,k} \Bigl\{ \Bigl(
\overline{w_j}{\partial^3\over\partial\overline{z_k}\partial z_k\partial z_j}
- \overline{z_j} {\partial^3\over\partial\overline{z_k}\partial z_k\partial w_j}
- \gd_{jk} 
{\partial^2\over\partial z_k\partial w_j} \Bigr) \cr
&\qquad + \Bigl( \overline{w_j} {\partial^3\over\partial\overline{w_k}\partial w_k\partial z_j}
+\gd_{jk} {\partial^2\over\partial w_k\partial z_j}
- \overline{z_j}{\partial^3\over\partial\overline{w_k}\partial w_k\partial w_j}
 \Bigr)
\Bigr\},
\end{align*}
which are clearly equal, since the $\gd_{jk}$ terms cancel.

If $f$ is harmonic, i.e., $\gD f=0$, then $\gD (Rf)= R(\gD f)=0$, so $Rf$ is
harmonic.
\end{proof}

\proclaim{Lemma \ref{LRL*R*commute}}
The operators $L$ and $L^*$ commute with $R$ and $R^*$, and we have
\begin{equation}
\label{RR*commute}
R^*R-RR^* = \sum_j \Bigl( z_j{\partial\over\partial z_j}
+w_j{\partial\over\partial w_j}
-\overline{z_j}{\partial\over\partial\overline{z_j}}
-\overline{w_j}{\partial\over\partial\overline{w_j}} \Bigr),
\end{equation}
\begin{equation}
\label{LL*commute}
L^*L - LL^* = \sum_{j} \Bigl( 
z_j {\partial\over\partial z_j} -w_j{\partial\over\partial w_j}
-\overline{z_j}{\partial\over\partial\overline{z_j}} 
+\overline{w_j}{\partial\over\partial\overline{w_j}} \Bigr). 
\end{equation}
\endproclaim

\begin{proof} This is by direct computation. 
We give indicative cases. 
First consider
$$ L = \sum_j \Bigl( w_j{\partial\over\partial z_j}
-\overline{z_j}{\partial\over\partial\overline{w_j}} \Bigr), \qquad
R^* = \sum_k \Bigl( -w_k{\partial\over\partial\overline{z_k}}
+z_k{\partial\over\partial\overline{w_k}} \Bigr). $$
We have
$$ L R^*
= \sum_{j,k} \Bigl\{ w_j \Bigl( -w_k{\partial^2\over\partial z_j\partial\overline{z_k}}
+z_k{\partial^2\over\partial z_j\partial\overline{w_k}} 
+\gd_{jk}{\partial\over\partial\overline{w_k}} \Bigr)
-\overline{z_j} \Bigl( 
-w_k{\partial^2\over\partial\overline{w_j}\partial\overline{z_k}}
+z_k {\partial^2\over\partial\overline{w_j}\partial\overline{w_k}} \Bigr) \Bigr\},$$
$$ R^*L = \sum_{k,j} \Bigl\{ 
 -w_k \Bigl( w_j{\partial^2\over\partial\overline{z_k}\partial z_j}
-\overline{z_j}{\partial^2\over\partial\overline{z_k}\partial\overline{w_j}} 
-\gd_{jk}{\partial\over\partial\overline{w_j}} \Bigr)
+z_k \Bigl( w_j{\partial^2\over\partial\overline{w_k}\partial z_j}
-\overline{z_j}{\partial^2\over\partial\overline{w_k}\partial\overline{w_j}} 
\Bigr) \Bigr\}, $$
which are clearly equal (when applied to polynomials).

Now consider $R^*R-RR^*$. We have
$$ R^*R= \sum_{k,j} \Bigl\{ -w_k
\Bigl( \overline{w_j}{\partial^2\over\partial\overline{z_k}\partial z_j}
-\overline{z_j}{\partial^2\over\partial\overline{z_k}\partial w_j} 
-\gd_{jk}{\partial\over\partial w_j} \Bigr)
+z_k \Bigl( \overline{w_j}{\partial^2\over\partial\overline{w_k}\partial z_j}
+\gd_{jk}{\partial\over\partial z_j}
-\overline{z_j}{\partial^2\over\partial\overline{w_k}\partial w_j}\Bigr) \Bigr\}, $$
$$ RR^* = \sum_{j,k} \Bigl\{ \overline{w_j}
\Bigl( -w_k{\partial^2\over\partial z_j\partial\overline{z_k}}
+z_k{\partial^2\over\partial z_j\partial\overline{w_k}} 
+\gd_{jk}{\partial\over\partial\overline{w_k}} \Bigr)
-\overline{z_j}
\Bigl( 
-w_k{\partial\over\partial w_j\partial\overline{z_k}}
-\gd_{jk}{\partial\over\partial\overline{z_k}}
+z_k{\partial^2\over\partial w_j\partial\overline{w_k}} \Bigr) \Bigr\}. $$
Taking the difference of these gives
\begin{align*}
R^*R-RR^*
&= \sum_{j,k} \gd_{jk}\Bigl( w_k{\partial\over\partial w_j}
+z_k{\partial\over\partial z_j}
- \overline{w_j}{\partial\over\partial\overline{w_k}}
-\overline{z_j}{\partial\over\partial\overline{z_k}} \Bigr) \cr
&= \sum_j \Bigl( z_j{\partial\over\partial z_j}
+w_j{\partial\over\partial w_j}
-\overline{z_j}{\partial\over\partial\overline{z_j}}
-\overline{w_j}{\partial\over\partial\overline{w_j}} \Bigr),
\end{align*}
as supposed.
\end{proof}

\proclaim{Lemma \ref{dimVkablemma}}
For $0\le a,b\le{k\over2}$, we have that
\begin{align}
\label{dimVkabF}
\dim(V_k^{(a,b)}(\Hd)) 
&= F(k,m,M,d)+F(k,m-1,M-1,d) \cr
&\quad-F(k,m-1,M,d)-F(k,m,M-1,d),
\end{align}
where $F$ is given by (\ref{FkmMddef}), and $m=\min\{a,b\}$, $M=\max\{a,b\}$.
In particular,
$$ \dim(V_k^{(a,b)}(\HH^2)) = (m+1)(k-2M+1).  $$
The zonal polynomials in $V_k^{(a,b)}(\Hd)^{U_q}$ is given by
$$ \dim(V_k^{(a,b)}(\Hd)^{U_q}) = 
\begin{cases}
1, & a=b; \cr
0, & a\ne b.
\end{cases}
$$
\endproclaim

\begin{proof}
From Lemma \ref{kerRstarRorthogdecomp}, we have
\begin{align*}
K(k-a,a) &= (K(k-a,a)\cap \ker L^*) \oplus L K(k-a+1,a-1), \cr
H(k-b,b) &= (H(k-b,b)\cap \ker R^*) \oplus R H(k-b+1,b-1),
\end{align*}
which gives the orthogonal direct sum decompositions
\begin{align*}
H_k^{(a,b)}(\Hd)
&= (\ker L^*\cap \ker R^*\cap H_k^{(a,b)}(\Hd))
\oplus L (\ker R^*\cap H_k^{(a-1,b)}(\Hd)) \cr
& \qquad \oplus R (\ker L^*\cap H_k^{(a,b-1)}(\Hd))
\oplus LR H_k^{(a-1,b-1)}(\Hd), \cr
 H_k^{(a,b)}(\Hd) & = (H_k^{(a,b)}(\Hd)\cap \ker L^*) \oplus L H_k^{(a-1,b)}(\Hd) \cr
 & \Implies 
H_k^{(a,b-1)}(\Hd)  = (H_k^{(a,b-1)}(\Hd)\cap \ker L^*) \oplus L H_k^{(a-1,b-1)}(\Hd), \cr
H_k^{(a,b)}(\Hd) &= (H_k^{(a,b)}(\Hd)\cap \ker R^*) \oplus R H_k^{(a,b-1)}(\Hd) \cr
& \Implies
 H_k^{(a-1,b)}(\Hd) = (H_k^{(a-1,b)}(\Hd)\cap \ker R^*) \oplus R H_k^{(a-1,b-1)}(\Hd).
\end{align*}
Since the action of $R,L$ and $RL$ in the above summands is $1$--$1$, we obtain 
\begin{align*}
\dim(V_k^{(a,b)}(\Hd))
&= \dim(H_k^{(a,b)}(\Hd)) -\dim(H_k^{(a-1,b-1)}(\Hd)) \cr
& \qquad -\dim(\ker R^*\cap H_k^{(a-1,b)}(\Hd)) 
-\dim(\ker L^*\cap H_k^{(a,b-1)}(\Hd)) \cr
&= \dim(H_k^{(a,b)}(\Hd)) -\dim(H_k^{(a-1,b-1)}(\Hd)) \cr
& \qquad -\{\dim(H_k^{(a-1,b)}(\Hd))-\dim(H_k^{(a-1,b-1)}(\Hd))\}
\cr
& \qquad -\{\dim(H_k^{(a,b-1)}(\Hd))-\dim(H_k^{(a-1,b-1)}(\Hd)\} \cr
&=\dim(H_k^{(a,b)}(\Hd)) +\dim(H_k^{(a-1,b-1)}(\Hd)) \cr
& \qquad -\dim(H_k^{(a-1,b)}(\Hd)) -\dim(H_k^{(a,b-1)}(\Hd)).
\end{align*}
The formula for $\dim(V_k^{(a,b)}(\Hd))$ above
is symmetric in $a$ and $b$, and so depends only on
$m=\min\{a,b\}$ and $M=\max\{a,b\}$. Since $0\le a,b\le{k\over2}$,
this $m$ and $M$ are the same as those given by (\ref{mMdefn}). 
Suppose, without loss of generality, 
that $a\le b\le{k\over2}$.
Then we have
$$ m:=m_{a,b}^{(k)}=a,\quad M:=M_{a,b}^{(k)}=b, \qquad
m_{a-1,b-1}^{(k)}=m-1,\quad M_{a-1,b-1}^{(k)}=M-1, $$
$$ m_{a-1,b}^{(k)}=m-1,\quad  M_{a-1,b}^{(k)}=M, \qquad
m_{a,b-1}^{(k)}=m-\gd_{ab},\ M_{a,b-1}^{(k)}=M-1+\gd_{ab}, $$
and hence, for $a\ne b$, by Lemma \ref{Habdecomplemma} we have
\begin{align*}
\dim(V_k^{(a,b)}(\Hd)) 
&= F(k,m,M,d)+F(k,m-1,M-1,d) \cr
&\quad-F(k,m-1,M,d)-F(k,m,M-1,d).
\end{align*}
This is formula (\ref{dimVkabF}), which also holds for $a=b$, 
i.e., $m=M=a=b$, 
in which case 
\begin{align*}
\dim(V_k^{(a,a)}(\Hd)) 
&= F(k,m,m,d)+F(k,m-1,m-1,d) \cr
&\quad-F(k,m-1,m,d)-F(k,m-1,m,d).
\end{align*}

For $d=2$, the sum $F(k,m,M,d)$ simplifies to
$$ F(k,m,M,2)=\sum_{j=0}^m (j+1)(M-m+j+1)(k-M+m-2j+1). $$
Thus, from (\ref{dimVkabF}), we calculate
\begin{align*}
\dim(V_k^{(a,b)}(\HH^2))
&= (m+1)^2(k-2m+1)-(m+1)M(k-M-m+2) 
+\sum_{j=0}^{m-1} \bigl\{\quad\bigr\},
\end{align*}
where
\begin{align*}
\bigl\{\quad\bigr\}
&= 2(j+1)(M-m+j+1)(k-M+m-2j+1) \cr
&\qquad -(j+1)(M-m+1+j+1)(k-M+m-1-2j+1) \cr
&\qquad -(j+1)(M-1-m+j+1)(k-M+1+m-2j+1)\cr
&= 2+2j, 
\end{align*}
and so we obtain
\begin{align*}
\dim(V_k^{(a,b)}(\HH^2))
&= (m+1)^2(k-2m+1)-(m+1)M(k-M-m+2)+2\Bigl\{ m + {m(m-1)\over2}\Bigr\} \cr
&= (m+1)(k-2M+1).
\end{align*}

For any subgroup $G$ of $U(\Hd)$, following the above argument
(for $G=1$), we obtain
\begin{align*}
\dim(V_k^{(a,b)}(\Hd)^G)
&=\dim(H_k^{(a,b)}(\Hd)^G) +\dim(H_k^{(a-1,b-1)}(\Hd)^G) \cr
& \qquad -\dim(H_k^{(a-1,b)}(\Hd)^G) -\dim(H_k^{(a,b-1)}(\Hd)^G).
\end{align*}
In particular, for $G=U_q$, we have 
$$ \dim((H_k^{(a,b)}(\Hd)^{U_q})=m_{a,b}+1, \qquad
m_{a,b}=\min\{a,k-a,b,k-b\}, $$
and so 
$$ \dim(V_k^{(a,b)}(\Hd)^{U_q})
= m_{a,b}+m_{a-1,b-1}-m_{a-1,b}-m_{a,b-1}. $$
Calculating this gives
\begin{align*}
m_{a,b}+m_{a-1,b-1}-m_{a-1,b}-m_{a,b-1} &= a +(a-1)-(a-1)-(a-1)=1, \qquad a=b, \cr
m_{a,b}+m_{a-1,b-1}-m_{a-1,b}-m_{a,b-1} &= a +(a-1)-(a-1)-a=0, \qquad a<b,
\end{align*}
as supposed.
\end{proof}

We now show that $R$ and $L$ map the space of zonal polynomials for $q'=e_1$
to itself. This was assumed to be true for a general $q'$, 
but calculations show otherwise (for $L$). 

\begin{lemma}
\label{Rmapszonaltozonal}
 $R$ and $R^*$ map zonal polynomials to zonal polynomials, i.e.,
\begin{align}
R ( [a,b,c,d,r]) &= a[a-1,b,c,d+1,r] - b[a,b-1,c+1,d,r],
\label{Rzonal} \\
R^* ( [a,b,c,d,r]) &= -c[a,b+1,c-1,d,r] +d[a+1,b,c,d-1,r], 
\label{Rstarzonal}
\end{align}
and $L$ and $L^*$ map zonal polynomials for $q'=z'+jw'$ as follows
\begin{align}
\label{Lzonalzp}
\hskip-0.0truecm L ( [a,b,c,d,r]) &=  a [a-1,b,c,d,r] [0,1,0,0,0]_{\overline{z'}} 
 -d [a,b,c,d-1,r] [0,0,1,0,0]_{\overline{z'}} \nonumber \\
&\hskip-0.7truecm  -b[a,b-1,c,d,r] [1,0,0,0,0]_{j\overline{w'}} 
+c[a,b,c-1,d,r][0,0,0,1,0]_{j\overline{w'}}, \\
\label{Lstarzonalzp}
\hskip-0.0truecm L^* ( [a,b,c,d,r]) &=  b [a,b-1,c,d,r] [1,0,0,0,0]_{\overline{z'}}
 -c [a,b,c-1,d,r] [0,0,0,1,0]_{\overline{z'}} \nonumber \\
&\hskip-0.7truecm  -a[a-1,b,c,d,r] [0,1,0,0,0]_{j\overline{w'}} 
+d[a,b,c,d-1,r][0,0,1,0,0]_{j\overline{w'}}.
\end{align}
For $q'=z'\in\RR^n$,  
$L$ and $L^*$ map zonal polynomials
to zonal polynomials, i.e.,
\begin{align}
L ( [a,b,c,d,r]) &=  a[a-1,b+1,c,d,r] -d[a,b,c+1,d-1,r],
\label{Lzonal} \\
L^* ( [a,b,c,d,r]) &=  b[a+1,b-1,c,d,r] -c[a,b,c-1,d+1,r].
\label{Lstarzonal}
\end{align}
\end{lemma} 

\begin{proof} 
Let $T$ be any of $R,R^*,L$ or $L^*$. 
Since $T([0,0,0,0,1])=0$, by applying the 
product rule (\ref{RLproductrule}) to
$$ [a,b,c,d,r]=[1,0,0,0,0]^a [0,1,0,0,0]^b 
[0,0,1,0,0]^c [0,0,0,1,0]^d [0,0,0,0,1]^r, $$
we obtain
\begin{align}
\label{Tzonalform}
& \hskip-1cm T([a,b,c,d,r]) \cr
& = a [a-1,b,c,d,r] T([1,0,0,0,0])
+ b [a,b-1,c,d,r] T([0,1,0,0,0]) \cr
&\quad
+ c [a,b,c-1,d,r] T([0,0,1,0,0])
+ d [a,b,c,d-1,r] T([0,0,0,1,0]).
\end{align}
Thus it suffices to calculate $T$ applied
to $[1,0,0,0,0],\ldots,[0,0,0,1,0]$ directly. 
We give representative calculations.
By Lemma \ref{zonalformqq'}, for $q'=z'+jw',q=z+jw\in\HH^n$, we have
\begin{align*}
[1,0,0,0,0] &=\overline{z_1'}z_1+\cdots+\overline{z_n'}z_n
+\overline{w_1'}w_1+\cdots+\overline{w_n'}w_n, \cr
[0,1,0,0,0] &=z_1'w_1+\cdots+z_n'w_n
-w_1'z_1-\cdots-w_n'z_n, \cr
[0,0,1,0,0] &=z_1'\overline{z_1}+\cdots+z_n'\overline{z_n}
+w_1'\overline{w_1}+\cdots+w_n'\overline{w_n}, \cr
[0,0,0,1,0] &=\overline{z_1'}\overline{w_1}+\cdots+\overline{z_n'}\overline{w_n}
-\overline{w_1'}\overline{z_1}-\cdots-\overline{w_n'}\overline{z_n}.
\end{align*}
Applying the formulas of (\ref{Rformulagen}) and (\ref{Lformulagen}), gives
\begin{align*}
R([1,0,0,0,0]) 
&= \overline{w_1} \overline{z_1'} +\cdots 
+\overline{w_n} \overline{z_n'}
-\overline{z_1}\overline{w_1'} - \cdots
-\overline{z_n}\overline{w_n'}
= [0,0,0,1,0], \cr
R([0,1,0,0,0]) 
&= -\overline{w_1} w_1' +\cdots 
-\overline{w_n} w_n'
-\overline{z_1}z_1' - \cdots
-\overline{z_n}z_n'
= -[0,0,1,0,0], \cr
R([0,0,1,0,0]) &= R([0,0,0,1,0]) =0, \cr
L([1,0,0,0,0]) &=w_1\overline{z_1'}+\cdots+w_n\overline{z_n'} 
                =[0,1,0,0,0]_{\overline{z'}}, \cr
L([0,1,0,0,0]) &=-w_1{w_1'}-\cdots-w_n{w_n'} 
                 =-[1,0,0,0,0]_{j\overline{w'}}, \cr
L([0,0,1,0,0]) &= -\overline{z_1}{w_1'}-\cdots-\overline{z_n}{w_n'}
                 =[0,0,0,1,0]_{j\overline{w'}}, \cr
L([0,0,0,1,0]) &= -\overline{z_1}\overline{z_1'}-\cdots-\overline{z_n}\overline{z_n'} 
                 = -[0,0,1,0,0]_{\overline{z'}}.
\end{align*}
Similarly, we have
\begin{align*}
R^*([1,0,0,0,0])&=R^*([0,1,0,0,0]) =0, \cr
R^*([0,0,1,0,0])&= -[0,1,0,0,0]_{\overline{z'}}, \qquad
R^*([0,0,0,1,0]) = [1,0,0,0,0]_{j\overline{w'}}, \cr
L^*([1,0,0,0,0])&=-[0,1,0,0,0]_{j\overline{w'}}, \qquad
L^*([0,1,0,0,0]) = [1,0,0,0,0]_{\overline{z'}}, \cr
L^*([0,0,1,0,0])&= -[0,0,0,1,0]_{\overline{z'}}, \qquad
L^*([0,0,0,1,0]) = [0,0,1,0,0]_{j\overline{w'}},
\end{align*}
Substituting the above into (\ref{Tzonalform}) then gives the result.
\end{proof}

The action of $\Delta$ 
on zonal polynomials is similar
to the real and complex cases.

\begin{lemma} 
\label{LaplacianZonal}
The Laplacian maps zonal polynomials $\HH^n\to\CC$ to zonal polynomials,
i.e.,
\begin{align} 
\label{Laplacianofzonal}
{1\over4}\Delta ([a,b,c,d,r])
&= ac[a-1,b,c-1,d,r]+bd[a,b-1,c,d-1,r] \cr
&\qquad + r( k+2n-1 -r)[a,b,c,d,r-1].
\end{align}
\end{lemma}

\begin{proof} Given 
the correspondence (\ref{e1toq'subs}) between zonal 
polynomials with poles $q'$ and $e_1$, and the fact $\Delta$ commutes
with unitary maps, it suffices to prove 
(\ref{Laplacianofzonal}) for $q'=e_1$, i.e.,
$$  [a,b,c,d,r]= f g^r, \qquad f:=z_1^a w_1^b \overline{z_1}^c \overline{w_1}^d,
\quad
g:=\sum_j (z_j\overline{z_j}+w_j\overline{w_j}). $$
Differentiation gives
$$ {\partial^2\over\partial\overline{z_j} \partial z_j}(fg^r)
= {\partial^2 f\over\partial\overline{z_j}\partial z_j} g^r
 +{\partial f\over\partial z_j} rg^{r-1} z_j 
 +{\partial f\over\partial\overline{z_j}} rg^{r-1}\overline{z_j}
 +f r(r-1)g^{r-2} z_j\overline{z_j}
 +f rg^{r-1}. $$
For $j\ne1$, all the terms above are zero except for the last two. 
Summing these over all $j=1,\ldots,n$, 
together with the analogous terms for $w_j$, 
we get
$$ r(r-1) f g^{r-2}g+2nr f g^{r-1}
= r(r-1+2n) f g^{r-1} =r(r-1+2n)[a,b,c,d,r-1], $$
and the first three terms (for $j=1$) give
$$ ( ac[a-1,b,c-1,d,r]+bd[a,b-1,c,d-1,r]) 
+ r( (a+b)+ (c+d))[a,b,c,d,r-1]. $$
Since $a+b+c+d=k-2r$, adding these gives the result.
\end{proof}

This result is given in \cite{BN02} (Proposition 4.2), without proof.
It implies that:
\begin{itemize}
\item The Laplacian maps $\Hom_k^{(a,b)}(\Hd)$ to
$\Hom_{k-2}^{(a-1,b-1)}(\Hd)$.
\item The Laplacian maps $E_{w,w'}^{(k)}$ of (\ref{Ewwpkdefn}) 
to $E_{w,w'}^{(k-2)}$.
\end{itemize}

Combining Lemma \ref{RowMovementsLemma} and its counter part for $L$ gives the following.

\begin{lemma}
\label{RowandColMovementsLemma}
(Square array) For $0\le j\le{k\over2}$, $j\le a,b\le k-j$, we have
\begin{align*}
L^{a-j} R^{b-j} (\ker L^*\cap\ker R^*\cap H_k^{(j,j)})
& = (L^*)^{k-a-j} (R^*)^{k-b-j} (\ker L\cap \ker R\cap H_k^{(k-j,k-j)}) \cr
& = L^{a-j}(R^*)^{k-b-j}   (\ker L^*\cap\ker R \cap H_k^{(j,k-j)}) \cr
& = (L^*)^{k-a-j} R^{b-j} (\ker L\cap \ker R^*\cap H_k^{(k-j,j)}),
\end{align*}
with
$$ \dim(L^{a-j} R^{b-j} (\ker L^*\cap\ker R^*\cap H_k^{(j,j)})) 
= \dim (\ker L^*\cap\ker R^*\cap H_k^{(j,j)}). $$
\end{lemma}


\bibliographystyle{alpha}
\bibliography{references}
\nocite{*}

\vfil\eject


\end{document}